\documentclass[11pt,oneside,reqno]{article}
\pdfoutput=1
\usepackage[utf8]{inputenc}
\usepackage[a4paper,vmargin=3cm]{geometry}
\usepackage{amsmath}
\usepackage{amssymb}
\usepackage{amsthm}
\usepackage{mathtools}
\usepackage{color}
\usepackage{stackrel}
\usepackage[activate={true,nocompatibility},final,tracking=true,
  kerning=true,spacing=true,factor=1100,stretch=10,shrink=10]{microtype}
\usepackage{hyperref}

\definecolor{Darkgreen}{rgb}{0,0.4,0}
\hypersetup{colorlinks,linkcolor=Darkgreen,citecolor=Darkgreen,breaklinks=true,unicode}

\microtypecontext{spacing=nonfrench}

\usepackage{algorithm} 
\usepackage{algpseudocode}
\usepackage{float}

\newtheorem{theorem}{Theorem}
\newtheorem{lemma}[theorem]{Lemma}
\newtheorem{proposition}[theorem]{Proposition}
\newtheorem{corollary}[theorem]{Corollary}

\theoremstyle{definition}
\newtheorem{remark}[theorem]{Remark}
\newtheorem*{assumptions}{Assumptions}
\newtheorem*{acknowledgements}{Acknowledgements}

\newcommand{\finitegraph}{{\mathcal{G}_n}}
\newcommand{\treegraph}{{\mathbb T_d}}
\newcommand{\treegraphplus}{{\mathbb T^+_d}}

\newcommand{\rot}{{\textup{o}}}
\newcommand{\bbone}{\boldsymbol{1}}
\newcommand{\tx}{\mathtt{tx}}
\newcommand{\Var}{\textup{Var}}

\numberwithin{equation}{section}   
\numberwithin{theorem}{section}

\begin{document}

\title{\bf \large \uppercase{Level-set percolation of the {G}aussian free
    field on regular graphs {II}: finite expanders}}

\date{Preliminary draft}

\author{Angelo Ab\"acherli
  \thanks{Departement Mathematik, ETH Z\"urich,
    R\"amistrasse 101, 8092 Z\"urich, Switzerland}
  \and Ji\v{r}\'i \v{C}ern\'y
  \thanks{Departement Mathematik und Informatik,
    University of Basel, Spiegelgasse 1, 4051 Basel, Switzerland}}

\maketitle

\begin{abstract}
  We consider the zero-average Gaussian free field on a certain class of
  finite $d$-regular graphs for fixed $d\geq 3$. This class includes $d$-regular
  expanders of large girth and typical realisations of random $d$-regular
  graphs. We show that the level set of the zero-average Gaussian free
  field above level $h$ exhibits a phase transition at level $h_\star$,
  which agrees with the critical value for level-set percolation of the
  Gaussian free field on the \emph{infinite $d$-regular tree}. More
  precisely, we show that, with probability tending to one as the size of
  the \emph{finite} graphs tends to infinity, the level set above level
  $h$ does not contain any connected component of larger than logarithmic
  size whenever $h>h_\star$, and on the contrary, whenever $h<h_\star$, a
  linear fraction of the vertices is contained in connected components of
  the level set above level $h$ having a size of at least a small
  fractional power of the total size of the graph. It remains open
  whether in the supercritical phase $h<h_\star$, as the size of the
  graphs tends to infinity, one observes the emergence of a (potentially
    unique) \emph{giant} connected component of the level set above level
  $h$. The proofs in this article make use of results from the
  accompanying paper \cite{AC1}.
\end{abstract}

\setcounter{page}{1}



\section{Introduction}

In this article we study level-set percolation of the zero-average
Gaussian free field on a class of large $d$-regular graphs with $d\geq 3$.
This class contains $d$-regular expanders of large girth and typical
realisations of random $d$-regular graphs. Through suitable local
approximations of the zero-average Gaussian free field by the Gaussian
free field on the infinite $d$-regular tree we are able to establish  a
phase transition for level-set percolation of the zero-average Gaussian
free field which occurs at the critical value for level-set percolation
in the infinite model, that is, on the $d$-regular tree.

Level-set percolation and the local picture of the zero-average Gaussian
free field have been previously studied by the first author in \cite{36}
for the situation where the underlying sequence of finite graphs is given
by the discrete tori of growing side length in dimension $d\geq 3$.  The
motivation for investigating the zero-average Gaussian free field on the
different class of finite graphs considered here (see
  \eqref{0}--\eqref{2} below) stems from the insight that analysing
probabilistic models on these types of finite graphs has led to often
very explicit and strong results over the years. Examples include the
emergence of a giant connected component for Bernoulli bond percolation
(see e.g.~\cite{38} and recently \cite{46}), cutoff phenomena for random
walks (see e.g.~\cite{39}) and the appearance of a giant connected
component in the vacant set of simple random walk (see e.g.~\cite{21}).
Actually, we will borrow the assumptions  \eqref{0}--\eqref{2} on the
finite graphs from \cite{21}.

From a more general perspective, level-set percolation of the Gaussian
free field is a significant representative of a percolation model with
long-range dependencies and it has attracted attention for a long time,
dating back to \cite{31}, \cite{28} and \cite{12}. More recent
developments can be found for instance in \cite{7}, \cite{13}, \cite{22},
\cite{40} and \cite{DPR2}. For the particular case of the Gaussian free
field on regular trees we also refer to \cite{35}, \cite{SzniECP} and
\cite{AC1}; for more general transient trees to \cite{37}.

We now describe our results more precisely. We let $d\geq 3$ and assume
that $(\finitegraph)_{n\geq 1}$ is a sequence of graphs satisfying the
following conditions.
\begin{assumptions}
  There exist some $\alpha, \beta >0$ and an increasing sequence of
  positive integers $(N_n)_{n\geq 1}$ with
  $N_n \xrightarrow{n\to \infty} \infty$ such that for all $n\geq 1$
  \begin{flalign}
    \quad \ \,  \bullet& \ \, \finitegraph
    \text{ is $d$-regular, connected and has
      $N_n$ vertices} &
    \label{0}   \\
    \quad \ \, \bullet& \ \, \text{for all $x \in \finitegraph$ there is at most
      one cycle in the ball of radius $\lfloor \alpha \log_{d-1}(N_n) \rfloor$} &
    \nonumber \\
    \quad \ \,  &\ \,  \text{around $x$} &
    \label{1}  \\
    \quad \ \, \bullet& \ \, \text{the spectral gap  of $\finitegraph$, denoted by
      $\lambda_\finitegraph$, satisfies  $\lambda_\finitegraph \geq
      \beta$.}  &
    \label{2}
  \end{flalign}
\end{assumptions}
\noindent Here by spectral gap we mean  the smallest non-zero eigenvalue
of $I-P$, where $I$ is the identity matrix and $P$ is the transition
matrix of the simple random walk on the graph (see also \cite{44},
  Definition 2.1.3 and beneath it). For an explanation of why these
assumptions are satisfied by $d$-regular expanders of large girth and by
typical realisations of random $d$-regular graphs we refer to \cite{21},
Section 2.2 and Remark 1.4.

On $\finitegraph$ we consider the zero-average Gaussian free field (see
  Section \ref{subsection1.2} for more details about it) with law
$\mathbb P^\finitegraph$ on $\mathbb R^\finitegraph$ and canonical
coordinate process $(\Psi_\finitegraph(x))_{x \in \finitegraph}$ so that,
\begin{equation}
  \label{0.1}
  \parbox{12cm}{
    under $\mathbb P^\finitegraph$, $(\Psi_\finitegraph(x))_{x \in \finitegraph}$
    is a  centred Gaussian field on $\finitegraph$ with covariance $\mathbb
    E^\finitegraph [\Psi_\finitegraph(x) \Psi_\finitegraph(y)] =
    G_\finitegraph(x,y)$ for all $x,y \in \finitegraph$, where
    $G_\finitegraph(\cdot,\cdot)$ is the zero-average  Green function on
    $\finitegraph$ (see \eqref{1.5}).
  }
\end{equation}
The zero-average Gaussian free field is a natural version of the Gaussian
free field for finite graphs. However, due to the zero-average property
(see below \eqref{1.7}), it comes with some peculiarities like the
\emph{lack} of an FKG-inequality and of the domain Markov property.

Our main interest lies in analysing the size (i.e.~the number of
  contained vertices) of the  connected components of the level sets of
$\Psi_\finitegraph$, i.e.~of
\begin{equation}
  \label{0001}
  E_{\Psi_\finitegraph}^{\geq h}
  \coloneqq \{x\in \finitegraph \, | \,  \Psi_\finitegraph(x)\geq h\}
  \text{ for } h\in \mathbb R.
\end{equation}

In order to do so, it will be helpful to locally describe
$\Psi_\finitegraph$  via the Gaussian free field on  the infinite $d$-regular
tree $\treegraph$ with root denoted by $\rot$, that is, the centred
Gaussian field on $\treegraph$ with law $\mathbb P^\treegraph$ on
$\mathbb R^\treegraph$ and canonical coordinate process
$(\varphi_\treegraph(x))_{x \in \treegraph}$ so that,
\begin{equation}
  \label{0.2}
  \parbox{12cm}{
    under $\mathbb P^\treegraph$, $(\varphi_\treegraph(x))_{x \in \treegraph}$ is a
    centred Gaussian field on $\treegraph$ with covariance $\mathbb E^\treegraph
    [\varphi_\treegraph(x) \varphi_\treegraph(y)] = g_\treegraph(x,y)$ for all $x,y
    \in \treegraph$, where $g_\treegraph(\cdot,\cdot)$ is the Green function of
    simple random walk on $\treegraph$ (see \eqref{1.1}).
  }
\end{equation}
The Gaussian free field on $\treegraph$ has first been studied in
\cite{35}. Recently, more refined results have been obtained by the
authors in the accompanying paper \cite{AC1}. These results lay the
groundwork for the present article and they will be central in our
analysis of the zero-average Gaussian free field on the graphs
$(\finitegraph)_{n\geq 1}$. For now, we only recall the critical value of
level-set percolation of $\varphi_\treegraph$, that is,
\begin{equation}
  \label{0.5}
  h_\star \coloneqq \inf \Big \{ h\in \mathbb R \, \Big| \,   \mathbb
    P^\treegraph \big[ \big| \mathcal C_{\rot}^{\treegraph,h} \big| = \infty
      \big] = 0 \Big \}, \vspace{-0.2cm}
\end{equation}
where $\mathcal C_{\rot}^{\treegraph,h}$ is the connected component
of the level set
$E_{\varphi_\treegraph}^{\geq h}
\coloneqq \{x\in \treegraph \, | \, \varphi_\treegraph(x)\geq h\}$
of $\varphi_\treegraph$ above level $h$ containing the root
$\rot \in \treegraph$. There is a crucial spectral characterisation
of $h_\star$ derived in \cite{35}, which leads to the proof of
$0< h_\star < \infty$ on $\treegraph$ for $d\geq 3$ (see \cite{35},
  Proposition 3.3 and Corollary 4.5). Actually, in the accompanying paper
\cite{AC1} we make heavy use of this characterisation to obtain new
results about $\varphi_\treegraph$ on $\treegraph$.

Our main results concerning the size of the   connected components of the
level sets of $\Psi_\finitegraph$ on the finite graphs
$(\finitegraph)_{n\geq 1}$ satisfying \eqref{0}--\eqref{2} are the
following: we show in essence that (see Section \ref{subcriticalsection},
  Theorem~\ref{microscopiccomponents}, for the precise statement)
\begin{equation}
  \label{0.6}
  \parbox{12cm}{
    in the subcritical phase $h>h_\star$,  with high probability for
    large $n$, the level set $E_{\Psi_\finitegraph}^{\geq h}$ of
    $\Psi_\finitegraph$ only contains microscopic connected components
    (i.e.~containing at most a logarithmic number of vertices of
      $\finitegraph$);
  }
\end{equation}
and furthermore that (see Section \ref{supercriticalsection}, Theorem
  \ref{manymesoscopiccomponents}, for the precise statement)
\begin{equation}
  \label{0.7}
  \parbox{12cm}{
    in the supercritical phase $h<h_\star$,  with high probability for large $n$, a
    linear fraction of the vertices of $\finitegraph$ is contained in at least
    mesoscopic connected components of the level set $E_{\Psi_\finitegraph}^{\geq
      h}$ of $\Psi_\finitegraph$ (i.e.~containing a fractional power of the number of
      vertices of $\finitegraph$).
  }
\end{equation}

Although giving a strong hint to, the result \eqref{0.7} leaves open
whether in the supercritical phase $h<h_\star$, with high probability for
large $n$, there actually is a macroscopic (\emph{giant}) connected
component in the level set above level $h$, i.e.~containing a  number of
vertices comparable to $\finitegraph$. Furthermore, in the affirmative,
one could ask if this giant component is unique, that is, if the
second-largest connected component of the level set above level
$h<h_\star$ only contains a negligible number of vertices compared to
$\finitegraph$ (see also Remark \ref{openquestions}).

As a comparison, the emergence of a unique giant connected component in
the supercritical phase has been shown for Bernoulli bond percolation on
$d$-regular expanders of large girth in \cite{38}  (see also \cite{46})
and for vacant-set percolation of simple random walk on exactly the same
graphs $(\mathcal G_n)_{n\geq 1}$ like here in \cite{21}. In the latter,
this result is achieved by relating the model to vacant-set percolation
of random interlacements on $\treegraph$. Subsequently, more refined
results have been obtained about the vacant set of simple random walk on
random regular graphs in \cite{CoopFri} and \cite{42}.

In the models mentioned above, the assertion of existence and uniqueness
of a giant component in the supercritical phase is achieved by a
`sprinkling argument' starting from a statement like \eqref{0.7}. In our
situation, it would correspond to showing that distinct mesoscopic
connected components of $E_{\Psi_\finitegraph}^{\geq h}$ for a
supercritical level $h<h_\star$ are going to be connected at a slightly
smaller level $h'<h$ with high probability, thus forming large clusters.
As \cite{21} shows, it can be very involved to carry out sprinkling
arguments in the non-i.i.d.~setting. At present we have not been able to
do it in our context, one of the main restrictions stemming from the
defining zero-average property of the fields we are considering (see
  below \eqref{1.7}). We point out that sprinkling techniques have been
already  applied  in the discussion of level-set percolation of the
Gaussian free field in \cite{DreRod} to construct an infinite connected
component with the underlying graph being $\mathbb Z^d$ for high
dimension $d$.

Let us now comment on the proofs of Theorem~\ref{microscopiccomponents}
and Theorem~\ref{manymesoscopiccomponents} (corresponding to \eqref{0.6}
  and \eqref{0.7}). In both cases, the general philosophy is to locally
approximate $\Psi_\finitegraph$ on the finite graphs by
$\varphi_\treegraph$ on the $d$-regular tree and by that reduce the
analysis to the infinite model, which is easier to understand. A similar
strategy has been successfully carried out in \cite{38} and \cite{21}
where the connected components in question are locally approximated by
Galton-Watson trees. In our setting the situation is considerably more
complicated since neither the connected components of the level sets of
$\Psi_\finitegraph$ nor the connected components of the level sets of
$\varphi_\treegraph$ (used in the approximation) are locally
Galton-Watson trees, even if the connected components of
$E_{\varphi_\treegraph}^{\geq h}$ share some global properties with them,
as shown in \cite{AC1}. The exact way how the local approximation by
$\varphi_\treegraph$ is performed differs considerably between the
subcritical and supercritical phase.

In the supercritical phase $h<h_\star$, we use an approximation of
$\Psi_\finitegraph$ by $\varphi_\treegraph$ via local charts around
vertices of $\finitegraph$ with a tree-like neighbourhood
(Theorem~\ref{extendedtheorem2.2}). Then the proof of
Theorem~\ref{manymesoscopiccomponents} (corresponding to \eqref{0.7}) is,
roughly said,  a second moment computation based on this local
approximation and involving a good control of the supercritical level
sets of $\varphi_\treegraph$, obtained in the accompanying paper \cite{AC1}.

More precisely, to show \eqref{0.7} we prove that the number of vertices
contained in mesoscopic connected components of the level set
$E_{\Psi_\finitegraph}^{\geq h}$ concentrates around its expectation,
which we show to grow linearly in the total number of vertices. The
concentration follows by a variance computation and a second moment
inequality. Actually, when estimating the expectation and variance, it is
enough to consider only vertices with a tree-like neighbourhood since the
assumption \eqref{1} (together with \eqref{0}) guarantees that  the
number of vertices having a tree-like neighbourhood is comparable to the
total number of vertices in $\finitegraph$ (Remark \ref{842}). Thanks to
the approximation of $\Psi_\finitegraph$ by $\varphi_\treegraph$ around
such vertices (Theorem \ref{extendedtheorem2.2} mentioned above), we are
able to transfer the computations to the regular tree. The linear lower
bound on the expectation (\eqref{224} in Lemma
  \ref{expectationcomputation}) now follows rather direct from this
approximation and from \cite{AC1}, Theorem 4.3, showing that connected
components of the level sets of $\varphi_\treegraph$ are mesoscopic with
positive probability in the supercritical phase. The control of the
variance follows along similar lines (Lemma~\ref{variancecomputation}).
It requires the approximation of $\Psi_\treegraph$ by $\varphi_\treegraph$
on neighbourhoods of vertices with a tree-like \emph{and} disjoint
neighbourhood. This is provided by Theorem \ref{extendedtheorem2.2} as
well. Once we have reduced the computations to quantities for
$\varphi_\treegraph$ on $\treegraph$, we can apply a decoupling
inequality (\cite{13}, Corollary 1.3) and deduce the  bound on the
variance again from results on $\varphi_\treegraph$ developed in the
accompanying paper \cite{AC1}.

For the subcritical phase $h>h_\star$ (Theorem
  \ref{microscopiccomponents} corresponding to \eqref{0.6}) the local
approximation of $\Psi_\finitegraph$ by $\varphi_\treegraph$ around
vertices with tree-like neighbourhood is not good enough. On the one
hand, the connected components of $E_{\Psi_\finitegraph}^{\geq h}$ may
have a diameter that is larger than the diameter of those neighbourhoods
(at least if $h$ is close to $h_\star$). On the other hand, one expects
that the connected components are typically `thin'. These two points of
`thinness'  and of `escaping the local charts' suggest that the
approximation of $\Psi_\finitegraph$ by $\varphi_\treegraph$ should
rather be carried out \emph{along} the connected components of
$E_{\Psi_\finitegraph}^{\geq h}$. We achieve this by employing an
exploration process uncovering the connected component  of the level set
containing a given vertex (Algorithm~\ref{algorithm} in
  Section~\ref{subcriticalsection}). Roughly said, by exploring
$\Psi_\finitegraph$ vertex by vertex we are able to couple it  vertex by
vertex to a number of independent copies of $\varphi_\treegraph$ on
$\treegraph$, hence bringing back the problem to the tree. Results from
\cite{AC1}  on $\varphi_\treegraph$ in the subcritical phase then
conclude the proof.

More precisely, the exploration process aggregates the vertices found in
the connected component of $E_{\Psi_\finitegraph}^{\geq h}$ containing a
fixed $x \in \finitegraph$ into a union of disjoint subtrees of
$\finitegraph$. The decomposition into a union of disjoint subtrees is
determined \emph{during} the exploration and it  is dictated by the
\emph{geometric} properties of the graph $\finitegraph$ and of the
evolving set of explored vertices. These geometric conditions guarantee
that for each of the disjoint subtrees we can approximate the
zero-average Gaussian free field $\Psi_\finitegraph$ on the subtree by an
independent copy of the Gaussian free field $\varphi_\treegraph$ on
$\treegraph$ (Lemma \ref{806}). In order to do so, it is crucial to have
a good understanding of the conditional distribution of the zero-average
Gaussian free field (Lemma \ref{conditionaldistribution} and Proposition
  \ref{535}). As a consequence, the size of each disjoint subtree of
$\finitegraph$ constructed by the exploration process is dominated by the
size of the connected component containing the root
$\rot\in \treegraph$ of the level set of $\varphi_\treegraph$
above a slightly lower level $h-\varepsilon$ (Corollary \ref{841}). The
last two ingredients for the proof of \eqref{0.6} are now a control on
the number of disjoint subtrees (Lemma \ref{532}, already proven in
  \cite{21}) and a control on the exponential moments of the size of the
connected component of the level set of $\varphi_\treegraph$ containing
the root $\rot\in \treegraph$ in the subcritical phase (see
  \cite{AC1}, Theorem 5.1).

Incidentally, let us point out that exploration processes are frequently
used in the Bernoulli percolation literature and actually, a variant of
such an algorithm was applied in \cite{21} to deal with the vacant set of
simple random walk in the subcritical phase. However, in our setting we
cannot follow the `standard' procedure. Usually, to show statements like
\eqref{0.6}, a good control on the termination time of the exploration
process is necessary, i.e.~on the time by when the connected component is
completely uncovered. This is typically done by comparing the number of
yet unexplored vertices to a random walk of negative drift. In our case
this is not possible, essentially again   because locally the connected
components of $E_{\Psi_\finitegraph}^{\geq h}$ are not approximated by
Galton-Watson trees (as mentioned earlier).

The structure of the article is as follows. In Section \ref{section1} we
collect the notation and some results on the Gaussian free fields on both
the finite graphs and the infinite tree. In particular, in Section
\ref{sectiontree} we recall results on $\varphi_\treegraph$ from
\cite{35} and \cite{AC1}. Then in Section \ref{localpicturesection} we
investigate the local picture of the zero-average Gaussian free field on
$\finitegraph$ and its connection to the Gaussian free field on
$\treegraph$. The content of these first two sections will be
subsequently used to show Theorem~\ref{microscopiccomponents}
(corresponding to \eqref{0.6}) and Theorem~\ref{manymesoscopiccomponents}
(corresponding to \eqref{0.7}). More precisely, in Section
\ref{subcriticalsection} we deal with the subcritical phase, ultimately
proving the non-existence of  connected components of
$E_{\Psi_\finitegraph}^{\geq h}$ for $h>h_\star$ of larger than
logarithmic size (Theorem \ref{microscopiccomponents}). Finally,  in
Section \ref{supercriticalsection} we conclude with the proof of Theorem
\ref{manymesoscopiccomponents} showing that  for $h<h_\star$ most
vertices of $\finitegraph$ live in a connected component of
$E_{\Psi_\finitegraph}^{\geq h}$ of at least mesoscopic size.

\begin{acknowledgements}
  The authors wish to express their gratitude to A.-S.~Sznitman for
  suggesting the problem and for the valuable comments made at various
  stages of the project.
\end{acknowledgements}



\section{Notation and useful results}
\label{section1}

In this section we introduce our main notation and recall the essential
material about the Gaussian free field on the $d$-regular tree
$\treegraph$ that will be needed in the study of the zero-average
Gaussian free field on the finite graphs $(\finitegraph)_{n\geq 1}$
(Section \ref{sectiontree}). We end the section with results on the
zero-average Green function and some basic properties of the zero-average
Gaussian free field on $\finitegraph$ (Section \ref{subsection1.2}).

As mentioned earlier, we consider for fixed $d\geq 3$ the $d$-regular
graphs $(\finitegraph)_{n\geq 1}$, satisfying the assumptions
\eqref{0}--\eqref{2}. For the constants $\alpha$ and $\beta$ appearing in
these assumptions we assume without loss of generality that
\begin{equation}
  \label{wlog}
  \alpha \leq 1 \qquad \text{and} \qquad \beta \leq 2.
\end{equation}
Indeed, for $\alpha$ this is trivial and for $\beta$ it follows from the
fact that the  matrix $P$ (see below \eqref{2})  is a symmetric
stochastic matrix and thus all its eigenvalues are contained in the
interval $[-1,1]$. Consequently the eigenvalues of $I-P$ are contained in
$[0,2]$.

For the general graph notation introduced in the next two paragraphs,
$\mathcal G$ stands either for $\finitegraph$ or for $\treegraph$ with
root $\rot$.

By $x\in \mathcal G$ resp.~$U\subseteq \mathcal G$ we mean a vertex
resp.~a subset of vertices of the graph $\mathcal G$. We let
$d_\mathcal G(\cdot,\cdot)$ denote the graph distance on $\mathcal G$.
For any  $U\subseteq \mathcal G$,  $|U|$ stands for its cardinality, and
$\partial_\mathcal G U \coloneqq \{ y \in \mathcal G \setminus U \,
  | \,  y \text{ has some neighbour } x\in U \text{ in } \mathcal G \}$
denotes its (outer) boundary in $\mathcal G$. For any $R\geq0 $ and
$x \in \mathcal G$ we define the balls and spheres of radius $R $ around
$x$ to be
$B_\mathcal G(x,R)
\coloneqq \{y \in \mathcal G \, | \, d_\mathcal G(x,y) \leq R \} $
and
$S_\mathcal G(x,R)
\coloneqq \{y \in \mathcal G \, | \, d_\mathcal G(x,y) = R \} $.
The maximum number of edges that can be deleted from the subgraph of
$\mathcal G$ induced by some connected subset $U \subseteq \mathcal G $
while keeping it connected is called tree excess of $U $ and we denote it
by $\tx(U)$. Note that $\tx(U)=0 $ if and only if (the subgraph induced
  by) $U $ is a tree. (In particular, the assumption \eqref{1} could be
  rewritten as
  $\tx(B_\finitegraph(x,\lfloor \alpha \log_{d-1}(N_n) \rfloor )) \leq 1$
  for all $n\geq 1$ and $x \in \finitegraph$.) For $x,z \in \mathcal G$ a
path from $x$ to $z $ is a sequence of vertices $x=y_0,y_1,\ldots,y_m=z$
in $\mathcal G $ for some $m\geq 0$ such that $y_i$ and $y_{i-1} $ are
neighbours for all $i=1,\ldots,m$ (if $m\geq 1 $). It is a
\emph{non-backtracking} path from $x$ to $z $ if in addition
$y_i \neq y_{i-2}$ for all $i=2,\ldots,m$ (if $m\geq 2$).

We write $P_x^{\mathcal G}$ for the canonical law of the simple random
walk on $\mathcal G$ starting at $x \in \mathcal G$ as well as
$E_x^{\mathcal G}$ for the corresponding expectation. The canonical
process for the discrete-time walk is denoted by $(X_k)_{k\geq 0}$. For
the continuous-time walk with i.i.d.~mean-one exponential holding times
we write $(\overline{X}_t)_{t\geq 0}$. Given $U\subseteq \mathcal G$ we
write $T_U \coloneqq \inf \{ {k\geq 0} \, | \, {X_k  \notin U}\}$ for the
exit time from $U$ and
$H_U \coloneqq \inf \{ {k\geq 0} \, | \, X_k \in U \}$ for the entrance
time in $U$ of the discrete-time walk (here we set
  $\inf \emptyset \coloneqq \infty$). For the continuous-time simple
random walk $T_U$ and $H_U$ are defined accordingly. In the special case
of $U=\{z \}$ we use $H_z$ in place of $H_{\{z\}}$.

For $\mathcal G= \treegraph$ we need some extra notation. In this case,
there is a unique non-backtracking path of length $d_\treegraph(x,z)$
between any two vertices $x,z \in \treegraph$ (namely the \emph{geodesic}
  path). For $x\in \treegraph \setminus \{\rot\}$ let $\overline x$
be the unique neighbour of $x$ on the non-backtracking path from $x$ to
$\rot$. Moreover, let $\overline{\rot} \in \treegraph$ denote
a fixed neighbour of the root $\rot \in \treegraph$. For
$x\in \treegraph$ we define
\begin{equation}
  \label{830}
  U_x  \coloneqq   \{  z \in \treegraph \,
    | \, \text{the non-backtracking path from
      $z$ to $x$ does not contain $\overline x$} \}.
\end{equation}
In particular $\treegraph = \{\rot \} \cup \bigcup_{i=1}^d U_{x_i}$
if $S_\treegraph(\rot,1) \eqqcolon \{x_1,\ldots,x_d\}$. In the
special case of $x = \rot$ we  write
$\treegraphplus \coloneqq U_\rot$. We also set
$B_\treegraph^+(\rot,R)
\coloneqq \{y \in \treegraphplus \, | \, d_\treegraph(\rot,y) \leq R \}$
and similarly
$S_\treegraph^+(\rot,R)
\coloneqq \{y \in \treegraphplus \, | \, d_\treegraph(\rot,y) = R \}$
for $R\geq 0$.

Finally, some notation for the finite graphs $(\finitegraph)_{n\geq 1}$.
For all $n\geq 1$ and $x\in \finitegraph$ we fix a cover tree $\pi_{n,x}$
of $\finitegraph$ at $x$, that is, a surjective map
$\pi_{n,x}: \treegraph \to \finitegraph$ such that
$\pi_{n,x} (\rot) = x$ and such that for all $y \in \treegraph$ one
has $\pi_{n,x}(S_\treegraph(y,1)) = S_\finitegraph(\pi_{n,x}(y),1)$,
meaning that $\pi_{n,x}$ preserves the neighbourhood of radius 1  of any
$y \in \treegraph$. Note that:
\begin{flalign}
  \quad \ \,  \bullet& \ \, \text{if $x \in \finitegraph$ with
    $\tx(B_\finitegraph(x,R))=0$ for some $R \geq 0$, then the map
    $\pi_{n,x}$ restricted} & \nonumber \\
  \quad \ \,  & \ \,   \text{to $B_\treegraph(\rot,R)$ induces a graph
    isomorphism from $B_\treegraph(\rot,R)$ to $B_\finitegraph(x,R)$} &
  \label{20} \\
  \quad \ \, \bullet& \ \, \text{a sequence of vertices  $\rot=y_0,
    y_1,\ldots, y_m \in \treegraph$, $m\geq 0$, is a non-backtracking}
  & \nonumber  \\
  \quad \ \,  & \ \, \text{path in $\treegraph$ starting at $\rot$ if and only if
    $x=\pi_{n,x}(y_0),\pi_{n,x}(y_1),\ldots,\pi_{n,x}(y_m)\in \finitegraph$} &
  \nonumber \\
  \quad \ \,  & \ \, \text{is a non-backtracking path in $\finitegraph$ starting at $x$.} &
  \label{21}
\end{flalign}
Furthermore, for the cover tree $\pi_{n,x}$ of $\finitegraph$ at $x$, the
process $(\pi_{n,x}( X_k))_{k\geq 0}$ under $P_\rot^\treegraph$ has
the same law as $(X_k)_{k\geq 0}$ under $P_x^\finitegraph$. Hence
\begin{equation}
  \label{4}
  P_x^\finitegraph[X_k \in U] = P_\rot^\treegraph[\pi_{n,x}(X_k) \in U]=
  P_\rot^\treegraph[X_k \in \pi_{n,x}^{-1}(U)] \quad \text{for } U
  \subseteq \finitegraph, \, k\geq 0.
\end{equation}

A final word on the convention followed concerning constants: by
$c,c',\ldots$ we denote  positive constants with values changing from
place to place and which only depend on the dimension $d$ and the
constants $\alpha$ and $\beta$ from the assumptions \eqref{0}--\eqref{2}.
Numbered constants $c_0,c_1,\ldots$ are defined in the place of first
occurrence and thereafter remain fixed. The dependence of constants on
additional parameters  appears in the notation.



\subsection{Some properties of the Gaussian free field on regular trees}
\label{sectiontree}

In this  section we recall basic facts related to the Green function and
the Gaussian free field on $\treegraph$. We also restate a couple of
results about $\varphi_\treegraph$ that were  derived by the authors in
the accompanying paper \cite{AC1} and that will be used in several
occasions throughout the rest of this article.

The Green function $g_\treegraph(\cdot,\cdot)$ of simple random walk on
$\treegraph$ is (see \cite{33}, Lemma 1.24, for the explicit computation)
\begin{equation}
  \label{1.1}
  g_\treegraph(x,y) \coloneqq E_x^\treegraph \Big[\sum_{k=0}^\infty
    \bbone_{\{X_k=y\}} \Big] =
  \frac{d-1}{d-2} \Big( \frac1{d-1} \Big)^{d_\treegraph(x,y)} \ \ \text{for } x,y
  \in \treegraph.
\end{equation}
For $U\subseteq \treegraph$ the Green function $g_\treegraph^U(\cdot,\cdot)$
of simple random walk on $\treegraph$ killed when exiting $U$ is
$g_\treegraph^U(x,y)
\coloneqq E_x^\treegraph \big[\sum_{0\leq k < T_U} \bbone_{\{X_k=y\}} \big]$.
The functions
$g_\treegraph(\cdot,\cdot)$ and $g_\treegraph^U(\cdot,\cdot)$ are related by
the identity
\begin{equation}
  \label{1.4}
  g_\treegraph(x,y) = g_\treegraph^U(x,y) + E_x^\treegraph
  \big[g_\treegraph(X_{T_U},y) \bbone_{\{T_U < \infty\}}\big] \quad \text{for
  } x,y \in \treegraph.
\end{equation}

We continue by collecting known results and properties of
$\varphi_\treegraph$. Recall from \eqref{0.2} that
$(\varphi_\treegraph(x))_{x \in \treegraph}$ is the centred Gaussian
field with covariance given by $g_\treegraph(\cdot,\cdot)$. An important
feature of the Gaussian free field is the domain Markov property: for
$U\subseteq \treegraph$ let $(\varphi_\treegraph^U(x))_{x\in \treegraph}$
be a new field defined by
\begin{equation*}
  \varphi_\treegraph^U(x) \coloneqq
  \varphi_\treegraph(x) - E_x^\treegraph[ \varphi_\treegraph(X_{T_U})
    \bbone_{\{T_U < \infty\}} ] \quad \text{for } x\in \treegraph.
\end{equation*}
Then,
\begin{equation}
  \label{1.18}
  \parbox{12.3cm}{
    under $\mathbb P^\treegraph$,
    $(\varphi^U_\treegraph(x))_{x \in \treegraph}$ is a
    centred Gaussian field on $\treegraph$
    which is independent from
    $(\varphi_\treegraph(x))_{x\in \treegraph \setminus U}$ and has
    covariance
    $\mathbb E^\treegraph [\varphi_\treegraph^U(x) \varphi_\treegraph^U(y)]
    = g_\treegraph^U(x,y)$ for all $x,y \in \treegraph$.
  }
\end{equation}

As a consequence of \eqref{1.18}, the Gaussian free field on $\treegraph$
can be obtained by the following recursive construction (explained in
  detail in \cite{AC1}, Section 1.1). Let $(Y_x)_{x\in\treegraph}$ be a
collection of independent centred Gaussian variables defined on some
auxiliary probability space $(\Omega ,\mathcal A, \mathbb P)$ such that
$Y_\rot \sim \mathcal N(0,g_\treegraph(\rot,\rot)) =  \mathcal N(0,\frac{d-1}{d-2})$
and
$Y_x\sim \mathcal N(0,g_\treegraph^{U_x}(x,x)) = \mathcal N(0,  \frac{d}{d-1})$
for $x\neq \rot$. Define recursively
\begin{equation}
  \label{2544}
  \widetilde \varphi (\rot) \coloneqq Y_\rot \quad \text{ and }
  \quad \widetilde \varphi(x) \coloneqq \frac 1 {d-1} \widetilde \varphi (\overline x)
  + Y_x \quad\text{for }x \in \treegraph \setminus \{ \rot \}.
\end{equation}
Then,
\begin{equation}
  \label{2545}
  \text{under $\mathbb P$, the law of $(\widetilde \varphi (x))_{x\in \treegraph}$
    is $\mathbb P^\treegraph$,}
\end{equation}
so that \eqref{2544} can be used as an alternative description of
$(\varphi_\treegraph(x))_{x\in \treegraph}$. In particular, it gives  a
representation of the conditional distribution of $\varphi_\treegraph$
given $\varphi_\treegraph(\rot)=a\in \mathbb R$,
\begin{equation}
  \label{e:defPa}
  \mathbb P_a^\treegraph
  \big[ (\varphi_\treegraph(y))_{y\in \treegraph} \in \cdot \,  \big]
  \coloneqq \mathbb P^\treegraph \big[
    (\varphi_\treegraph(y))_{y\in \treegraph} \in \cdot \ \big| \,
    \varphi_\treegraph(\rot)=a \big],
\end{equation}
with corresponding expectation $\mathbb E_a^\treegraph$.

We turn to known results about level-set percolation of the Gaussian free
field on $\mathbb T^d$ from \cite{35} and \cite{AC1}. First, there is a
characterisation of the critical value $h_\star$ through eigenvalues
$(\lambda_h)_{h\in \mathbb R}$ of certain self-adjoint operators
$(L_h)_{h\in \mathbb R}$ (see \cite{35}, Section~3, summarised in
  \cite{AC1}, Proposition 1.1). Important for us will be that (see
  \cite{35}, Proposition 3.3)
\begin{equation}
  \label{19}
  \parbox{12cm}{
    the map $h \mapsto \lambda_h$ is a decreasing homeomorphism from $\mathbb R$ to
    $(0,d-1)$ and $h_\star$  is the unique value in $\mathbb R$ such that
    $\lambda_{h_\star} = 1$.
  }
\end{equation}
To restate the other results we remind that
$\mathcal C_{\rot}^{\treegraph,h}$ denotes the  connected component
of the level set $E_{\varphi_\treegraph}^{\geq h}$ above level $h$
containing the root $\rot \in \treegraph$ (see below \eqref{0.5}).
The second result says that (see \cite{AC1}, Theorem 4.1)
\begin{equation}
  \label{2512}
  \parbox{12cm}{
    the `forward percolation probability'  $h \mapsto \eta^+(h)$ given by
    $\eta^+(h) \coloneqq  \mathbb P^\treegraph \big[\big| \mathcal
      C_{\rot}^{\treegraph,h} \cap \treegraphplus \big| = \infty \big]$  is
    continuous and positive on $(-\infty,h_\star)$ and vanishes on
    $(h_\star,\infty)$.
  }
\end{equation}
The third result  controls the subcritical behaviour
(see \cite{AC1}, Theorem 5.1). It shows  that
\begin{equation}
  \label{2513}
  \parbox{12cm}{
    for $h > h_\star$ there exists $\delta_h>0$ such that $g_{h}(a) \coloneqq
    \mathbb E_a^\treegraph \big[ (1+\delta_h)^{| \mathcal
        C_{\rot}^{\treegraph,h}\cap \treegraphplus   |} \big]$ defines a finite function,
    continuous on $[h,\infty)$. Furthermore, $g_{h}(a) = (1+\delta_h)
    \mathbb E^Y \big[ g_{h} (\tfrac{a}{d-1} + Y ) \big]^{d-1}$ for all $a\geq h$,
    where $Y \sim \mathcal N(0,\tfrac{d}{d-1})$ and  $\mathbb E^Y$ is taken with
    respect to $Y$.
    Moreover, there exist $c_{h}, c_{h}'>0$ such
    that
    $g_{h}(a) \leq c_{h,\gamma} \exp(c_{h}' a^{3/2})$ for all $a\geq h$.
  }
\end{equation}
Finally, the last result about $\varphi_\treegraph$ needed in the sequel
in the supercritical regime is the following fact in which the $\lambda_h$,
$h\in \mathbb R$, from \eqref{19} appear: by \cite{AC1}, Theorem 4.3,
\begin{equation}
  \label{2141}
  \text{for $h<h_\star$ it holds that } \lim_{k\to \infty}  \mathbb P^\treegraph
  \Big[  \big|\mathcal C_{\rot}^{\treegraph,h} \cap
    S_\treegraph^+(\rot,k) \big| \geq \frac{\lambda_{h}^{k}}{k^2} \Big] =
  \eta^+(h) >0.
\end{equation}



\subsection{The Green function and the zero-average Gaussian free field on
  \texorpdfstring{$\finitegraph$}{finite graphs}}
\label{subsection1.2}

We now introduce the zero-average Green function associated to the
simple random walk on $\finitegraph$ and prove an upper bound on it
(Proposition \ref{proposition1.1}). Along the way we also remind of a
basic property of the zero-average Gaussian free field on $\finitegraph$
of similar type as \eqref{1.18} (see \eqref{450} and \eqref{460}).

The zero-average Green function $G_\finitegraph(\cdot,\cdot)$ associated
with the simple random walk on $\finitegraph$ is given by
\begin{equation}
  \label{1.5}
  G_\finitegraph(x,y) \coloneqq  \int_0^\infty \Big(
    P_x^\finitegraph[\overline{X}_t = y] - \frac1{N_n} \Big) \, dt \quad \text{for
  } x,y \in \finitegraph.
\end{equation}
It is symmetric, finite and positive-semidefinite, i.e.~for any
$f: \finitegraph \to \mathbb R$ one has
$\sum_{x,y \in \finitegraph} f(x) G_\finitegraph(x,y) f(y) \geq 0$ (see
  \cite{36}, Remark 1.2). For $U\subseteq \finitegraph$ we define
$g_\finitegraph^U(\cdot,\cdot)$ to be the Green function of simple random
walk on $\finitegraph$ killed when exiting $U$, that is,
\begin{equation}
  \label{1.6}
  g_\finitegraph^U(x,y) \coloneqq E_x^\finitegraph \Big[\sum_{0\leq k < T_U}
    \bbone_{\{X_k=y\}} \Big] = \sum_{k=0}^\infty P_x^\finitegraph[X_k=y, k<T_U]
  \quad \text{for } x,y \in \finitegraph.
\end{equation}
As $g_\treegraph^U(\cdot,\cdot)$ it is symmetric, finite and vanishes for
$x\notin U$ or $y\notin U$. The functions $G_\finitegraph(\cdot,\cdot)$
and $g_\finitegraph^U(\cdot,\cdot)$ are related by a similar expression
as the identity \eqref{1.4} for the Green functions on $\treegraph$. More
precisely, for $U \subsetneq \finitegraph$  it holds (see \cite{36},
  Lemma 1.4)
\begin{equation}
  \label{1.7}
  G_\finitegraph(x,y) = g_\finitegraph^U(x,y) + E_x^\finitegraph
  \big[G_\finitegraph(X_{T_U},y) \big] -\frac{1}{N_n} E_x^\finitegraph [T_U]
  \quad \text{for } x,y \in \finitegraph.
\end{equation}
(Lemma 1.4 in \cite{36} is stated  in the case of a discrete $d$-dimensional
  torus as underlying graph. However, its proof applies as well to the
  graph $\finitegraph$.)

Recall from \eqref{0.1} that $(\Psi_\finitegraph(x))_{x \in \finitegraph}$
is the centred Gaussian field with covariance given by
$G_\finitegraph(\cdot,\cdot)$. We point out that the Green function
$G_\finitegraph(\cdot,\cdot)$ is called 'zero-average' since its average
over $\finitegraph$ in any of the two arguments is zero. This implies
that the average of $\Psi_\finitegraph(x)$ over $x \in \finitegraph$
vanishes $\mathbb P^\finitegraph$-almost surely and explains the name
'zero-average Gaussian free field'.

In the same way as the identity \eqref{1.4} allows for the property
\eqref{1.18} of the Gaussian free field $\varphi_\treegraph$ on
$\treegraph$, the identity \eqref{1.7} implies a similar (but not equal)
property of the zero-average Gaussian free field $\Psi_\finitegraph$ on
$\finitegraph$. It is given below and follows from \cite{36}, Lemma 1.7.
There it is stated and proved for the zero-average Gaussian free field on
the discrete $d$-dimensional torus but the proof applies, with the
obvious adjustments, also to our situation. For
$U \subsetneq \finitegraph$ set
\begin{equation}
  \label{450}
  \varphi_\finitegraph^U(x) \coloneqq  \Psi_\finitegraph(x) - E_x^\finitegraph[
    \Psi_\finitegraph(X_{T_U}) ] \quad \text{for $x\in \finitegraph$}.
\end{equation}
Then,
\begin{equation}
  \label{460}
  \parbox{12cm}{
    under $\mathbb P^\finitegraph$, $(\varphi_\finitegraph^U(x))_{x\in
      \finitegraph}$ is a centred Gaussian field on $\finitegraph$ with covariance
    $\mathbb E^\finitegraph [\varphi_\finitegraph^U(x) \varphi_\finitegraph^U(y)] =
    g_\finitegraph^U(x,y)$ for all $x,y \in \finitegraph$.
  }
\end{equation}
Note that $(\varphi_\finitegraph^U(x))_{x\in \finitegraph}$ cannot be
independent from $(\Psi_\finitegraph(x))_{x\in \finitegraph \setminus U}$
due to the zero-average property of $\Psi_\finitegraph$.

We conclude Section \ref{section1} with an upper bound on
$G_\finitegraph(\cdot,\cdot)$ which is going to be of particular use in
the proof of Proposition \ref{proposition2.1} needed for the
supercritical phase. Note that the obtained bound \eqref{30}  resembles
the expression for the Green function $g_\treegraph(\cdot,\cdot)$ on
$\treegraph$ (see \eqref{1.1}). We first define the new constant
\begin{equation}
  \label{32}
  c_0 \coloneqq  \frac{\alpha \beta}{d-1} \overset{\eqref{wlog}}{\in} (0,1) .
\end{equation}
\begin{proposition}
  \label{proposition1.1}
  For all $n\geq 1$ and $x,y \in \finitegraph$ it holds that
  \begin{equation}
    \label{1.8}
    G_\finitegraph(x,y) \leq  \frac{16}{7} \frac{d-1}{d-2} \Big( \frac1{d-1}
      \Big)^{d_\finitegraph(x,y)}+  2 \ln(N_n) N_n^{-\frac{c_0}{\beta}}   +
    \frac1{\beta N_n^{c_0}}.
  \end{equation}
  In particular, for all $n$ large enough and $x,y \in \finitegraph$ with
  $d_\finitegraph(x,y) \leq \frac{c_0}{3}\log_{d-1}(N_n)$ it holds that
  \begin{equation}
    \label{30}
    G_\finitegraph(x,y) \leq  3  \frac{d-1}{d-2} \Big( \frac1{d-1}
      \Big)^{d_\finitegraph(x,y)}.
  \end{equation}
\end{proposition}
\begin{proof} \let\qed\relax
  We set
  $t_\finitegraph \coloneqq \frac{c_0}{\beta} \ln(N_n)
  =\frac{\alpha}{d-1} \ln(N_n) \overset{\eqref{wlog}}{\leq} \ln(N_n)$.
  By \cite{44}, Corollary 2.1.5, one then has (the stationary
    distribution of $(\overline{X}_t)_{t \geq 0}$ is the uniform
    distribution on $\finitegraph$ due to \eqref{0})
  \begin{equation}
    \label{1.12}
    \begin{split}
      \int_{t_\finitegraph}^\infty \Big| P_x^\finitegraph[\overline{X}_t = y] -
      \frac1{N_n}  \Big| \, dt
      \leq \int_{t_\finitegraph}^\infty  e^{-\lambda_\finitegraph t} \, dt
      =\frac{e^{-\lambda_\finitegraph t_\finitegraph}}{\lambda_\finitegraph}
      \overset{\eqref{2}}{\leq}  \frac{1}{\beta N_n^{c_0}} .
    \end{split}
  \end{equation}
  On the other hand, by switching to the discrete-time walk
  $(X_k)_{k\geq 0}$ and with $M_t\sim \textup{Poi}(t)$ for $t \geq 0$
  describing the number of jumps of the continuous-time simple random
  walk up to time $t$, we have
  \begin{equation}
    \label{839}
    \begin{split}
      &\int_0^{t_\finitegraph} \Big| P_x^\finitegraph[\overline{X}_t = y] -
      \frac1{N_n}  \Big| \, dt  \leq \int_0^{t_\finitegraph} \sum_{k=0}^\infty
      \mathbb P[M_t = k] P_x^\finitegraph[X_k = y] \, dt + \frac{t_\finitegraph}{N_n}
      \\
      &\leq \sum_{k=0}^{\lfloor \alpha \log_{d-1}(N_n) \rfloor} \!\!\!
      P_x^\finitegraph[X_k = y]  \underbrace{\int_0^\infty \frac{t^k}{k!}e^{-t} \,
        dt}_{=1} + \int_0^{t_\finitegraph} \mathbb P[M_t \geq \alpha \log_{d-1}(N_n)]
      \, dt + \frac{t_\finitegraph}{N_n}.
    \end{split}
  \end{equation}
  Note that for $0\leq t \leq t_\finitegraph$  by  Markov's inequality one has
  $\mathbb P[M_t \geq \alpha \log_{d-1}(N_n)] =\mathbb P[(d-1)^{M_t} \geq
    N_n^\alpha]  \leq N_n^{-\alpha} \mathbb E[e^{\ln(d-1)M_t}] = N_n^{-\alpha}
  \exp(t(d-2)) \leq N_n^{-\alpha}  \exp(t_\finitegraph(d-2))= N_n^{-\alpha}
  \exp(\alpha \ln(N_n)) \exp(-t_\finitegraph) =
  \exp(-t_\finitegraph)=N_n^{-{c_0}/\beta}$. Therefore \eqref{839} implies
  \begin{equation}
    \label{1.13}
    \begin{split}
      \int_0^{t_\finitegraph} \Big| &P_x^\finitegraph[\overline{X}_t = y] -
      \frac1{N_n}  \Big| \, dt  \leq \sum_{k=0}^{\lfloor \alpha \log_{d-1}(N_n)
        \rfloor} P_x^\finitegraph[X_k = y]  + t_\finitegraph N_n^{-\frac{c_0}\beta} +
      \frac{t_\finitegraph}{N_n} \\
      &\stackrel[\eqref{32}]{\eqref{4}}{\leq}  \sum_{k=0}^{\lfloor \alpha
        \log_{d-1}(N_n) \rfloor} P_\rot^\treegraph[X_k \in \pi_{n,x}^{-1}(\{y\})]
      + 2 \ln(N_n) N_n^{-\frac{c_0}{\beta}}
    \end{split}
  \end{equation}
  for the cover tree $\pi_{n,x}$ of $\finitegraph$ at $x$. To bound the
  sum appearing on the right hand side of \eqref{1.13} we consider
  different cases for
  $\pi_{n,x}^{-1}(\{y\}) \cap
  B_\treegraph(\rot,\lfloor\alpha \log_{d-1}(N_n)\rfloor )$.

  If $\big| \pi_{n,x}^{-1}(\{y\}) \cap B_\treegraph(\rot,\lfloor\alpha
    \log_{d-1}(N_n)\rfloor ) \big| = 0$, then
  the sum on  the last line of \eqref{1.13} vanishes and together with
  \eqref{1.12} this shows \eqref{1.8}.

  If $\big| \pi_{n,x}^{-1}(\{y\}) \cap B_\treegraph(\rot,\lfloor\alpha
    \log_{d-1}(N_n)\rfloor ) \big| = 1$, say the  intersection is $\{u\}$ (this is
    in particular the case if $y \in B_\finitegraph(x,\lfloor\alpha
      \log_{d-1}(N_n)\rfloor )$ and $\tx(B_\finitegraph(x,\lfloor\alpha
        \log_{d-1}(N_n)\rfloor ))=0$),
  then the sum appearing on the right hand side of \eqref{1.13} can be
  rewritten as
  \begin{equation*}
    \begin{split}
      \sum_{k=0}^{\lfloor \alpha \log_{d-1}(N_n) \rfloor}
      P_\rot^\treegraph[&X_k \in \pi_{n,x}^{-1}(\{y\})]
      \overset{\phantom{\eqref{1.1}}}{=}
      \sum_{k=0}^{\lfloor \alpha \log_{d-1}(N_n) \rfloor}
      P_\rot^\treegraph[X_k = u]
      \overset{\eqref{1.1}}{\leq} g_\treegraph(\rot,u) \\
      &   \overset{\eqref{1.1}}{=} \frac{d-1}{d-2} \Big( \frac1{d-1}
        \Big)^{d_\treegraph(\rot,u)} = \frac{d-1}{d-2} \Big( \frac1{d-1}
        \Big)^{d_\finitegraph(x,y)}
    \end{split}
  \end{equation*}
  and together with \eqref{1.12} this shows \eqref{1.8}.

  It remains to consider the last case, that is,
  $\big| \pi_{n,x}^{-1}(\{y\}) \cap
  B_\treegraph(\rot,\lfloor\alpha \log_{d-1}(N_n)\rfloor ) \big|  \geq 2$.
  Then $B_\finitegraph(x,\lfloor\alpha \log_{d-1}(N_n)\rfloor )$ contains
  a (unique by \eqref{1}) cycle of some length $\ell$. Let us abbreviate
  $B \coloneqq B_\treegraph(\rot,\lfloor\alpha \log_{d-1}(N_n)\rfloor )$
  and define for $m\geq 0$ the disjoint intervals
  $I_m \coloneqq [d_\finitegraph(x,y)+m\ell,d_\finitegraph(x,y)+(m+1)\ell)$
  of length  $\ell$. We claim that  one has the disjoint union
  \begin{equation}
    \label{40}
    \begin{split}
      &\pi_{n,x}^{-1}(\{y\}) \cap B = \bigcup_{m=0}^\infty \big\{z \in
        \pi_{n,x}^{-1}(\{y\}) \cap B \, | \, d_\treegraph(\rot,z) \in I_m \big\}
      \\
      &\text{with } \big| \big\{z \in \pi_{n,x}^{-1}(\{y\}) \cap B \, | \,
        d_\treegraph(\rot,z) \in I_m \big\} \big| \leq 2 \text{ for }  m\geq 0.
    \end{split}
  \end{equation}
  This fact is a direct consequence of Lemma \ref{unhappylemma} stated
  and proved below. We first conclude the proof of Proposition
  \ref{proposition1.1} assuming \eqref{40}.  The sum on the last line of
  \eqref{1.13} can be bounded, in case
  $\big| \pi_{n,x}^{-1}(\{y\}) \cap B_\treegraph(\rot,
    \lfloor\alpha \log_{d-1}(N_n)\rfloor ) \big| \geq 2$,
  by
  \begin{equation}
    \label{1.15}
    \begin{split}
      & \sum_{k=0}^{\lfloor \alpha \log_{d-1}(N_n) \rfloor}
      P_\rot^\treegraph[X_k \in \pi_{n,x}^{-1}(\{y\})]
      \overset{\phantom{\eqref{40}}}{\leq} \sum_{k=0}^{\infty}  \sum_{z \in
        \pi_{n,x}^{-1}(\{y\}) \cap B}  P_\rot^\treegraph[X_k =z] \\
      & \quad \stackrel{\, \eqref{1.1}\,}{=} \sum_{z \in
        \pi_{n,x}^{-1}(\{y\}) \cap B} \frac{d-1}{d-2} \Big( \frac1{d-1}
        \Big)^{d_\treegraph(\rot,z)}
      \overset{\eqref{40}}{\leq}  2 \frac{d-1}{d-2} \sum_{m=0}^\infty  \Big(
        \frac1{d-1} \Big)^{d_\finitegraph(x,y)+m\ell} \\
      &  \quad \overset{\phantom{\eqref{40}}}{=}   2 \frac{d-1}{d-2} \Big(
        \frac1{d-1} \Big)^{d_\finitegraph(x,y)}  \frac{1}{1-\big( \frac1{d-1}
          \big)^{\ell}}
      \overset{\phantom{\eqref{40}}}{\leq} \frac{16}{7}  \frac{d-1}{d-2} \Big(
        \frac1{d-1} \Big)^{d_\finitegraph(x,y)},
    \end{split}
  \end{equation}
  where in the last step we use that $d\geq 3$ and $\ell \geq 3$, too, since
  $\ell$ is the length of a cycle.
  The combination of \eqref{1.12}, \eqref{1.13}  and \eqref{1.15}  concludes the
  proof of \eqref{1.8} also in this case, once \eqref{40} is asserted.
  To derive \eqref{30} from \eqref{1.8} it is enough to recall  that $c_0 \leq 1$
  and $\beta \leq 2$ (see \eqref{32} and \eqref{wlog}). Hence one has
  $\frac{c_0}{3} < \min\{c_0,\frac{c_0}{\beta}\}$ and therefore for $n$ large
  enough also
  \begin{equation}
    \label{31}
    2 \ln(N_n) N_n^{-\frac{c_0}{\beta}} + \frac1{\beta N_n^{c_0}}  \leq
    \frac1{N_n^{\frac{c_0}{3}}} = \Big( \frac1{d-1}
      \Big)^{\frac{c_0}{3}\log_{d-1}(N_n)}  \leq \underbrace{\frac{5}{7}
      \frac{d-1}{d-2}}_{\geq 1} \Big( \frac1{d-1} \Big)^{d_\finitegraph(x,y)},
  \end{equation}
  assuming $x,y \in \finitegraph$ are such that $d_\finitegraph(x,y) \leq
  \frac{c_0}{3}\log_{d-1}(N_n)$. We can combine \eqref{1.8} with \eqref{31} to
  obtain \eqref{30}.

  To conclude the proof of Proposition \ref{proposition1.1} it
  only remains to show \eqref{40}, which follows directly from the next lemma.
\end{proof}

\begin{lemma}
  \label{unhappylemma}
  Let $x\in \finitegraph$, $R\geq 0$ and assume $B_\finitegraph(x,R)$ contains a
  unique cycle of length $\ell$. Recall that $\pi_{n,x}$ is the fixed cover tree
  of $\finitegraph$ at $x$ and assume  $y\in B_\finitegraph(x,R)$. Then
  \begin{equation}
    \label{41}
    \Big| \big\{ z \in \pi_{n,x}^{-1}(\{y\}) \cap B_\treegraph(\rot,R) \,
      \big| \,  d_\treegraph(\rot,z) \in [0,d_\finitegraph(x,y)) \big\} \Big| =
    0.
  \end{equation}
  Moreover, for all $k \geq 0$ one has
  \begin{equation}
    \label{42}
    \begin{split}
      \Big| \big\{ z \in \pi_{n,x}^{-1}(\{y\}) \cap B_\treegraph(\rot,R) \,
        \big| \,  d_\treegraph(\rot,z) \in [k,k+\ell) \big\} \Big| \leq 2.
    \end{split}
  \end{equation}
\end{lemma}
\begin{proof}
  For any vertex
  $z \in \pi_{n,x}^{-1}(\{y\}) \cap B_\treegraph(\rot,R)$ there is
  a unique non-backtracking path of length $d_\treegraph(\rot,z)$
  from $\rot$ to $z$ in $B_\treegraph(\rot,R)$. Therefore, by
  the one-to-one correspondence from \eqref{21}, every such $z$ uniquely
  determines a non-backtracking path of length
  $d_\treegraph(\rot,z)$ connecting $x$ to $y$ in
  $B_\finitegraph(x,R)$. Thus \eqref{41} is clear and for \eqref{42} it
  is enough to show that for all $k\geq 0$ one has
  \begin{equation}
    \label{99}
    \Big|\big\{\text{non-backtracking paths from $x$ to $y$ in
        $B_\finitegraph(x,R)$ of length in $[k,k\!+\!\ell)$} \big\}\Big|
     \leq  2.
  \end{equation}

  Let us denote by
  $C \coloneqq \{c_1,\ldots,c_\ell\} \subseteq \finitegraph$ the unique
  cycle of length $\ell$ in $B_\finitegraph(x,R)$ and by
  $x=x_0,\ldots,x_i$ for some $i\geq 0$ the unique non-backtracking path
  in $B_\finitegraph(x,R)$ from $x$ to $C$ such that $x_i \in C$ and
  $x_0,\ldots,x_{i-1} \notin C$ (if $i\geq1$). This path is unique for if
  $x=\widetilde x_0,\ldots,\widetilde x_{j}$ was another such path, then
  one could find a cycle different from $C$ in
  $\{x_0,\ldots,x_i,\widetilde x_0,\ldots,
    \widetilde x_{j},c_1,\ldots,c_\ell \} \subseteq B_\finitegraph(x,R)$.
  Analogously, we let $y=y_0,\ldots,y_j$ for some $j\geq 0$ be the unique
  non-backtracking path in $B_\finitegraph(x,R)$ from $y$ to $C$ such
  that $y_j \in C$ and $y_0,\ldots,y_{j-1} \notin C$ (if $j\geq1$). We
  distinguish two cases: either
  $\{x_0,\ldots,x_i \} \cap \{y_0,\ldots,y_j\} = \emptyset$ or the
  intersection is not empty.

  In the first case any non-backtracking path from $x$ to $y$ in
  $B_\finitegraph(x,R)$ starts with the segment $x_0,\ldots,x_i$ from $x$
  to $C$ and  ends with the segment $y_j,\ldots,y_0$ from $C$ to $y$
  because a  non-backtracking path $v_0,\ldots,v_s$ from $x$ to $y$ in
  $B_\finitegraph(x,R)$ with $(v_0,\ldots,v_i) \neq (x_0,\ldots,x_i)$ or
  $(v_{s-j},\ldots,v_s) \neq (y_j,\ldots,y_0)$ would imply the existence
  of a cycle different from $C$ in
  $\{ v_0,\ldots,v_s,c_1,\ldots,c_\ell,x_0,\ldots,x_i,y_0,\ldots,y_j \}
  \subseteq B_\finitegraph(x,R)$.
  In between the segments $x_0,\ldots,x_i$ and $y_j,\ldots,y_0$ any of
  those non-backtracking paths can only visit vertices in $C$ (else there
    would be another cycle in $B_\finitegraph(x,R)$) and they can only do
  so in clockwise or anti-clockwise direction (because they are
    non-backtracking). To wrap up: any non-backtracking path from $x$ to
  $y$ in $B_\finitegraph(x,R)$ starts with the segment $x_0,\ldots,x_i$,
  then goes $M$ times (for some $M\geq 0$ and some direction) around the
  cycle $C$ from $x_i$ to $x_i$, then continues (in the same direction)
  along the cycle from $x_i$ to $y_j$ (note that $x_i \neq y_j$ by
    assumption) and then ends with the segment $y_j,\ldots,y_0$.

  In the second case, i.e.~if
  $\{x_0,\ldots,x_i \} \cap \{y_0,\ldots,y_j\} \neq \emptyset$, let
  $m\in \{0,\ldots,i\}$ and $m' \in \{0,\ldots,j\}$ be such that
  $x_m=y_{m'}$ and
  $\{x_0,\ldots,x_{m-1} \} \cap \{y_0,\ldots,y_{m'-1}\} = \emptyset$. In
  other words, $x_m=y_{m'}$ is the first common vertex of the paths
  $x_0,\ldots,x_i$ and $y_0,\ldots,y_j$. Any non-backtracking path from
  $x$ to $y$ in $B_\finitegraph(x,R)$  starts with the segment
  $x_0,\ldots,x_m$ and  ends with the segment $y_{m'},\ldots,y_0$ because
  a  non-backtracking path $v_0,\ldots,v_s$ from $x$ to $y$ in
  $B_\finitegraph(x,R)$ with $(v_0,\ldots,v_m) \neq (x_0,\ldots,x_m)$ or
  $(v_{s-m'},\ldots,v_s) \neq (y_{m'},\ldots,y_0)$ would imply the
  existence of  a cycle  in
  $\{ v_0,\ldots,v_s,c_1,\ldots,c_\ell,x_0,\ldots,x_i,y_0,\ldots,y_j \}
  \subseteq B_\finitegraph(x,R)$
  different from  $C$. In between the segments $x_0,\ldots,x_m$ and
  $y_{m'},\ldots,y_0$ any of those non-backtracking paths either does not
  do anything (possible since $x_m=y_{m'}$ by definition, i.e.~the full
    path is $x_0,\ldots,x_m,y_{m'-1},\ldots,y_0$) or it has to form a
  non-backtracking path from $x_m$ to itself  of non-zero length. Note
  that in any graph a non-backtracking path (of non-zero length) from a
  vertex to itself necessarily contains vertices of a cycle. In our
  situation $C$ is the only cycle in $B_\finitegraph(x,R)$ and so any
  non-backtracking path (of non-zero length) from $x_m$ to $x_m$
  necessarily touches $C$. Therefore, it has to start with the segment
  $x_m,\ldots,x_i$ from $x_m$ to $C$ and end with the segment
  $x_i,\ldots,x_m$ from $C$ to $x_m$ (else there would be a cycle
    different from $C$ in $B_\finitegraph(x,R)$). Between the segments
  $x_m,\ldots,x_i$ and $x_i,\ldots,x_m$ it can only visit vertices in $C$
  (else there would be another cycle in $B_\finitegraph(x,R)$) and it has
  to do at least one full turn around the cycle in clockwise or
  anti-clockwise direction (because  non-backtracking). To wrap up: any
  non-backtracking path from $x$ to $y$ in $B_\finitegraph(x,R)$ is
  either of the form $x_0,\ldots,x_m,y_{m'-1},\ldots,y_0$ or between the
  initial segment    $x_0,\ldots,x_m$ and the final segment
  $y_{m'},\ldots,y_0$ it continues with the segment $x_m,\ldots,x_i$,
  then goes $M$ times (for some $M\geq 1$ and some direction) around the
  cycle $C$ from $x_i$ to $x_i$ and then goes back to $y_{m'}$ through
  $x_i,\ldots,x_m$.

  In any of the two cases, different non-backtracking paths from $x$ to
  $y$ in $B_\finitegraph(x,R)$ differ by at least $\ell$ in length (the
    length of the cycle) except if they go around the full cycle both $M$
  times but in different directions (clockwise or anti-clockwise). This
  shows \eqref{99} and concludes the proof of Lemma \ref{unhappylemma}
  and hence also of Proposition \ref{proposition1.1}.
\end{proof}



\section{The local picture of the zero-average Gaussian free field}
\label{localpicturesection}

In this section we investigate the local behaviour of the zero-average
Gaussian free field and we derive key results and estimates that will be
used in Section \ref{subcriticalsection} and Section
\ref{supercriticalsection} for proving the main theorems of this article
(Theorem \ref{microscopiccomponents} and Theorem
  \ref{manymesoscopiccomponents} corresponding to \eqref{0.6} and
  \eqref{0.7}). The results in this section support the intuition that
the \emph{local picture} of the zero-average Gaussian free field
$\Psi_\finitegraph$ on $\finitegraph$ is given by  the Gaussian free
field $\varphi_\treegraph$ on $\treegraph$. We will see two instances
here: first we show in Section \ref{subsectionapproximation} that one can
locally approximate $\Psi_\finitegraph$ around vertices of $\finitegraph$
with a tree-like neighbourhood (Theorem~\ref{extendedtheorem2.2}). This
will be the type of approximation of $\Psi_\finitegraph$ by
$\varphi_\treegraph$ needed to deal with the supercritical phase in
Section \ref{supercriticalsection} and to prove
Theorem~\ref{manymesoscopiccomponents} (corresponding to \eqref{0.7}).
Then in Section~\ref{sectionconditionaldistribution}  we compute
conditional distributions of $\Psi_\finitegraph$ (Lemma
  \ref{conditionaldistribution})  and we derive that in certain
situations they resemble conditional distributions of $\varphi_\treegraph$
(Proposition \ref{535}, see also \eqref{837}). This will be the crucial
ingredient for approximating $\Psi_\finitegraph$ by $\varphi_\treegraph$
along the connected components of subcritical level sets and ultimately
proving Theorem \ref{microscopiccomponents} (corresponding to
  \eqref{0.6}) in Section \ref{subcriticalsection}.



\subsection{A local approximation
  \texorpdfstring{of $\Psi_\finitegraph$ by $\varphi_\treegraph$}{}
  on tree-like neighbourhoods}
\label{subsectionapproximation}

The goal of this section is to prove  Theorem \ref{extendedtheorem2.2}
below, stating the approximation of the zero-average Gaussian free field
$\Psi_\finitegraph$ on neighbourhoods of vertices with tree-like
surroundings by the Gaussian free field $\varphi_\treegraph$ on
$\treegraph$. This supports the intuition that the local picture of
$\Psi_\finitegraph$ on $\finitegraph$ is given by $\varphi_\treegraph$ on
$\treegraph$. The approximation derived here will be used in Section
\ref{supercriticalsection} to prove the main result \eqref{0.7},
i.e.~that a linear fraction of the vertices of $\finitegraph$ is
contained in mesoscopic connected components of the level set above level
$h$ if $h<h_\star$. Theorem \ref{extendedtheorem2.2} will allow us to
reduce the required computations on $\Psi_\finitegraph$ to computations
on $\varphi_\treegraph$.

For the remainder of  Section \ref{subsectionapproximation} we introduce
some notation. If $n\geq 1$, $x\in \finitegraph$ and $R\geq 1$ with
$\tx(B_\finitegraph(x,R))=0$, then (see \eqref{20}) let
$\rho_{x,R}: B_\finitegraph(x,R) \to B_\treegraph(\rot,R)$ denote
the graph isomorphism given by
$(\pi_{n,x}\big|_{B_\treegraph(\rot,R)})^{-1}$. Furthermore, for
all $n\geq 1$ and pairs $x,x'\in \finitegraph$ we fix
$z_{x,x'} \in \pi_{n,x}^{-1}(\{x'\}) \subseteq \treegraph$. Finally, if
$n\geq 1$, $x,x' \in \finitegraph$, $R\geq 1$ with
$\tx(B_\finitegraph(x,R)) = 0$, $\tx(B_\finitegraph(x',R)) = 0$ and
$B_\finitegraph(x,R) \cap B_\finitegraph(x',R) = \emptyset$, then  let
$\rho_{x,x',R}: B_\finitegraph(x,R)\cup B_\finitegraph(x',R)
\to B_\treegraph(\rot,R)\cup B_\treegraph(z_{x,x'},R)$
denote the graph isomorphism given by
$(\pi_{n,x}\big|_{B_\treegraph(\rot,R)\cup B_\treegraph(z_{x,x'},R)})^{-1}$.
Finally, recall the constant $c_0$ from \eqref{32}. The main result of
this section is the following

\begin{theorem}
  \label{extendedtheorem2.2}
  For all $n$ large enough, $x,x' \in \finitegraph$, $1\leq r < R \leq
  \frac{c_0}{6} \log_{d-1}(N_n)$ such that
  $\tx(B_\finitegraph(x,2R))=0$, $\tx(B_\finitegraph(x',2R)) =0$
  and $B_\finitegraph(x,2R) \cap B_\finitegraph(x',2R) = \emptyset$,
  there exists a coupling $\mathbb Q_n$ of $\Psi_\finitegraph$ and
  $\varphi_\treegraph$ such that for all $\varepsilon >0$
  \begin{equation}
    \label{226}
    \begin{split}
      &\mathbb Q_n \bigg[ \sup_{y \in B_\finitegraph(x,r) \cup B_\finitegraph(x',r)}
        \big| \Psi_\finitegraph(y) - \varphi_\treegraph(\rho_{x,x',2R}(y))  \big|  >
        \varepsilon \bigg] \\
      &\qquad \qquad \qquad \qquad  \leq  8 d(d-1)^{r}\exp
      \Big(-\frac{\varepsilon^2(d-1)(d-2)}{24 d^2}(d-1)^{R-2r} \Big).
    \end{split}
  \end{equation}
  In particular, for all $n$ large enough, $x\in \finitegraph$,
  $1\leq r < R \leq \frac{c_0}{6} \log_{d-1}(N_n)$ such that
  $\tx(B_\finitegraph(x,2R))=0$, there exists a coupling $\mathbb Q_n$ of
  $\Psi_\finitegraph$ and $\varphi_\treegraph$ such that for all
  $\varepsilon >0$ the same bound as in \eqref{226} applies to
  $\mathbb Q_n \big[ \sup_{y \in B_\finitegraph(x,r)} \big|
    \Psi_\finitegraph(y) - \varphi_\treegraph(\rho_{x,2R}(y))  \big|
    > \varepsilon \big]$.
\end{theorem}

We now proceed with some preparations for the proof of Theorem
\ref{extendedtheorem2.2}. The first goal is an easy preliminary coupling
of $\Psi_\finitegraph$ and $\varphi_\treegraph$ around vertices of
$\finitegraph$ with tree-like neighbourhood (Lemma \ref{lemma1.2}). In
its proof we  use the following observation.

\begin{remark}
  \label{extraremark4}
  Let $x,x'\in \finitegraph$ and $R\geq 1$ satisfy
  $\tx(B_\finitegraph(x,R)) = 0$, $\tx(B_\finitegraph(x',R)) = 0$ and
  $B_\finitegraph(x,R) \cap B_\finitegraph(x',R) = \emptyset$. Assume
  $U  \subseteq B_\finitegraph(x,R-1) \cup B_\finitegraph(x',R-1)$,  so
  that
  $\partial_\finitegraph U  \subseteq B_\finitegraph(x,R) \cup B_\finitegraph(x',R)$.
  Then for any
  $y \in B_\finitegraph(x,R) \cup B_\finitegraph(x',R)\subseteq \finitegraph$
  the image under $\pi_{n,x}$ of the law of the simple random walk on
  $\treegraph$ started at
  $\rho_{x,x',R}(y) \in B_\treegraph(\rot,R)
  \cup B_\treegraph(z_{x,x'},R) \subseteq \treegraph$
  and stopped when exiting $\rho_{x,x',R}(U)$ is the same as the law of
  the simple random walk on $\finitegraph$ started at $y$ and stopped
  when exiting $U$. In particular, the hitting distribution of the
  boundary $\partial_\finitegraph U$ of the walk on $\finitegraph$ is the
  image under $\pi_{n,x}$ of the hitting distribution of
  $\partial_\treegraph \rho_{x,x',R}(U)$ of the walk on $\treegraph$,
  that is
  \begin{equation}
    \label{extraequation21}
    \begin{split}
      P_y^\finitegraph[X_{T_U} = z] =
      P_{\rho_{x,x',R}(y)}^\treegraph[X_{T_{\rho_{x,x',R}(U)}} = \rho_{x,x',R}(z)]
      \quad \text{for all } y \in U \text{ and } z \in \partial_\finitegraph U.
    \end{split}
  \end{equation}
  Similarly, for any $x\in \finitegraph$ with
  $\tx(B_\finitegraph(x,R)) = 0$, $U\subseteq B_\finitegraph(x,R-1)$, and
  $y \in B_\finitegraph(x,R)\subseteq \finitegraph$ the image under
  $\pi_{n,x}$ of the law of the simple random walk on $\treegraph$
  started at
  $\rho_{x,R}(y) \in B_\treegraph(\rot,R) \subseteq \treegraph$ and
  stopped when exiting $\rho_{x,R}(U)$ is the same as the  law of  the
  simple random walk on $\finitegraph$ started at $y$ and stopped when
  exiting $U$. So \eqref{extraequation21} holds for $\rho_{x,x',R}$
  replaced by $\rho_{x,R}$. \qed
\end{remark}

As a direct implication of the above Remark \ref{extraremark4} we obtain
a straightforward way to couple  $\Psi_\finitegraph$ on
$B_\finitegraph(x,R) \cup B_\finitegraph(x',R)$ with $\varphi_\treegraph$
on  $B_\treegraph(\rot,R) \cup B_\treegraph(z_{x,x'},R)$.

\begin{lemma}
  \label{lemma1.2}
  Assume $x,x' \in \finitegraph$ with $\tx (B_\finitegraph(x,R))=0$ and
  $\tx (B_\finitegraph(x',R))=0$ satisfy
  $B_\finitegraph(x,R) \cap B_\finitegraph(x',R) = \emptyset$ for some
  $R \geq 1$. Let
  $U \subseteq B_\finitegraph(x,R-1) \cup B_\finitegraph(x',R-1)$. Then
  there exists a coupling of $\Psi_\finitegraph$ and $\varphi_\treegraph$
  such that
  \begin{equation}
    \label{1.20}
    \begin{split}
      \Psi_\finitegraph(y) - E_y^\finitegraph[ \Psi_\finitegraph(X_{T_{U}}) ]  =
      \varphi_\treegraph(\rho_{x,x',R}(y)) &- E_{\rho_{x,x',R}(y)}^\treegraph[
        \varphi_\treegraph(X_{T_{\rho_{x,x',R}(U)}}) ] \\
      & \text{for all } y \in B_\finitegraph(x,R) \cup B_\finitegraph(x',R).
    \end{split}
  \end{equation}
  Similarly, if we only have  $x \in \finitegraph$ with
  $\tx(B_\finitegraph(x,R))=0$ for some $R \geq 1$ and
  $U \subseteq B_\finitegraph(x,R-1)$, then \eqref{1.20}  holds for all
  $y \in B_\finitegraph(x,R)$ with $\rho_{x,x',R}$ replaced by $\rho_{x,R}$.
\end{lemma}
\begin{proof}
  The proof is analogous to the proof of Lemma 1.10 in \cite{36}. Since
  both sides of \eqref{1.20} describe centred Gaussian fields, it is
  enough to check that the covariance is the same. By  \eqref{460}
  resp.~by \eqref{1.18} the covariance of the field for
  $y,z \in B_\finitegraph(x,R) \cup B_\finitegraph(x',R)$  is
  $g_\finitegraph^U(y,z)$ on the left
  resp.~$g_\treegraph^{\rho_{x,x',R}(U)}(\rho_{x,x',R}(y),\rho_{x,x',R}(z))$ on
  the right hand side. These two covariances are equal by Remark
  \ref{extraremark4} and hence the proof is complete.
\end{proof}

We can now lay out the strategy for proving Theorem
\ref{extendedtheorem2.2}. The idea is to combine the coupling of
$\Psi_\finitegraph$ and $\varphi_\treegraph$ from Lemma \ref{lemma1.2}
(for some suitable choice of $U$) with uniform bounds  on the variance of
the expectations appearing in \eqref{1.20}. These uniform bounds are
shown in Proposition \ref{proposition2.1} and will ultimately  lead to
the proof of Theorem \ref{extendedtheorem2.2}. Before that,  we show a
simple estimate of the hitting distribution of a sphere by the simple
random walk on $\treegraph$ (Lemma \ref{25}). This estimate  is needed
for the proof of the bounds in Proposition \ref{proposition2.1}.
\begin{lemma}
  \label{25}
  Let $R\geq 0$. Then for all $y \in B_\treegraph(\rot,R)$ and $z \in
  S_\treegraph(\rot,R)$ one has
  \begin{equation}
    \label{3}
    P_y^\treegraph[X_{H_{S_\treegraph(\rot,R)}}=z]  \leq \Big( \frac1{d-1}
      \Big)^{R-d_\treegraph(y,\rot)}.
  \end{equation}
\end{lemma}
\begin{proof}
  Note that the statement we need to prove only depends on the distance
  of the vertex $y$ to the centre of $B_\treegraph(\rot,R)$. We
  denote by $\rot=y_0,y_1,\ldots,y_R$ a fixed non-backtracking path
  from $\rot$ to $S_\treegraph(\rot,R) \eqqcolon S$, so that
  $d_\treegraph(y_k,\rot)=k$ for $k=0,\ldots,R$. First, we argue
  that
  $P_{y_k}^\treegraph[X_{H_S}=z] \leq P_{y_k}^\treegraph[X_{H_S}=y_R]$
  for all $z\in S$ and $k=0,\ldots,R$. Indeed, fix $z \in S$ and
  $k \in \{0,\ldots,R\}$ and let
  \begin{equation*}
    i_0 \coloneqq \max\{i \in \{0,\ldots,R\} \, | \, y_i \text{ is on the
        non-backtracking path from $\rot$ to $z$} \},
  \end{equation*}
  so that $y_{i_0}$ is the last common vertex of the two non-backtracking
  paths from $\rot$ to $z$ resp.~to $y_R$. Note that any path from
  $y_k$ to $z$ in $\treegraph$ has to pass through $y_{i_0}$ and also
  that
  $P_{y_{i_0}}^\treegraph[X_{H_S}=z]= P_{y_{i_0}}^\treegraph[X_{H_S}=y_R]$
  because $z,y_R \in S$ and
  $d_\treegraph(y_{i_0},z) = d_\treegraph(y_{i_0},y_R)$ by definition of
  $y_{i_0}$. As claimed, one obtains
  \begin{equation*}
    \begin{split}
      &P_{y_k}^\treegraph[X_{H_S}=z] = P_{y_k}^\treegraph[X_{H_S}=z, H_{y_{i_0}} \leq
        H_S] \overset{(*)}{=} P_{y_k}^\treegraph[H_{y_{i_0}} \leq H_S]
      P_{y_{i_0}}^\treegraph[X_{H_S}=z]  \\
      &= P_{y_k}^\treegraph[H_{y_{i_0}} \leq H_S] P_{y_{i_0}}^\treegraph[X_{H_S}=y_R]
      \overset{(*)}{=} P_{y_k}^\treegraph[X_{H_S}=y_R, H_{y_{i_0}} \leq H_S] \leq
      P_{y_k}^\treegraph[X_{H_S}=y_R],
    \end{split}
  \end{equation*}
  where in both $(*)$ we use the strong Markov property.

  It remains to show
  $P_{y_k}^\treegraph[X_{H_S}=y_R]  \leq (d-1)^{-(R-k)}$ for
  $k=0,\ldots,R$. To this end, let
  $A_k \coloneqq S_\treegraph(\rot, R) \cap U_{y_k}$ (see
    \eqref{830}). By definition we have $y_R \in A_k$ and
  $|A_k| = (d-1)^{R-k}$. Moreover, by  symmetry it holds
  $P_{y_k}^\treegraph[X_{H_S}=z] = P_{y_k}^\treegraph[X_{H_S}=y_R]$ for
  all $z\in A_k$. Hence
  \begin{equation*}
    1\ge P_{y_k}^\treegraph[X_{H_S} \in A_k]
    = \sum_{z\in A_k}P_{y_k}^\treegraph[X_{H_S}=z] =
    (d-1)^{R-k}P_{y_k}^\treegraph[X_{H_S}=y_R],
  \end{equation*}
  from which the required claim follows directly.
\end{proof}

\begin{proposition}
  \label{proposition2.1}
  For all $R\geq 1$ and $y \in B_\treegraph(\rot,R)$ one has
  \begin{equation}
    \label{2.1}
    \Var_{\mathbb P^\treegraph}  \Big(
      E_y^\treegraph[\varphi_\treegraph(X_{H_{S_\treegraph(\rot,R)}})] \Big)
    \leq \frac{d^2}{(d-1)(d-2)} \Big( \frac1{d-1} \Big)^{R-2
      d_\treegraph(y,\rot)}.
  \end{equation}
  Also, for all $n$ large enough, $x\in \finitegraph$ and
  $1 \leq R \leq \frac{c_0}{6} \log_{d-1}(N_n)$ with
  $\tx(B_\finitegraph(x,2R))=0$ and $y \in B_\finitegraph(x,R)$ one has
  \begin{equation}
    \label{2.2}
    \begin{split}
      \Var_{\mathbb P^\finitegraph}  \Big(
        E_y^\finitegraph[\Psi_\finitegraph(X_{H_{S_\finitegraph(x,R)}})]  \Big) \leq
      &\frac{3d^2}{(d-1)(d-2)} \Big( \frac1{d-1} \Big)^{R-2 d_\finitegraph(y,x)}.
    \end{split}
  \end{equation}
\end{proposition}
\begin{proof}
  We start with \eqref{2.1}. Let us abbreviate
  $S \coloneqq S_\treegraph(\rot,R)$. We first expand the variance
  to obtain
  \begin{equation}
    \label{2.3}
    \begin{split}
      \Var_{\mathbb P^\treegraph}  \Big(
        E_y^\treegraph[\varphi_\treegraph(X_{H_S})]  \Big)
      &\overset{\phantom{\eqref{3}}}{=}  \sum_{z_1,z_2 \in S}
      P_y^\treegraph[X_{H_S}=z_1] P_y^\treegraph[X_{H_S}=z_2] g_\treegraph(z_1,z_2)
      \\
      &\stackrel[\eqref{1.1}]{\eqref{3}}{\leq} \frac{d-1}{d-2} \Big( \frac1{d-1}
        \Big)^{2(R-d_\treegraph(y,\rot))}  \sum_{z_1,z_2 \in S}  \Big(
        \frac1{d-1} \Big)^{d_\treegraph(z_1,z_2)} .
    \end{split}
  \end{equation}
  Fix $z_1 \in S$. Note that all vertices of $S$ are at even distance
  from $z_1$ and more precisely that in $S$
  \begin{equation}
    \label{6}
    \!\!\!\!
    \begin{cases}
      \text{there is one vertex  at distance $0$ from $z_1$ (namely $z_1$ itself),} \\
      \text{there are $(d-2)(d-1)^{j-1}$ vertices at distance $2j$ from $z_1$
        for $1 \leq   j \leq  R-1$,}                                      \\
      \text{there are $(d-1)^R$ vertices at distance $2R$ from $z_1$. }
    \end{cases}
  \end{equation}
  This implies that for fixed $z_1 \in S$ it holds
  \begin{align}
    \sum_{z_2 \in S} \Big( \frac1{d-1} \Big)^{d_\treegraph(z_1,z_2)} & =  1 +
    \sum_{j=1}^{R-1} (d-2)(d-1)^{j-1}   \Big( \frac1{d-1} \Big)^{2j}  +    (d-1)^R
    \Big( \frac1{d-1} \Big)^{2R}   \nonumber\\
    & =  1 + (d-2)\sum_{j=1}^{R-1}  \Big( \frac1{d-1} \Big)^{j+1} + \Big(
      \frac1{d-1} \Big)^R
    \label{2.4}\\
    & =1 + \Big( \frac1{d-1} \Big) \Big(1-\Big( \frac1{d-1} \Big)^{R-1} \Big)    +
    \Big( \frac1{d-1} \Big)^R  = \frac{d}{d-1}.   \nonumber
  \end{align}
  Since $|S| = d(d-1)^{R-1}$, we can combine \eqref{2.3} and \eqref{2.4} to obtain
  \begin{equation*}
    \Var_{\mathbb P^\treegraph}  \Big(
      E_y^\treegraph[\varphi_\treegraph(X_{H_S})]  \Big) \leq
    \frac{d-1}{d-2} \Big( \frac1{d-1} \Big)^{2(R-d_\treegraph(y,\rot))}
    d(d-1)^{R-1}  \frac{d}{d-1},
  \end{equation*}
  which is equal to the right hand side of \eqref{2.1} and concludes the
  proof of the first part.

  For the proof of \eqref{2.2} we proceed similarly. Let us abbreviate
  $S' \coloneqq S_\finitegraph(x,R)$ and note that $P_y^\finitegraph$-almost
  surely $H_{S'} = T_{B_\finitegraph(x,R-1)}$. Since by assumption we
  have $\tx(B_\finitegraph(x,2R))=0$, Remark \ref{extraremark4} implies
  that for every $z \in S'$ we have
  \begin{equation*}
    \begin{split}
      P_y^\finitegraph[X_{H_{S'}} = z]
      &\overset{\phantom{\eqref{extraequation21}}}{=}
      P_{\rho_{x,R}(y)}^\treegraph[X_{H_{\rho_{x,R}(S')}} = \rho_{x,R}(z)] \\
      &\overset{\eqref{3}}{\leq} \Big( \frac1{d-1}
        \Big)^{R-d_\treegraph(\rho_{x,R}(y),\rot)}  = \Big( \frac1{d-1}
        \Big)^{R-d_\finitegraph(y,x)}.
    \end{split}
  \end{equation*}
  Furthermore, for $n$ large enough, the inequality \eqref{30} in
  Proposition \ref{proposition1.1} applies to $G_\finitegraph(z_1,z_2)$
  with $z_1,z_2 \in S'$ since
  $d_\finitegraph(z_1,z_2) \leq 2R \leq \frac{c_0}{3} \log_{d-1}(N_n)$ by
  assumption on $R$. Therefore, by expanding the variance  we obtain
  similarly to \eqref{2.3} the inequality
  \begin{equation}
    \label{2.5}
    \Var_{\mathbb P^\finitegraph}  \Big(
      E_y^\finitegraph[\Psi_\finitegraph(X_{H_{S'}})]  \Big) \leq
    3\frac{d-1}{d-2}\Big( \frac1{d-1} \Big)^{2(R-d_\finitegraph(y,x))}
    \sum_{z_1,z_2 \in S'}  \Big( \frac1{d-1} \Big)^{d_\finitegraph(z_1,z_2)},
  \end{equation}
  assuming $n$ is large enough. We now argue that for fixed $z_1 \in S'$
  the vertices in $S'$ can be again characterised by \eqref{6}. Indeed,
  the assumption $\tx(B_\finitegraph(x,2R))=0$ implies that any shortest
  path from $z_1$ to some $z_2 \in  S'$ necessarily remains in
  $B_\finitegraph(x,R)$ for which $\tx(B_\finitegraph(x,R))=0$ holds.
  Therefore, $d_\finitegraph(z_1,z_2)$ can be computed by only
  considering the shortest connection in $B_\finitegraph(x,R)$ between
  $z_1$ and $z_2$ and so we are in the tree-like situation of \eqref{6}.
  Thus, the same computation as in \eqref{2.4} leads to
  $\sum_{z_2 \in S'}  \big( \frac1{d-1} \big)^{d_\finitegraph(z_1,z_2)}
  = \frac{d}{d-1}$.
  This combined with  \eqref{2.5} concludes the proof of \eqref{2.2}
  since $|S'| = d(d-1)^{R-1}$ as $\tx(B_\finitegraph(x,R))=0$.
\end{proof}

We  now have all the ingredients for the proof of Theorem
\ref{extendedtheorem2.2}, by which we conclude Section
\ref{subsectionapproximation}.

\begin{proof}[Proof of Theorem \ref{extendedtheorem2.2}]
  Let us abbreviate
  $V\coloneqq B_\finitegraph(x,r) \cup B_\finitegraph(x',r)$. Under the
  assumptions of the theorem we can apply Lemma \ref{lemma1.2} with
  $U \coloneqq B_\finitegraph(x,R-1) \cup B_\finitegraph(x',R-1) \supseteq V$.
  Thus we obtain a coupling $\mathbb Q_n$  of $\Psi_\finitegraph$ and
  $\varphi_\treegraph$ such that for all $\varepsilon>0$
  \begin{align}
    & \mathbb Q_n \bigg[ \sup_{y \in V} \big| \Psi_\finitegraph(y) -
      \varphi_\treegraph(\rho_{x,x',R}(y))  \big|  > \varepsilon \bigg]
    \label{227}                                                       \\
    & \leq  \mathbb Q_n \bigg[ \sup_{y \in V}  \Big|
      E_{\rho_{x,x',R}(y)}^\treegraph[ \varphi_\treegraph(X_{T_{\rho_{x,x',R}(U)}})
        ]\Big| > \frac{\varepsilon}2 \bigg]  +  \mathbb Q_n \bigg[ \sup_{y \in V}
      \Big| E_y^\finitegraph[ \Psi_\finitegraph(X_{T_U})] \Big| > \frac{\varepsilon}2
      \bigg] ,  \nonumber
  \end{align}
  where $\rho_{x,x',R}(U)  = B_\treegraph(\rot,R-1) \cup
  B_\treegraph(z_{x,x'},R-1) \subseteq B_\treegraph(\rot,2R) \cup
  B_\treegraph(z_{x,x'},2R)$.
  We now consider the two terms on the right hand side of \eqref{227}
  separately. For the first term a union bound leads to, abbreviating
  $S\coloneqq S_\treegraph(\rot,R)$ and
  $S' \coloneqq S_\treegraph(z_{x,x'},R)$,
  \begin{equation}
    \label{228}
    \begin{split}
      &\mathbb Q_n \bigg[ \sup_{y \in V}  \Big| E_{\rho_{x,x',R}(y)}^\treegraph[
          \varphi_\treegraph(X_{T_{\rho_{x,x',R}(U)}}) ]\Big| > \frac{\varepsilon}2
        \bigg] \\
      & = \mathbb P^\treegraph \bigg[ \sup_{y \in B_\treegraph(\rot,r) \cup
          B_\treegraph(z_{x,x'},r)}  \Big| E_y^\treegraph[ \varphi_\treegraph(X_{H_{S
                \cup S' }}) ]\Big| > \frac{\varepsilon}2 \bigg] \\
      & \leq  \sum_{y \in B_\treegraph(\rot,r)}  \!\!\!\!  \mathbb
      P^\treegraph \bigg[   \Big| E_y^\treegraph[ \varphi_\treegraph(X_{H_S}) ]\Big|
        > \frac{\varepsilon}2 \bigg] + \!\!\!\!  \sum_{y \in B_\treegraph(z_{x,x'},r)}
      \!\!\!\!  \mathbb P^\treegraph \bigg[   \Big| E_y^\treegraph[
          \varphi_\treegraph(X_{H_{S'}}) ]\Big| > \frac{\varepsilon}2 \bigg]   \\
      &  =   2  \!\!\!\!  \sum_{y \in B_\treegraph(\rot,r)}  \!\!\!\!  \mathbb
      P^\treegraph \bigg[   \Big| E_y^\treegraph[ \varphi_\treegraph(X_{H_S}) ]\Big|
        > \frac{\varepsilon}2 \bigg],
    \end{split}
  \end{equation}
  where the last equality follows by symmetry. Now for each
  $y \in B_\treegraph(\rot,r)$ the expectation appearing inside the
  probability on the right hand side of \eqref{228} is a centred Gaussian
  variable with respect to $\mathbb P^\treegraph$. Thus the exponential
  Markov inequality implies that
  \begin{align}
    2  \!\!\!\! & \sum_{y \in B_\treegraph(\rot,r)}  \!\!\!\!  \mathbb
    P^\treegraph \bigg[   \Big| E_y^\treegraph[ \varphi_\treegraph(X_{H_S}) ]\Big|
      > \frac{\varepsilon}2 \bigg] \leq 4 \!\!\!\! \sum_{y \in
      B_\treegraph(\rot,r)}   \!\!\!\! \exp\bigg(-\frac{(\varepsilon/2)^2}{2
        \Var_{\mathbb P^\treegraph} \big(
          E_y^\treegraph[\varphi_\treegraph(X_{H_S})]\big)} \bigg)  \nonumber                   \\
    & \qquad \overset{\eqref{2.1}}{\leq} 4 \big| B_\treegraph(\rot,r) \big|
    \exp\bigg(-\frac{\varepsilon^2 (d-1)(d-2)}{8 d^2}(d-1)^{R-2r} \bigg).
  \end{align}
  For the second term on the right hand side of \eqref{227} we similarly
  have by a union bound that, abbreviating
  $\overline{S}\coloneqq S_\finitegraph(x,R)$ and
  $\overline{S}' \coloneqq S_\finitegraph(x',R)$,
  \begin{equation}
    \label{230}
    \begin{split}
      &\mathbb Q_n \bigg[ \sup_{y \in V} \Big| E_y^\finitegraph[
          \Psi_\finitegraph(X_{T_U})] \Big| > \frac{\varepsilon}2 \bigg] = \mathbb
      P^\finitegraph \bigg[ \sup_{y \in V} \Big| E_y^\finitegraph[
          \Psi_\finitegraph(X_{H_{\overline{S} \cup \overline{S}'}})] \Big| >
        \frac{\varepsilon}2 \bigg]  \\
      & \quad \leq   \!\!\!\!  \sum_{y \in B_\finitegraph(x,r)}  \!\!\!\!  \mathbb
      P^\finitegraph \bigg[ \Big| E_y^\finitegraph[
          \Psi_\finitegraph(X_{H_{\overline{S}}}) ]\Big| > \frac{\varepsilon}2 \bigg] +
      \!\!\!\!  \sum_{y \in B_\finitegraph(x',r)}   \!\!\!\! \mathbb P^\finitegraph
      \bigg[  \Big| E_y^\finitegraph[ \Psi_\finitegraph(X_{H_{\overline{S}'}})
          ]\Big| > \frac{\varepsilon}2 \bigg].
    \end{split}
  \end{equation}
  The expectations appearing inside the probabilities on the right hand
  side of \eqref{230} are centred Gaussian variables with respect to
  $\mathbb P^\finitegraph$. By \eqref{2.2} their variance can be bounded
  by $\frac{3d^2}{(d-1)(d-2)} ( \frac1{d-1})^{R-2r}$. Hence the
  exponential Markov inequality implies  that
  \begin{equation}
    \label{231}
    \begin{split}
      &\sum_{y \in B_\finitegraph(x,r)}  \!\!\!\!  \mathbb P^\finitegraph \bigg[
        \Big| E_y^\finitegraph[ \Psi_\finitegraph(X_{H_{\overline{S}}}) ]\Big| >
        \frac{\varepsilon}2 \bigg] +  \!\!\!\! \sum_{y \in B_\finitegraph(x',r)}
      \!\!\!\!  \mathbb P^\finitegraph \bigg[   \Big| E_y^\finitegraph[
          \Psi_\finitegraph(X_{H_{\overline{S}'}}) ]\Big| > \frac{\varepsilon}2 \bigg] \\
      &\quad \leq 2 \Big( \big| B_\finitegraph(x,r) \big| + \big|
        B_\finitegraph(x',r) \big| \Big) \exp\bigg(-\frac{\varepsilon^2 (d-1)(d-2)}{24
          d^2}(d-1)^{R-2r} \bigg)  .
    \end{split}
  \end{equation}
  The combination of \eqref{227}--\eqref{231} concludes the proof  of
  Theorem \ref{extendedtheorem2.2} since
  $|  B_\finitegraph(x,r) | =  | B_\finitegraph(x',r) |
  =   | B_\treegraph(\rot,r) |
  =\frac{d(d-1)^{r}-2}{d-2} \leq d(d-1)^{r}$
  as $\tx(B_\finitegraph(x,r))= \tx(B_\finitegraph(x',r))=0$ by assumption.
\end{proof}



\subsection{Conditional distribution of the zero-average Gaussian free field}
\label{sectionconditionaldistribution}

In this section we investigate the conditional distributions of the
zero-average Gaussian free field. Their detailed understanding  will be
needed in Section~\ref{subcriticalsection} to control the behaviour of
the exploration process used in the proof of the main subcritical
result~\eqref{0.6}. We start with the exact computation of the
conditional distribution of $\Psi_\finitegraph(x)$ for $x\in \finitegraph$
given $\Psi_\finitegraph$ on some $A\subsetneq \finitegraph$ (Lemma
  \ref{conditionaldistribution}). We  then see that, under certain
geometric conditions on $x$ and $A$ (see \eqref{485}--\eqref{473}),   the
conditional distribution of $\Psi_\finitegraph(x)$ given
$\Psi_\finitegraph$ on $A\subsetneq \finitegraph$ shows strong
similarities with the conditional distribution of the Gaussian free field
$\varphi_\treegraph$ on $\treegraph$ (Proposition \ref{535}, see also
  \eqref{837}). This feature reflects the general philosophy that the
local picture of $\Psi_\finitegraph$ on $\finitegraph$ is given by
$\varphi_\treegraph$ on $\treegraph$.

\begin{lemma}
  \label{conditionaldistribution}
  Let $A \subsetneq \finitegraph$ non-empty and $x\in \finitegraph$. Then
  $\mathbb P^\finitegraph$-almost surely
  \begin{equation}
    \label{449}
    \mathbb E^\finitegraph  \big[\Psi_\finitegraph(x)  \big|
      \sigma(\Psi_\finitegraph(y), y\in A)  \big]   = E_x^\finitegraph[
      \Psi_\finitegraph(X_{H_A}) ]
    - \frac{E_x^\finitegraph[H_A]}{E_\pi^\finitegraph[H_A]
    }E_\pi^\finitegraph[\Psi_\finitegraph(X_{H_A})]
  \end{equation}
  and
  \begin{equation}
    \label{456}
    \begin{split}
      &\Var_{\mathbb P^\finitegraph} \big(\Psi_\finitegraph(x)  \big|
        \sigma(\Psi_\finitegraph(y), y\in A)  \big)    \\
      &\qquad \qquad \qquad= G_\finitegraph(x,x) -
      E_x^\finitegraph[G_\finitegraph(X_{H_A},x)] +
      \frac{E_x^\finitegraph[H_A]}{E_\pi^\finitegraph[H_A]
      }E_\pi^\finitegraph[G_\finitegraph(X_{H_A},x)].
    \end{split}
  \end{equation}
  Here $E_\pi^\finitegraph$ is the expectation with respect to
  $\frac1{N_n} \sum_{z \in \finitegraph}P_z^\finitegraph$, i.e.~the
  canonical law of simple random walk on $\finitegraph$ starting at a
  uniformly chosen vertex.
\end{lemma}

\begin{proof}
  We will abbreviate
  $U\coloneqq \finitegraph \setminus A \subsetneq \finitegraph$. In
  particular $T_U=H_A$. Note that by \eqref{450} one can write
  $\Psi_\finitegraph(x) = \varphi^U_\finitegraph(x) + E_x^\finitegraph [\Psi_\finitegraph(X_{H_A})]$,
  the second term actually being $\sigma(\Psi_\finitegraph(y), y\in A)$-measurable.
  Hence
  $\mathbb E^\finitegraph  \big[\Psi_\finitegraph(x)  \big|
    \sigma(\Psi_\finitegraph(y), y\in A)  \big]
  = E_x^\finitegraph[ \Psi_\finitegraph(X_{H_A}) ]
  + \mathbb E^\finitegraph \big[\varphi^U_\finitegraph(x) \big|
    \sigma(\Psi_\finitegraph(y), y\in A) \big]$
  and moreover  also
  $\Var_{\mathbb P^\finitegraph} \big(\Psi_\finitegraph(x)
    \big|  \sigma(\Psi_\finitegraph(y), y\in A) \big)
  = \Var_{\mathbb P^\finitegraph}  \big(\varphi^U_\finitegraph(x)
    \big| \sigma(\Psi_\finitegraph(y), y\in A)  \big)$.
  For \eqref{449} it is therefore enough to show that
  $\mathbb P^\finitegraph$-almost surely
  \begin{equation}
    \label{455}
    \mathbb E^\finitegraph \big[\varphi^U_\finitegraph(x) \, \big| \,
      \sigma(\Psi_\finitegraph(y), y\in A)  \big] = -
    \frac{E_x^\finitegraph[T_U]}{E_\pi^\finitegraph[T_U]
    }E_\pi^\finitegraph[\Psi_\finitegraph(X_{T_U})].
  \end{equation}
  On the other hand, for \eqref{456} it is enough to show (use
    \eqref{1.7} to manipulate the first two terms on the right hand side
    of \eqref{456})
  \begin{equation}
    \label{811}
    \Var_{\mathbb P^\finitegraph}  \big(\varphi^U_\finitegraph(x)  \big|
      \sigma(\Psi_\finitegraph(y), y\in A)  \big) = g^U_\finitegraph(x,x) +
    \frac{E_x^\finitegraph[T_U]}{E_\pi^\finitegraph[T_U] }
    \Big(E_\pi^\finitegraph[G_\finitegraph(X_{T_U},x)]  -\frac{E_\pi [T_U]}{N_n}
      \Big).
  \end{equation}
  Let us fix $x_0 \in A$.  We claim that
  \begin{equation}
    \label{453}
    \sigma(\Psi_\finitegraph(y), y\in A) = \sigma \Big( \textstyle \sum_{z\in U}
      \varphi_\finitegraph^U(z), \Psi_\finitegraph(y)-\Psi_\finitegraph(x_0), y\in A
      \Big).
  \end{equation}
  To see \eqref{453} first note that
  $\sigma(\Psi_\finitegraph(y), y\in A)
  = \sigma (\Psi_\finitegraph(x_0), \Psi_\finitegraph(y)-\Psi_\finitegraph(x_0), y\in A)$.
  Moreover, by the zero-average property of $\Psi_\finitegraph$ (see
    below \eqref{1.7}), one  $\mathbb P^\finitegraph$-almost surely has
  \begin{align*}
    \Psi_\finitegraph(x_0)
    &\overset{\phantom{\eqref{450}}}{=} -\frac{1}{N_n}\sum_{z\in
      \finitegraph}(\Psi_\finitegraph(z) - \Psi_\finitegraph(x_0)) \\
    &\overset{\eqref{450}}{=} -\frac{1}{N_n}\sum_{z\in \finitegraph}(\varphi^U_\finitegraph(z) +
      E_z^\finitegraph[\Psi_\finitegraph(X_{T_U})]- \Psi_\finitegraph(x_0) )  \\
    &\overset{\phantom{\eqref{450}}}{=} -\frac{1}{N_n}\sum_{z\in U} \varphi^U_\finitegraph(z)  -\frac{1}{N_n}\sum_{z\in \finitegraph}
    E_z^\finitegraph[\Psi_\finitegraph(X_{T_U}) -\Psi_\finitegraph(x_0)].
  \end{align*}
  The latter sum is
  $\sigma(\Psi_\finitegraph(y)-\Psi_\finitegraph(x_0), y \in A)$-measurable.
  Thus
  $\sigma (\Psi_\finitegraph(x_0), \Psi_\finitegraph(y)-\Psi_\finitegraph(x_0), y\in A )
  = \sigma \big( \textstyle \sum_{z\in U} \varphi_\finitegraph^U(z),
    \Psi_\finitegraph(y)-\Psi_\finitegraph(x_0), y\in A \big)$,
  which shows \eqref{453}.

  Now note that
  \begin{equation}
    \label{457}
    \parbox{11.5cm}{
      for $z\in \finitegraph$ and $y \in A$ the Gaussian random variables
      $\varphi^U_\finitegraph(z)$ and $\Psi_\finitegraph(y)-\Psi_\finitegraph(x_0)$
      are independent.
    }
  \end{equation}
  Indeed,
  $\mathbb E^\finitegraph
  \big[\varphi^U_\finitegraph(z)(\Psi_\finitegraph(y)-\Psi_\finitegraph(x_0))
    \big]
  =G_\finitegraph(z,y) - E_z^\finitegraph[G_\finitegraph(X_{T_U},y)] -
  G_\finitegraph(z,x_0) + E_z^\finitegraph[G_\finitegraph(X_{T_U},x_0)]$
  by \eqref{450} and \eqref{0.1}, which is equal to
  $g^U_\finitegraph(z,y) - g^U_\finitegraph(z,x_0)=0$  by \eqref{1.7} and
  \eqref{1.6} (since $y,x_0 \notin U$).

  Recall that for random variables $U,Y,Z$ such that $U$ is integrable
  and $Z$ is independent of $\sigma(U,Y)$ one has
  $\mathbb E[U | \sigma(Y,Z)] = \mathbb E[U| \sigma(Y)]$ almost surely
  (see e.g.~\cite{41}, 9.7(k)). Hence we get
  $\mathbb E^\finitegraph \big[\varphi^U_\finitegraph(x)
    \big| \sigma(\Psi_\finitegraph(y), y\in A)  \big]
  = \mathbb E^\finitegraph \big[\varphi^U_\finitegraph(x) \big|
    \sigma(\textstyle \sum_{z\in U} \varphi_\finitegraph^U(z) )  \big]$
  $\mathbb P^\finitegraph$-almost surely  by \eqref{453} and \eqref{457}.
  Due to the general formula
  $\Var(X|\sigma(Y)) = \mathbb E[X^2|\sigma(Y)] - \mathbb E[X|\sigma(Y)]^2$,
  the same observation also shows that
  $\Var_{\mathbb P^\finitegraph} \big(\varphi^U_\finitegraph(x) \big|
    \sigma(\Psi_\finitegraph(y), y\in A)  \big)
  = \Var_{\mathbb P^\finitegraph}  \big(\varphi^U_\finitegraph(x)  \big|
    \sigma( \textstyle \sum_{z\in U} \varphi_\finitegraph^U(z) )   \big)$.
  Therefore, the conditional expectation/variance to be considered in
  \eqref{455} and \eqref{811} are actually only with respect to the
  sigma-algebra generated by the single Gaussian random variable
  $\textstyle \sum_{z\in U} \varphi_\finitegraph^U(z)$. So  by the
  formula  for conditional expectation/variance of the bivariate centred
  Gaussian distribution we have
  \begin{equation}
    \label{454}
    \begin{split}
      \mathbb E^\finitegraph \big[\varphi^U_\finitegraph(x) \, \big| \,
        \sigma(\Psi_\finitegraph(y), y\in A)  \big]
      &= \frac{\mathbb E^\finitegraph \big[\varphi^U_\finitegraph(x)   \textstyle
          \sum_{z\in U} \varphi_\finitegraph^U(z)  \big] }{\mathbb E^\finitegraph \big[
          \big(\textstyle\sum_{z\in U} \varphi_\finitegraph^U(z) \big)^2 \big] }
      \sum_{z\in U} \varphi_\finitegraph^U(z),    \\
      \Var_{\mathbb P^\finitegraph}  \big(\varphi^U_\finitegraph(x)  \big|
        \sigma(\Psi_\finitegraph(y), y\in A)  \big)
      &=  \mathbb E^\finitegraph \big[\varphi^U_\finitegraph(x) ^2  \big] -
      \frac{\mathbb E^\finitegraph \big[\varphi^U_\finitegraph(x)    \textstyle
          \sum_{z\in U} \varphi_\finitegraph^U(z)  \big]^2 }{\mathbb E^\finitegraph \big[
          \big(\textstyle\sum_{z\in U} \varphi_\finitegraph^U(z) \big)^2 \big] }.
    \end{split}
  \end{equation}
  We observe that for $u \in \finitegraph$ one has
  $\sum_{z\in U}  \mathbb E^\finitegraph \big[\varphi^U_\finitegraph(u) \varphi_\finitegraph^U(z)  \big]
  =  \sum_{z\in U} g^U_\finitegraph(u,z)  =   E^\finitegraph_u[T_U]$
  by \eqref{460} and \eqref{1.6}. By applying this and \eqref{1.18}
  inside \eqref{454} we obtain
  \begin{equation}
    \label{812}
    \begin{split}
      \mathbb E^\finitegraph \big[\varphi^U_\finitegraph(x) \, \big| \,
        \sigma(\Psi_\finitegraph(y), y\in A)  \big]
      &= \frac{E_x^\finitegraph [T_U] }{\sum_{z\in U}  E_z^\finitegraph [T_U] }
      \sum_{z\in U} \varphi_\finitegraph^U(z),    \\
      \Var_{\mathbb P^\finitegraph}  \big(\varphi^U_\finitegraph(x)  \big|
        \sigma(\Psi_\finitegraph(y), y\in A)  \big)
      &= g^U_\finitegraph(x,x) -  \frac{E_x^\finitegraph [T_U]^2 }{\sum_{z\in U}
        E_z^\finitegraph [T_U] }.
    \end{split}
  \end{equation}
  We are almost done. Observe that by \eqref{460}, \eqref{450} and the
  zero-average property of $\Psi_\finitegraph$ it $\mathbb P^\finitegraph$-almost
  surely  holds
  \begin{align*}
    \sum_{z\in U} \varphi_\finitegraph^U(z) =  \sum_{z\in \finitegraph} \varphi_\finitegraph^U(z)
    =  \sum_{z\in \finitegraph}  \big(\Psi_\finitegraph(z) - E_z^\finitegraph[
        \Psi_\finitegraph(X_{T_U}) ] \big)
    = -\sum_{z\in \finitegraph} E_z^\finitegraph[ \Psi_\finitegraph(X_{T_U}) ].
  \end{align*}
  This combined with \eqref{812} shows \eqref{455}. On the other hand, by
  the formula above \eqref{812}, \eqref{1.7} and the zero-average
  property of $G_\finitegraph(\cdot,\cdot)$  (see below  \eqref{1.7}) one
  has
  \begin{align*}
    -\frac{E_x^\finitegraph[T_U]}{N_n}
    &= -\frac{1}{N_n}\sum_{z\in U} g_\finitegraph^U(z,x)
    =  -\frac{1}{N_n}\sum_{z\in \finitegraph} g_\finitegraph^U(z,x)   \\
    &= \frac{1}{N_n}\sum_{z\in \finitegraph} \big(
      E_z^\finitegraph [G_\finitegraph(X_{T_U}),x] - \frac{1}{N_n}E_z^\finitegraph[T_U]
      -G_\finitegraph(z,x) \big)   \\
    &= E_\pi^\finitegraph[G_\finitegraph(X_{T_U},x)]  -\frac{E_\pi [T_U]}{N_n}.
  \end{align*}
  This combined with \eqref{812} shows \eqref{811} and concludes the
  proof of Lemma \ref{conditionaldistribution}.
\end{proof}

Lemma \ref{conditionaldistribution} above shows that for any
$x\in \finitegraph$,  $\Psi_\finitegraph(x)$ conditionally on
$(\Psi_\finitegraph(y))_{y\in A}$ for $A\subsetneq \finitegraph$
non-empty is a Gaussian random variable with mean and variance given by
the right hand sides of \eqref{449} and \eqref{456}. Comparable (but
  easier) statements for the Gaussian free field $\varphi_\treegraph$ on
$\treegraph$ follow directly from \eqref{1.18}. In particular, if
$x'\in \treegraph$ and $A' \coloneqq \treegraph \setminus U_{x'}$ (recall
  definition \eqref{830}), then by \eqref{2544} and \eqref{2545} one has
\begin{equation}
  \label{837}
  \begin{split}
    \mathbb E^\treegraph \big[\varphi_\treegraph(x')  \big|
      \sigma(\varphi_\treegraph(y), y\in A') \big] &=  \tfrac1{d-1}
    \varphi_\treegraph(\overline x'),   \\
    \Var_{\mathbb P^\treegraph}  \big(\varphi_\treegraph(x')  \big|
      \sigma(\varphi_\treegraph(y), y\in A') \big) &= \tfrac{d}{d-1}.
  \end{split}
\end{equation}
As we will show in Proposition \ref{535} below, a similar behaviour can
be observed for the zero-average Gaussian free field $\Psi_\finitegraph$
on $\finitegraph$, at least in specific situations. We  now introduce the
requirements on $x \in \finitegraph$ and $A\subseteq \finitegraph$.
Define for $A \subsetneq \finitegraph$ non-empty and $r\geq 1$  the set
$B_\finitegraph(A,r) \coloneqq \{z \in \finitegraph \,
  | \,  z \in B_\finitegraph(w,r) \text{ for some } w \in A \}$.
Moreover, for  $x \in \partial_\finitegraph A$ we set
\begin{equation*}
  F_A(x,r) \coloneqq \{z \in B_\finitegraph(A,r)
    \setminus A \, | \,  \text{$z$ is connected to $x$ in $B_\finitegraph(A,r)
      \setminus A$} \}.
\end{equation*}
In particular $x \in F_A(x,r)$. We set
\begin{equation}
  \label{471}
  s_n \coloneqq \max \{1 \, , \,  \lfloor  8 \log_{d-1}(\log_{d-1}(N_n)) \rfloor
    \} \quad \text{for $n\geq 1$}.
\end{equation}
and say that $x \in \partial_\finitegraph A$ is a \emph{good vertex at
  the boundary of $A$} if the following properties hold
\begin{flalign}
  \quad\bullet& \ \, |B_\finitegraph(x,1) \cap A| = 1, \text{ write
    $\overline x \in A$ for the unique vertex in this intersection}  \nonumber
  \\&\ \, \text{(note that for $x'\in \treegraph$ the notation
      $\overline x'$ has been   defined above \eqref{830})}
  \label{485}   \\
  \quad\bullet& \ \, \text{$\tx(F_A(x,s_n)) = 0$}
  \label{474}
  \\
  \quad\bullet& \ \, \text{for all $y \in \partial_\finitegraph A
    \setminus  \{x\}$ every path in $\finitegraph  \setminus  A$ from $y$ to
    $x$ leaves $B_\finitegraph(A,s_n)$.}
  \label{473}
\end{flalign}
Equivalently, $F_A(x,s_n)$ is \emph{proper} in the notation of \cite{21}
(see Figure~\ref{fig:goodvertex} for an illustration of the conditions
  \eqref{485}--\eqref{473}).

\begin{figure}[H]
  \begin{center}
    \includegraphics[width=7cm]{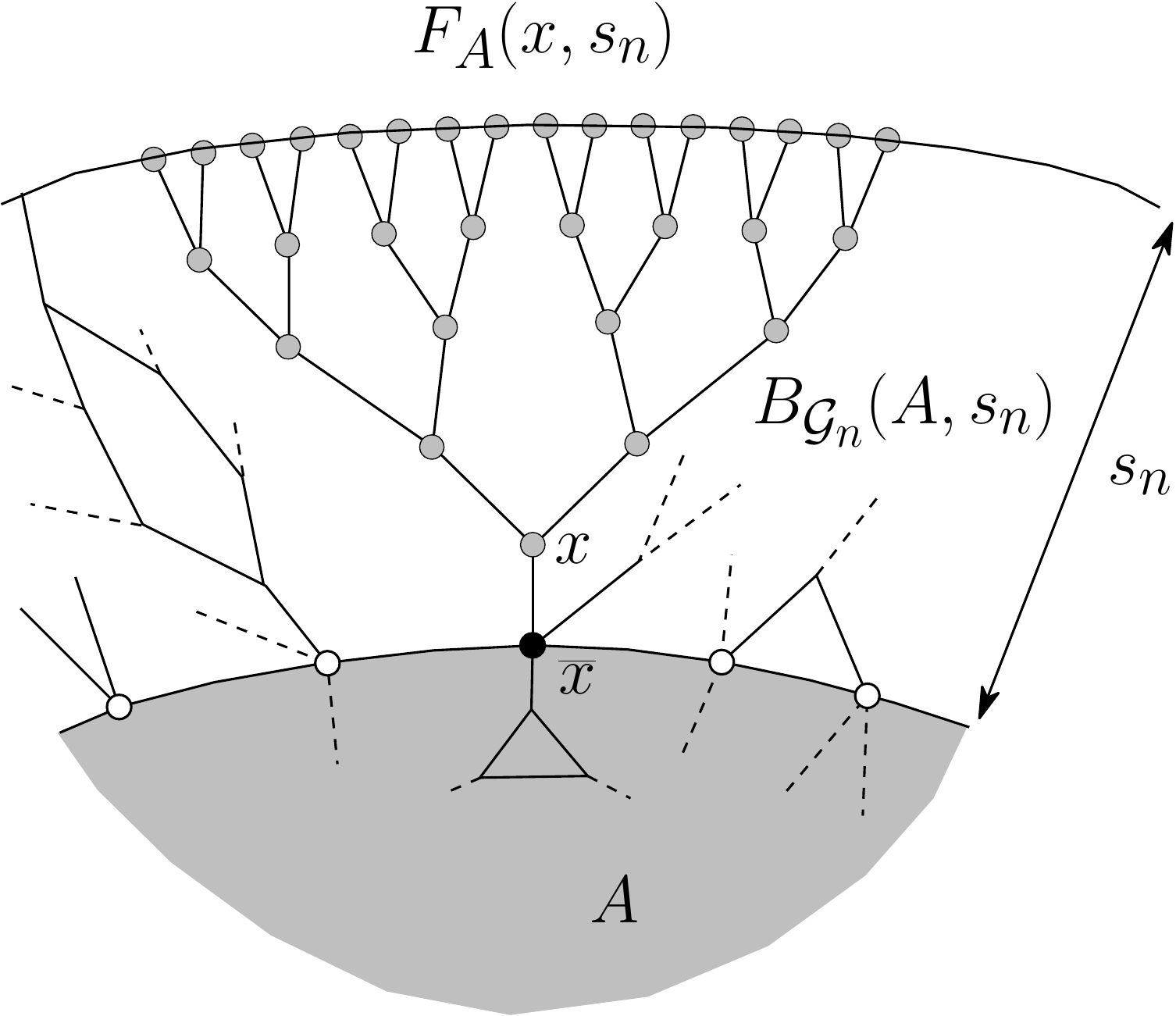}
    \caption{(adapted from \cite{21}) The point
      $x \in \partial_\finitegraph A$ is a good vertex at the boundary of
      $A$, the points in $F_A(x,s_n)$ are marked grey.}
    \label{fig:goodvertex}
  \end{center}
\end{figure}

\noindent For $A\subsetneq \finitegraph$ non-empty we set
\begin{equation}
  \label{850}
  G_A \coloneqq \{ \text{good vertices at the boundary of $A$} \}.
\end{equation}
We are now ready to state Proposition \ref{535}. Observe the analogies
between its statement and \eqref{837}.

\begin{proposition}
  \label{535}
  For every $b,b'>0$ there exists $c_{b,b'}>0$ such that for $n\geq 1$,
  $A \subseteq \finitegraph$ non-empty with $|A| \leq b \ln(N_n)$,
  $x\in G_A$ and on the event
  $\big\{ \sup_{z\in A} | \Psi_\finitegraph(z) | \leq b' \sqrt{\ln(N_n)} \big \}$
  it holds
  \begin{align}
    \Big| \mathbb E^\finitegraph \big[ \Psi_\finitegraph(x) \, \big| \,
      \sigma(\Psi_\finitegraph(y),y\in A) \big ] -\frac{1}{d-1}
    \Psi_\finitegraph(\overline x) \Big|
    & \leq c_{b,b'} (\ln(N_n))^{-2},
    \label{518}      \\
    \Big| \Var_{\mathbb P^\finitegraph} \big( \Psi_\finitegraph(x) \, \big|
      \, \sigma(\Psi_\finitegraph(y),y\in A) \big ) - \frac{d}{d-1} \Big|
    & \leq c_{b,b'} (\ln(N_n))^{-3}.
    \label{547}
  \end{align}
  Recall that $\overline x \in A$ denotes the unique neighbour of $x$ in
  $A$ (see \eqref{485}).
\end{proposition}

To show Proposition \ref{535} and conclude this section we will
manipulate the explicit expressions for the conditional expectation and
variance obtained in Lemma \ref{conditionaldistribution}. In these
expressions one considers the hitting time of $A$ for the simple random
walk on $\finitegraph$ and in the proof of Proposition \ref{535} we will
look at different situations for when the hitting happens (see the
  beginning of the proof of Proposition \ref{535} below). Since
$x \in G_A$ in the statement, the simple random walk started at $x$ has
to leave $F_A(x,s_n)$ to hit $A$ and so  $P_x^\finitegraph$-almost surely
either $H_A = T_{F_A(x,s_n)}$ or $H_A > T_{F_A(x,s_n)}$. We will further
split the latter case into whether $H_A$ happens before or after an
additional time
\begin{equation}
  \label{97}
  t_n \coloneqq \frac{1}{\lambda_\finitegraph} (\ln(N_n))^2,
\end{equation}
by which the distribution of the simple random walk is very close to the
stationary distribution (here the uniform distribution on $\finitegraph$).
This follows from e.g.~\cite{44}, Corollary~2.1.5. So in the proof of
Proposition \ref{535} we will consider the three situations
$H_A = T_{F_A(x,s_n)}$, $T_{F_A(x,s_n)} < H_A <  T_{F_A(x,s_n)} +t_n$ and
$H_A \geq  T_{F_A(x,s_n)} +t_n$ separately. Before that, we collect in
Lemma \ref{technicalSRWlemma} some preliminary observations about the
simple random walk on $\finitegraph$  and subsequently start with the
proof of Proposition \ref{535}. For the rest of this section we will
abbreviate $F_A \coloneqq F_A(x,s_n)$. It is also convenient to consider
the \emph{continuous-time} simple random walk $(\overline{X}_t)_{t\geq 0}$.
We remind that for the exit time from $U\subseteq \finitegraph$
(resp.~for the entrance time in $U\subseteq \finitegraph$) of this walk
we use the same notation $T_U$ (resp.~$H_U$) as for the discrete-time
simple random walk.

\begin{lemma}
  \label{technicalSRWlemma}
  For $n\geq1$, $A \subseteq \finitegraph$ non-empty and $x\in G_A$  one has
  \begin{flalign}
    \textup{(i)} & \ \,   \tfrac{1}{d-1} -  c (\ln(N_n))^{-7} \leq P_x^\finitegraph[ H_A = T_{F_A}] \leq
    \tfrac{1}{d-1}   \quad \text{and} \quad
    E_x^\finitegraph[T_{F_A}]  \leq c\ln(N_n)  &
    \label{0007} \\[6pt]
    \textup{(ii)} & \ \,   P_x^\finitegraph[H_A>T_{F_A}]= \sum_{z\in B_\finitegraph(A,s_n)^\mathsf{c}}
    P_x^\finitegraph[\overline{X}_{T_{F_A}}=z, H_A>T_{F_A}].   &
    \label{96}
  \end{flalign}
  Moreover, for every $b>0$ there exists $c_b>0$ such that for $n\geq1$,
  $A \subseteq \finitegraph$ non-empty with $|A| \leq b \ln(N_n)$ and
  $x\in G_A$  one has
  \begin{flalign}
    \textup{(iii)} & \ \,  P_x^\finitegraph [T_{F_A} < H_A <  T_{F_A} +t_n]  \leq c_b(\ln(N_n))^{-5} &
    \label{0005}  \\[6pt]
    \textup{(iv)} & \ \,   \sum_{w\in \finitegraph}   \big|  P_z^\finitegraph[ \overline{X}_{t_n}=w ,
      H_A \geq t_n]   -  \tfrac{1}{N_n} \big|   \leq c_b(\ln(N_n))^{-5} \quad \text{for } z \in \finitegraph  &
    \label{0006}  \\[6pt]
    \textup{(v)} & \ \,    \Big| \tfrac{E_x^\finitegraph[H_A]}{E_\pi^\finitegraph[H_A] }  -
    P_x^\finitegraph[H_A>T_{F_A}]  \Big|  \leq  c_b (\ln(N_n))^{-3} .
    \label{557}
  \end{flalign}
\end{lemma}

\begin{proof}
  Due to \eqref{474}, the probability $P_x^\finitegraph[H_A = T_{F_A}]$
  is equal to the probability that a (discrete-time) random walk on
  $\mathbb Z$ started at 1   and jumping with probability $\frac{d-1}{d}$
  to the right and $\frac{1}{d}$ to the left hits 0 before hitting $s_n+1$.
  Similarly, $E_x^\finitegraph[T_{F_A}]$  is equal to the expected  time
  until this random walk hits 0 or $s_n+1$. Thus (see e.g.~\cite{45},
    (2.4) and (3.4) in Chapter 14) it holds
  \begin{equation*}
    \begin{split}
      &P_x^\finitegraph[ H_A = T_{F_A}] =1-
      \tfrac{d-2}{d-1}(1-(\tfrac{1}{d-1})^{s_n+1})^{-1} \leq 1 - \tfrac{d-2}{d-1} = \tfrac{1}{d-1},   \\
      &E_x^\finitegraph[T_{F_A}] = \tfrac{d}{d-2}\Big( (s_n+1) \tfrac{d-2}{d-1}
        \frac{1}{1-(\frac{1}{d-1})^{s_n+1}}-1  \Big)
      \leq 2\tfrac{d}{d-1}  (s_n+1) \overset{\eqref{471}}{\leq} c\ln(N_n).
    \end{split}
  \end{equation*}
  Since
  $(1-(\frac{1}{d-1})^{s_n+1})^{-1} \leq
  (1- (\log_{d-1}(N_n))^{-8})^{-1} \leq 1 + c (\ln(N_n))^{-7}$,
  one also has
  $P_x^\finitegraph[ H_A = T_{F_A}] \geq 1-  \tfrac{d-2}{d-1}(1 + c
    (\ln(N_n))^{-7}) \geq \tfrac{1}{d-1} - c (\ln(N_n))^{-7}$.
  Thus \eqref{0007} is shown.

  To see \eqref{96} observe that on the event $\{H_A > T_{F_A} \}$, at
  the moment the simple random walk started at $x$ leaves $F_A$, it is in
  some
  $z\in \partial_\finitegraph F_A \cap B_\finitegraph(A,s_n)^\mathsf{c}$
  (note that indeed $z \notin B_\finitegraph(A,s_n)$ since else there
    would exist a path like those excluded by \eqref{473}). In other words,
  \begin{equation}
    \label{0003}
    P_x^\finitegraph \text{-almost surely } \overline X_{T_{F_A}}
    \in \partial_\finitegraph F_A \cap B_\finitegraph(A,s_n)^\mathsf{c}
    \text{ on the event } \{H_A > T_{F_A} \}.
  \end{equation}
  This shows \eqref{96}. To derive \eqref{0005} we apply the strong
  Markov property of  simple random walk for  time $T_{F_A}$ and obtain
  for $n\geq 1$
  \begin{equation}
    \label{0004}
    P_x^\finitegraph [T_{F_A} < H_A <  T_{F_A} +t_n]
    \overset{\eqref{0003}}{\leq} \sup_{z \in
      B_\finitegraph(A,s_n)^\mathsf{c}}  P_z^\finitegraph[H_A <
      t_n] .
  \end{equation}
  Roughly speaking, the right hand side of \eqref{0004} is small since it
  is difficult for the simple random walk to hit $A$ within time $t_n$
  because it starts at distance larger than $s_n$ from $A$ and the
  environment is nearly treelike (see $\eqref{1}$). More precisely,  we
  can apply \cite{21}, Lemma 3.4 (for $T \coloneqq t_n$, $r\coloneqq 0$,
    $s\coloneqq s_n$ and using \eqref{1}) to find  $c,c'>0$  such that
  for $z \in B_\finitegraph(A,s_n)^\mathsf{c}$  one has for $n\geq 1$
  \begin{equation*}
    P_z^\finitegraph[H_A < t_n]
    \leq \sum_{y \in A} P_z^\finitegraph[H_y < t_n]
    \leq |A| \big( c t_n (d-1)^{-s_n} + e^{-c't_n} \big)
    \overset{\eqref{471}}{\leq} c_b(\ln(N_n))^{-5},
  \end{equation*}
  where the last inequality also uses the assumption on $A$, \eqref{97}
  and \eqref{2}. This combined with \eqref{0004} gives \eqref{0005}.

  For \eqref{0006} the idea is that on the event $\{H_A \geq t_n\}$ the
  simple random walk started at $z$ has, roughly speaking, reached the
  stationary distribution by time $t_n$ without having hit $A$. We
  observe that for $z,w\in \finitegraph$ one has
  \begin{align*}
    \big| P_z^\finitegraph[ \overline{X}_{t_n}=w ,  H_A \geq t_n]     -
    \tfrac{1}{N_n} \big|
    &\overset{\phantom{\eqref{97}}}{\leq}  P_z^\finitegraph[ \overline{X}_{t_n}=w ,  H_A < t_n]
    +    \big|  P_z^\finitegraph[ \overline{X}_{t_n}=w] -  \tfrac{1}{N_n} \big|  \\
    &\overset{\, \ (*) \ \, }{\leq} P_z^\finitegraph[ \overline{X}_{t_n}=w ,  H_A < t_n]+
    \exp(-\lambda_\finitegraph t_n) \\
    &\overset{\eqref{97}}{=} P_z^\finitegraph[ \overline{X}_{t_n}=w ,  H_A
      < t_n]+ \exp(-(\ln(N_n))^2),
  \end{align*}
  where in $(*)$ we apply \cite{44}, Corollary 2.1.5. Hence for $n\geq 1$,
  $z \in \finitegraph$, one has
  \begin{equation*}
    \begin{split}
      &\sum_{w\in \finitegraph}   \big|  P_z^\finitegraph[ \overline{X}_{t_n}=w ,
        H_A \geq t_n]   -  \tfrac{1}{N_n} \big|   \leq P_z^\finitegraph[ H_A < t_n]
      + N_n \exp(-(\ln(N_n))^2),
    \end{split}
  \end{equation*}
  which together with the above estimate on $P_z^\finitegraph[ H_A < t_n]$
  gives \eqref{0006}. It remains to show \eqref{557}. We start by
  computing (using also (3.20) of \cite{21} in the second inequality)
  \begin{equation}
    \label{816}
    \frac{E_x^\finitegraph[H_A \bbone_{H_A=T_{F_A}}]}{E_\pi^\finitegraph[H_A] }
    \leq \frac{E_x^\finitegraph[T_{F_A}]}{E_\pi^\finitegraph[H_A] }
    \overset{\eqref{0007}}{\leq} c\ln(N_n)  \frac{4|A|}{N_n} \leq
    \frac{c_b(\ln(N_n))^2}{N_n}.
  \end{equation}
  Now by \eqref{0003} and the  strong Markov property of simple random
  walk for time $T_{F_A}$ one has
  $E_x^\finitegraph[H_A \bbone_{H_A>T_{F_A}}]
  = \sum_{z\in B_\finitegraph(A,s_n)^\mathsf{c}} P_x^\finitegraph[
    \overline{X}_{T_{F_A}}=z, H_A>T_{F_A}] E_z^\finitegraph[H_A]$.
  This combined with \eqref{96} shows
  \begin{equation}
    \label{815}
    \begin{split}
      &\bigg| \frac{E_x^\finitegraph[H_A
          \bbone_{H_A>T_{F_A}}]}{E_\pi^\finitegraph[H_A] }  -
      P_x^\finitegraph[H_A>T_{F_A}]  \bigg|
      \leq   \sup_{z \in  B_\finitegraph(A,s_n)^\mathsf{c}}  \bigg|
      \frac{E_z^\finitegraph[H_A]}{E_\pi^\finitegraph[H_A] } - 1 \bigg|.
    \end{split}
  \end{equation}
  By \cite{21}, Proposition 3.5, we can bound  the absolute value on the
  right hand side of \eqref{815} by
  $c|A|(d-1)^{-s_n}(\ln(N_n))^4 \leq c_b(\ln(N_n))^{-3}$. Since
  $P_x^\finitegraph$-almost surely either $H_A=T_{F_A}$ or $H_A > T_{F_A}$,
  the combination of \eqref{816} and \eqref{815} concludes the proof.
\end{proof}

\begin{proof}[Proof of Proposition \ref{535}]
  We start with the basic observation that by \eqref{449} one has
  $\big| \mathbb E^\finitegraph \big[ \Psi_\finitegraph(x) \,
    \big| \, \sigma(\Psi_\finitegraph(y),y\in A) \big ]
  -\frac{1}{d-1} \Psi_\finitegraph(\overline x) \big|
  \leq U_{A,x}^\finitegraph + V_{A,x}^\finitegraph +W_{A,x}^\finitegraph $,
  where
  \begin{equation*}
    \begin{split}
      U_{A,x}^\finitegraph &\coloneqq \Big|  E_x^\finitegraph
      \big[\Psi_\finitegraph(\overline{X}_{H_A}) \bbone_{\{ H_A = T_{F_A} \}}
        \big]  - \tfrac{1}{d-1} \Psi_\finitegraph(\overline x) \Big|, \\
      V_{A,x}^\finitegraph &\coloneqq  \big| E_x^\finitegraph
      \big[\Psi_\finitegraph(\overline{X}_{H_A})  \bbone_{\{ T_{F_A} < H_A <
            T_{F_A} +t_n \}}  \big]  \big|, \\
      W_{A,x}^\finitegraph &\coloneqq  \Big| E_x^\finitegraph
      \big[\Psi_\finitegraph(\overline{X}_{H_A})   \bbone_{\{ H_A \geq  T_{F_A}
            +t_n\}}  \big]
      -   \tfrac{E_x^\finitegraph[H_A]}{E_\pi^\finitegraph[H_A] }
      E_\pi^\finitegraph \big [\Psi_\finitegraph(\overline{X}_{H_A})  \big]
      \Big|.
    \end{split}
  \end{equation*}
  Hence the proof of \eqref{518} follows once we show  that
  \begin{equation}
    \label{520}
    \parbox{12cm}{
      there exists $c_{b,b'}>0$ such that  for $n\geq1$, $A \subseteq \finitegraph$
      non-empty with $|A| \leq b \ln(N_n)$, $x\in G_A$ and on the  event
      $\big\{\textstyle \sup_{z\in A} | \Psi_\finitegraph(z) | \leq b'
        \sqrt{\ln(N_n)} \big \}$ one has $U_{A,x}^\finitegraph + V_{A,x}^\finitegraph +
      W_{A,x}^\finitegraph \leq c_{b,b'} (\ln(N_n))^{-2} $.
    }
  \end{equation}
  Similarly we have
  $\big| \Var_{\mathbb P^\finitegraph} \big(
    \Psi_\finitegraph(x) \, \big| \, \sigma(\Psi_\finitegraph(y),y\in A) \big ) -
  \frac{d}{d-1} \big|  \leq   \overline{U}_{A,x}^\finitegraph +
  \overline{V}_{A,x}^\finitegraph +\overline{W}_{A,x}^\finitegraph$
  by \eqref{456}, where
  \begin{equation*}
    \begin{split}
      \overline{U}_{A,x}^\finitegraph &\coloneqq \Big|  G_\finitegraph(x,x)  -
      E_x^\finitegraph \big[G_\finitegraph(X_{H_A},x) \bbone_{\{ H_A = T_{F_A}
            \}} \big]  - \tfrac{d}{d-1}\Big|,   \\
      \overline{V}_{A,x}^\finitegraph &\coloneqq E_x^\finitegraph
      \big[G_\finitegraph(X_{H_A},x) \bbone_{\{ T_{F_A} < H_A < T_{F_A} +t_n
            \}}  \big],  \\
      \overline{W}_{A,x}^\finitegraph &\coloneqq  \Big| E_x^\finitegraph
      \big[G_\finitegraph(X_{H_A},x) \bbone_{\{ H_A \geq  T_{F_A} +t_n\}}  \big]
      -   \tfrac{E_x^\finitegraph[H_A]}{E_\pi^\finitegraph[H_A] }
      E_\pi^\finitegraph[G_\finitegraph(X_{H_A},x)]   \Big|.
    \end{split}
  \end{equation*}
  Thus the proof of \eqref{547} follows once we show that
  \begin{equation}
    \label{549}
    \parbox{12cm}{
      there exists $c_{b}>0$ such that  for $n \geq 1$, $A \subseteq \finitegraph$
      non-empty with $|A| \leq b \ln(N_n)$ and $x\in G_A$ one has
      $\overline{U}_{A,x}^\finitegraph+ \overline{V}_{A,x}^\finitegraph +
      \overline{W}_{A,x}^\finitegraph \leq  c_b (\ln(N_n))^{-3} $.
    }
  \end{equation}

  It remains to show \eqref{520} and \eqref{549}. For \eqref{520} we
  bound the three terms $U_{A,x}^\finitegraph$, $V_{A,x}^\finitegraph$
  and $W_{A,x}^\finitegraph$ separately. On $\{ H_A = T_{F_A} \}$  one
  has $P_x^\finitegraph$-almost surely
  $\Psi_\finitegraph(\overline{X}_{H_A}) = \Psi_\finitegraph(\overline x)$
  due to $x\in G_A$. Therefore we deduce
  $U_{A,x}^\finitegraph
  = |\Psi_\finitegraph(\overline x) | \cdot \big| P_x^\finitegraph[H_A = T_{F_A}]
  - \tfrac{1}{d-1}   \big| \leq b' \sqrt{\ln(N_n)} c (\ln(N_n))^{-7}$
  by \eqref{0007}, where in the last inequality we  also use that
  $\overline{x}\in A$. This shows
  $U_{A,x}^\finitegraph \leq c_{b'} (\ln(N_n))^{-6}$.

  We turn to $V_{A,x}^\finitegraph$. By  \eqref{0005}  we have
  $V_{A,x}^\finitegraph \leq \sup_{y \in A} | \Psi_\finitegraph(y) |
  \cdot P_x^\finitegraph [T_{F_A} < H_A <  T_{F_A} +t_n]
  \leq b' \sqrt{\ln(N_n)} c_b(\ln(N_n))^{-5}$.
  This shows $V_{A,x}^\finitegraph \leq c_{b,b'} (\ln(N_n))^{-4}$.

  Finally, we consider  $W_{A,x}^\finitegraph$. Let us define
  \begin{equation}
    \label{545}
    \begin{split}
      Y_{A,x}^\finitegraph &\coloneqq \Big|
      \tfrac{E_x^\finitegraph[H_A]}{E_\pi^\finitegraph[H_A] }  -
      P_x^\finitegraph[H_A>T_{F_A}]  \Big|,
      \\
      Z_{A,x}^\finitegraph &\coloneqq \Big|  E_x^\finitegraph
      \big[\Psi_\finitegraph(\overline{X}_{H_A})   \bbone_{\{ H_A \geq  T_{F_A}
            +t_n\}}  \big] - P_x^\finitegraph[H_A > T_{F_A}]  E_\pi^\finitegraph \big
      [\Psi_\finitegraph(\overline{X}_{H_A}) \big] \Big|.
    \end{split}
  \end{equation}
  By adding and subtracting
  $\tfrac{1}{P_x^\finitegraph[H_A > T_{F_A}]}
  \tfrac{E_x^\finitegraph[H_A]}{E_\pi^\finitegraph[H_A] }  E_x^\finitegraph
  \big[\Psi_\finitegraph(\overline{X}_{H_A})   \bbone_{\{ H_A \geq  T_{F_A}
        +t_n\}}  \big]$
  inside the expression for $W_{A,x}^\finitegraph$ we obtain
  \begin{equation}
    \label{817}
    W_{A,x}^\finitegraph
    \leq \tfrac{| E_x^\finitegraph [\Psi_\finitegraph(\overline{X}_{H_A})
        \bbone_{\{ H_A \geq  T_{F_A} +t_n\}}  ] |}{P_x^\finitegraph[H_A > T_{F_A}]}
    \,Y_{A,x}^\finitegraph
    + \tfrac{1}{P_x^\finitegraph[H_A > T_{F_A}]}
    \tfrac{E_x^\finitegraph[H_A]}{E_\pi^\finitegraph[H_A] } Z_{A,x}^\finitegraph.
  \end{equation}
  To the first term on the right hand side of \eqref{817} we apply
  $P_x^\finitegraph[H_A > T_{F_A}] \geq \frac{d-2}{d-1}$ (by
    \eqref{0007}) as well as \eqref{557} and the assumption on the
  supremum of $\Psi_\finitegraph$ on $A$. For the second term we first
  observe \eqref{557} and then again use
  $P_x^\finitegraph[H_A > T_{F_A}] \geq \frac{d-2}{d-1}$. In this way we
  obtain
  \begin{equation}
    \label{821}
    W_{A,x}^\finitegraph
    \leq c_{b,b'}(\ln(N_n))^{-2}
    + (1+c_b (\ln(N_n))^{-3}) Z_{A,x}^\finitegraph.
  \end{equation}
  We proceed to bound $Z_{A,x}^\finitegraph$. By \eqref{0003} and the
  strong Markov property for time $T_{F_A}$ it holds
  \begin{align*}
    &E_x^\finitegraph \big[\Psi_\finitegraph(\overline{X}_{H_A})   \bbone_{\{
          H_A \geq  T_{F_A} +t_n\}}  \big] \\
    &\qquad \qquad = \sum_{z\in B_\finitegraph(A,s_n)^\mathsf{c}}
    P_x^\finitegraph[\overline{X}_{T_{F_A}}=z, H_A>T_{F_A}] E_z^\finitegraph
    \big[\Psi_\finitegraph(\overline{X}_{H_A})   \bbone_{\{ H_A \geq t_n\}}
      \big].
  \end{align*}
  This combined with \eqref{96} implies
  $Z_{A,x}^\finitegraph
  \leq   \sup_{z\in B_\finitegraph(A,s_n)^\mathsf{c}}    \big|
  E_z^\finitegraph[\Psi_\finitegraph(\overline{X}_{H_A}) \bbone_{\{ H_A \geq
        t_n \}}]     - E_{\pi}^\finitegraph[\Psi_\finitegraph(\overline{X}_{H_A})]
  \big|$.
  Now for $z\in B_\finitegraph(A,s_n)^\mathsf{c}$, by the Markov property applied
  at  time $t_n$ and the definition of $E_{\pi}^\finitegraph$,
  \begin{equation*}
    \begin{split}
      &\big| E_z^\finitegraph[\Psi_\finitegraph(\overline{X}_{H_A}) \bbone_{\{
            H_A \geq t_n\}}]     -
      E_{\pi}^\finitegraph[\Psi_\finitegraph(\overline{X}_{H_A})] \big| \\
      &\leq  \sum_{w\in \finitegraph}   \big|
      E_w^\finitegraph[\Psi_\finitegraph(\overline X_{H_A})] \big|  \cdot \big|
      P_z^\finitegraph[ \overline{X}_{t_n}=w ,  H_A \geq t_n]     -  \tfrac{1}{N_n}
      \big| \overset{\eqref{0006}}{\leq} c_{b,b'}(\ln(N_n))^{-4},
    \end{split}
  \end{equation*}
  where in the last inequality we also use the assumption on the supremum
  of $\Psi_\finitegraph$ on $A$. All in all we have shown
  $Z_{A,x}^\finitegraph \leq c_{b,b'}(\ln(N_n))^{-4}$. Thus by
  \eqref{821} we deduce
  $W_{A,x}^\finitegraph \leq c_{b,b'} (\ln(N_n))^{-2} $ and the proof of
  \eqref{520} is complete.

  We come to the proof of \eqref{549} for which we bound the three terms
  $\overline{U}_{A,x}^\finitegraph$, $\overline{V}_{A,x}^\finitegraph$
  and $\overline{W}_{A,x}^\finitegraph$ separately. For
  $\overline{U}_{A,x}^\finitegraph$ we first note that one has
  $E_x^\finitegraph \big[G_\finitegraph(X_{H_A},x)
    \bbone_{\{ H_A = T_{F_A} \}} \big]
  = E_x^\finitegraph \big[G_\finitegraph(X_{T_{F_A}},x) \bbone_{\{ H_A = T_{F_A} \}} \big]
  = E_x^\finitegraph \big[G_\finitegraph(X_{T_{F_A}},x)\big]
  - E_x^\finitegraph \big[G_\finitegraph(X_{T_{F_A}},x) \bbone_{\{ H_A > T_{F_A} \}} \big]$.
  By \eqref{0003}, on the event $\{ H_A > T_{F_A} \}$ the simple random
  walk started at $x$ is at distance $s_n$ from $x$ when it leaves $F_A$.
  Therefore
  \begin{equation*}
    E_x^\finitegraph \big[G_\finitegraph(X_{T_{F_A}},x) \bbone_{\{ H_A >
          T_{F_A} \}} \big]
    \leq  \sup_{z\in S_\finitegraph(x,s_n)}\!
    G_\finitegraph(z,x) \, \underbrace{P_x^\finitegraph[H_A> T_{F_A}]}_{\leq 1}
    \stackrel[\eqref{471}]{\eqref{30}}{\leq}  c(\ln(N_n))^{-8}.
  \end{equation*}
  Thus we have
  \begin{equation}
    \label{551}
    \begin{split}
      \overline{U}_{A,x}^\finitegraph
      &\overset{\phantom{\eqref{1.7}}}{\leq} \Big|  G_\finitegraph(x,x)  -
      E_x^\finitegraph \big[G_\finitegraph(X_{T_{F_A}},x) \big]  - \tfrac{d}{d-1}
      \Big|  + c(\ln(N_n))^{-8} \\
      &\stackrel[\eqref{0007}]{\eqref{1.7}}{\leq}  \big|  g_\finitegraph^{F_A}(x,x)
      - \tfrac{d}{d-1} \big| + \tfrac{c\ln(N_n)}{N_n} + c(\ln(N_n))^{-8}.
    \end{split}
  \end{equation}
  Note that by assumption $\tx(F_A)=0$. So if we define
  $B \coloneqq B_\treegraph^+(\rot,s_n)
  \setminus \{\rot\} \subseteq \treegraph$
  and take $x_1 \in S_\treegraph^+(\rot,1)$, then by definition we
  have $g_\finitegraph^{F_A}(x,x) = g_\treegraph^{B}(x_1,x_1)$. From
  \eqref{1.4} we see that
  \begin{align*}
    g_\treegraph^{B}(x_1,x_1) &= g_\treegraph(x_1,x_1) -
    E_{x_1}^\treegraph[g_\treegraph(X_{T_B},x_1)] \\
    &= g_\treegraph(x_1,x_1) -
    g_\treegraph(\rot,x_1) P_{x_1}^\treegraph[H_\rot=T_B] -
    g_\treegraph(z,x_1) P_{x_1}^\treegraph[H_\rot>T_B]
  \end{align*}
  for any fixed $z \in
  S_\treegraph^+(\rot,s_n+1)$.
  By \eqref{1.1} this shows that
  \begin{equation*}
    g_\treegraph^{B}(x_1,x_1) = \frac{d-1}{d-2} - \frac{d-1}{d-2} \frac1{d-1}
    P_x^\finitegraph[H_A = T_{F_A}] - \frac{d-1}{d-2} \big( \frac1{d-1} \big)^{s_n}
    P_x^\finitegraph[H_A > T_{F_A}].
  \end{equation*}
  So we have obtained
  \begin{equation*}
    \begin{split}
      \big|  g_\finitegraph^{F_A}(x,x)   - \tfrac{d}{d-1} \big|
      &\overset{\phantom{\eqref{0007}}}{\leq} \big|  \tfrac{d-1}{d-2} - \tfrac{1}{d-2}
      P_x^\finitegraph[H_A = T_{F_A}]   - \tfrac{d}{d-1} \big|  + \tfrac{d-1}{d-2}
      \big( \tfrac1{d-1} \big)^{s_n} P_x^\finitegraph[H_A > T_{F_A}] \\
      &\stackrel[\eqref{471}]{\eqref{0007}}{\leq} \tfrac{1}{d-2} c(\ln(N_n))^{-7}  +
      c(\ln(N_n))^{-8}.
    \end{split}
  \end{equation*}
  This, together with \eqref{551} shows
  $\overline{U}_{A,x}^\finitegraph \leq c(\ln(N_n))^{-7}$.

  We turn to $\overline{V}_{A,x}^\finitegraph$. By \eqref{1.8} there
  exists $c>0$ such that
  $\sup_{y,z\in \finitegraph}G_\finitegraph(y,z) \leq c$. Therefore
  $\overline{V}_{A,x}^\finitegraph
  \leq c  \,P_x^\finitegraph[T_{F_A}<H_A < T_{F_A}+t_n]$
  and so \eqref{0005} implies
  $\overline{V}_{A,x}^\finitegraph \leq c_b(\ln(N_n))^{-5}$.

  Finally, we consider $\overline{W}_{A,x}^\finitegraph$.
  Let us define
  \begin{equation*}
    \overline{Z}_{A,x}^\finitegraph \coloneqq \Big|  E_x^\finitegraph
    \big[G_\finitegraph(X_{H_A},x) \bbone_{\{ H_A \geq  T_{F_A} +t_n\}}  \big]
    - P_x^\finitegraph[H_A > T_{F_A}]
    E_\pi^\finitegraph[G_\finitegraph(X_{H_A},x)]    \Big|
  \end{equation*}
  and recall $Y_{A,x}^\finitegraph$ from \eqref{545}.
  Inside $\overline{W}_{A,x}^\finitegraph$ we can add and subtract
  $\tfrac{1}{P_x^\finitegraph[H_A > T_{F_A}]}
  \tfrac{E_x^\finitegraph[H_A]}{E_\pi^\finitegraph[H_A] }  \cdot E_x^\finitegraph
  \big[G_\finitegraph(X_{H_A},x)   \bbone_{\{ H_A \geq  T_{F_A} +t_n\}}
    \big]$
  to obtain
  \begin{equation}
    \label{95}
    \overline{W}_{A,x}^\finitegraph
    \leq \tfrac{E_x^\finitegraph [G_\finitegraph(X_{H_A},x)   \bbone_{\{ H_A
            \geq  T_{F_A} +t_n\}}  ] }{P_x^\finitegraph[H_A > T_{F_A}]}
    \,Y_{A,x}^\finitegraph
    + \tfrac{1}{P_x^\finitegraph[H_A > T_{F_A}]}
    \tfrac{E_x^\finitegraph[H_A]}{E_\pi^\finitegraph[H_A] }
    \overline{Z}_{A,x}^\finitegraph. \\
  \end{equation}
  To the first term on the right hand side of \eqref{95} we apply
  $P_x^\finitegraph[H_A > T_{F_A}] \geq \frac{d-2}{d-1}$ (by
    \eqref{0007}) as well as \eqref{557} and
  $\sup_{y,z\in \finitegraph}G_\finitegraph(y,z) \leq c$ (by
    \eqref{1.8}). For the second term we first observe \eqref{557} and
  then again use $P_x^\finitegraph[H_A > T_{F_A}] \geq \frac{d-2}{d-1}$.
  In this way we obtain
  \begin{equation}
    \label{823}
    \overline{W}_{A,x}^\finitegraph \leq c_b(\ln(N_n))^{-3}  + (1+c_b
      (\ln(N_n))^{-3}) \overline{Z}_{A,x}^\finitegraph.
  \end{equation}
  We proceed to bound $\overline{Z}_{A,x}^\finitegraph$.
  By \eqref{0003} and the strong Markov property it holds
  \begin{align*}
    &E_x^\finitegraph \big[G_\finitegraph(X_{H_A},x)  \bbone_{\{ H_A \geq
          T_{F_A} +t_n\}}  \big]  \\
    &\qquad \qquad  = \sum_{z\in B_\finitegraph(A,s_n)^\mathsf{c}}
    P_x^\finitegraph[\overline{X}_{T_{F_A}}=z, H_A>T_{F_A}]
    E_z^\finitegraph[G_\finitegraph(X_{H_A},x) \bbone_{\{ H_A \geq t_n \}}].
  \end{align*}
  This combined with  \eqref{96} gives
  $\overline{Z}_{A,x}^\finitegraph
  \leq   \sup_{z\in B_\finitegraph(A,s_n)^\mathsf{c}}   \big|
  E_z^\finitegraph[G_\finitegraph(X_{H_A},x) \bbone_{\{ H_A \geq t_n \}}]
  - E_{\pi}^\finitegraph[G_\finitegraph(X_{H_A},x)] \big|$.
  Now for  $z\in B_\finitegraph(A,s_n)^\mathsf{c}$, by the Markov
  property applied at  time $t_n$ and the definition of $E_{\pi}^\finitegraph$,
  \begin{equation*}
    \begin{split}
      &\big| E_z^\finitegraph[G_\finitegraph(X_{H_A},x) \bbone_{\{ H_A \geq t_n
            \}}]  - E_{\pi}^\finitegraph[G_\finitegraph(X_{H_A},x)] \big| \\
      &\leq \sum_{w\in \finitegraph}
      E_w^\finitegraph[G_\finitegraph(X_{H_A},x)]  \cdot  \big|  P_z^\finitegraph[
        \overline{X}_{t_n}=w , H_A \geq t_n]   -  \tfrac1{N_n} \big|
      \overset{\eqref{0006}}{\leq} c_b(\ln(N_n))^{-5},
    \end{split}
  \end{equation*}
  where in the last inequality we again use
  $\sup_{y,w\in \finitegraph}G_\finitegraph(y,w) \leq c$ by \eqref{1.8}.
  All in all we have shown
  $\overline{Z}_{A,x}^\finitegraph \leq c_b(\ln(N_n))^{-5}$. Thus by
  \eqref{823} we deduce
  $\overline{W}_{A,x}^\finitegraph \leq c_b(\ln(N_n))^{-3}$ and
  \eqref{549} is shown. This concludes the proof of  Proposition
  \ref{535} and Section \ref{sectionconditionaldistribution}.
\end{proof}



\section{Microscopic components in the subcritical phase}
\label{subcriticalsection}

We start the analysis of level-set percolation of the zero-average
Gaussian free field $\Psi_\finitegraph$ on $\finitegraph$. The goal of
this section is to show \eqref{0.6} in the form of Theorem
\ref{microscopiccomponents} below, i.e.~the existence of a subcritical
phase in which, with high probability for large $n$,  level sets of
$\Psi_\finitegraph$ only have connected components of cardinality at most
logarithmic in the size of the graph. To precisely state the result, we
recall from the introduction the critical value  $h_\star$  for level-set
percolation of the Gaussian free field $\varphi_\treegraph$ on
$\treegraph$ (see \eqref{0.5}) and also the notation
$E_{\Psi_\finitegraph}^{\geq h}$ for the level set of $\Psi_\finitegraph$
above level $h \in \mathbb R$ (see  \eqref{0001}). For $h\in \mathbb R$
we further denote by $\mathcal C_{\textup{max}}^{\finitegraph,h}$  an
arbitrary connected component of $E_{\Psi_\finitegraph}^{\geq h}$ with
maximal number of vertices. We will only be interested in its
cardinality. Moreover, for $x\in \finitegraph$ and $h\in \mathbb R$ we
define $\mathcal C_{x}^{\finitegraph,h}$ to be the connected component of
$E_{\Psi_\finitegraph}^{\geq h}$ containing $x$. The main result of this
section is

\begin{theorem}
  \label{microscopiccomponents}
  Let $h>h_\star$. Then for all $\kappa > 0$  there exist $c_{h,\kappa}>0$
  and $K_{h,\kappa}>0$ such that for all $n\geq 1$
  \begin{equation*}
    \mathbb P^\finitegraph \big[ |\mathcal C_{\textup{max}}^{\finitegraph,h}| \geq
      K_{h,\kappa} \ln(N_n)  \big] \leq   c_{h,\kappa} N_n^{-\kappa}.
  \end{equation*}
  In particular,  for some $K_h>0$ one has
  $\lim_{n\to \infty} 
  \mathbb P^\finitegraph \big[ |\mathcal C_{\textup{max}}^{\finitegraph,h}|
    \leq K_h \ln(N_n)  \big] =1$.
\end{theorem}

Before explaining the details of the proof of Theorem
\ref{microscopiccomponents}, let us make the basic observation that a
union bound reduces the problem to show  that for $h>h_\star$ and for all
$\kappa > 0$ there exist $c_{h,\kappa}>0$ and $K_{h,\kappa}>0$ such that
for all $n\geq 1$ and $x\in \finitegraph$
\begin{equation}
  \label{556}
  \mathbb P^\finitegraph \big[ |\mathcal C_{x}^{\finitegraph,h}| \geq
    K_{h,\kappa} \ln(N_n)  \big] \leq   c_{h,\kappa} N_n^{-1-\kappa}.
\end{equation}
So it remains to show \eqref{556}. We will make use of a certain
exploration process exploring $\mathcal C_{x}^{\finitegraph,h}$ for a
fixed $x \in \finitegraph$. This will enable us to control
$\mathbb P^\finitegraph \big[ |\mathcal C_{x}^{\finitegraph,h}|
  \geq K_{h,\kappa} \ln(N_n)  \big]$.
A similar approach has for example been followed in \cite{21} to prove a
result analogous  to the above Theorem \ref{microscopiccomponents}  but
for the vacant set of simple random walk on $\finitegraph$ in place of
the level set of the zero-average Gaussian free field.

We now give the idea of the proof of \eqref{556}. The details of the
exploration process itself are given afterwards. A crucial ingredient is
the precise understanding of the conditional distribution of the
zero-average Gaussian free field on non-explored vertices given its value
on already explored vertices. As we have seen in Proposition \ref{535} in
Section \ref{sectionconditionaldistribution}, under certain geometric
conditions the conditional distribution of $\Psi_\finitegraph$ shows
strong similarities with the conditional distribution of the Gaussian
free field $\varphi_\treegraph$ on $\treegraph$. While exploring
$\mathcal C_{x}^{\finitegraph,h}$, the exploration process will separate
the vertices found in $\mathcal C_{x}^{\finitegraph,h}$ into a union of
rooted disjoint subtrees of $\finitegraph$ in which all vertices except
for the root satisfy the aforementioned geometric conditions. In this way
we reduce the proof of \eqref{556} to a control of the number of vertices
contained in these union of subtrees (Proposition~\ref{562}). As a result
from \cite{21} shows (see also Lemma \ref{532}), the number of steps the
exploration process encounters a situation in which the geometric
assumptions fail to be satisfied is not too large. This controls the
number of distinct subtrees created by the exploration process because in
each subtree there is exactly one vertex  which does not satisfy the
conditions (its root). Since the other vertices of a subtree satisfy the
geometric conditions, we can employ the  similarity between the
conditional distribution of $\Psi_\finitegraph$ and $\varphi_\treegraph$
to couple the zero-average Gaussian free field on each distinct subtree
separately with an independent copy of the Gaussian free field
$\varphi_\treegraph$ on $\treegraph$ (Lemma \ref{806}). This translates
the question about the number of vertices contained in the disjoint
subtrees into the number of vertices contained in connected components of
the level set of $\varphi_\treegraph$ (Corollary~\ref{841}). A result
from \cite{AC1} (recalled in \eqref{2513}) about exponential moments of
the size of these connected components then ultimately leads to  the
proof of Proposition \ref{562} and hence of \eqref{556}.

We now describe the  exploration process exploring
$\mathcal C_{x}^{\finitegraph,h}$ for a fixed $x \in \finitegraph$ and to
facilitate the discussion we include a concrete algorithm implementing it
(Algorithm \ref{algorithm}). The exploration process is a \emph{modified}
breadth-first-search  that discovers the field $\Psi_\finitegraph$ on the
graph step by step. It employs \emph{two} queues (a primary and a
  secondary one) that work in the usual first-in-first-out manner and
store the vertices to be explored. The exploration process starts by
revealing $\Psi_\finitegraph(x)$. The vertices where $\Psi_\finitegraph$
has been revealed are called explored and they can be either part of
$\mathcal C_{x}^{\finitegraph,h}$ or not. If a vertex is explored and is
revealed to be part of $\mathcal C_{x}^{\finitegraph,h}$,  then its
neighbours which are neither already explored nor already in one of the
two queues are added to the primary queue. To avoid ambiguity, we suppose
that the vertices of $\finitegraph$ are equipped with some ordering and
that they   are added to the queue  following  this ordering. Vertices
taken out of the primary queue are first checked to be \emph{good
  vertices at the boundary of the so far explored vertices} (recall
  \eqref{850} and above it for the definition): if they are, the
exploration process proceeds with their exploration; if they are not,
they are transferred to the secondary queue and their exploration is
postponed. The first vertex in the secondary queue is only taken out to
be explored if the primary queue is empty.

To formalise this exploration process we now give an algorithm
implementing it (see Algorithm \ref{algorithm} below). The algorithm
constructs on some auxiliary probability space
$(\Omega,\mathcal A,\mathbb P)$ a family of random variables
$(\psi(z))_{z\in B}$ such that $(\psi(z))_{z\in B}$ under $\mathbb P$ has
the same distribution as $(\Psi_\finitegraph(z))_{z\in B}$ under
$\mathbb P^\finitegraph$. Here $B\subseteq \finitegraph$ is some (random)
connected set of vertices containing $x$. We use $\mathsf{PQ}$,
$\mathsf{SQ}$ and $\mathsf{E}$ to denote the evolving sets of vertices in
the primary queue, vertices in the secondary queue and explored vertices
during the run of the algorithm. Furthermore, we also keep track of the
explored vertices $z\in \mathsf{E}$ for which $\psi(z)\geq h$ using the
set $\mathsf{C} \subseteq \mathsf{E}$. Additionally to the exploration,
the algorithm aggregates the vertices discovered to be in $\mathsf{C}$
into disjoint subtrees $(\mathsf{T}^y)_y$ of $\finitegraph$ indexed by
\emph{bad vertices} $y\in \finitegraph$ (meaning they were in
  $\mathsf{SQ}$ at some point of the algorithm). Moreover, the algorithm
stops for one of two reasons: either because both the primary and
secondary queue are empty, or because it already discovered that
$\mathsf{C}$ has at least size $K_{h,\kappa} \ln(N_n)$ for some
$K_{h,\kappa}$ to be specified later (below \eqref{809}).

We need some more notation for the algorithm. Let
$(\xi_z)_{z\in\finitegraph}$ be i.i.d.~standard normal random variables
on the auxiliary probability space $(\Omega,\mathcal A,\mathbb P)$. For
$A\subseteq \finitegraph$ non-empty and $u\in \finitegraph$ we abbreviate
by  $a(u,\psi, A)$ the right hand side of \eqref{449} where $x$ and
$\Psi_\finitegraph$ are replaced by $u$ and $\psi$. In particular,
$a(u,\psi, A)$ is a random variable measurable with respect to
$\sigma(\psi(w),w\in A)$. By $b(u,A)$ we abbreviate the right hand side
of \eqref{456} where $x$ is replaced by $u$. For $A=\emptyset$ and
$u\in \finitegraph$ we define $a(u,\psi, \emptyset) \coloneqq 0$ and
$b(u,\emptyset) \coloneqq G_\finitegraph(u,u)$. By Lemma
\ref{conditionaldistribution} and the fact that $\psi$ is a Gaussian
field, we have that
\begin{equation}
  \label{94}
  \parbox{12cm}{
    for $A\subseteq \finitegraph$ and $u\in \finitegraph$ the random variable
    $a(u,\psi,A) + \xi_u \cdot b(u,A)^\frac12$ under $\mathbb P$ has the same
    distribution as $\Psi_\finitegraph(u)$ conditional on
    $\sigma(\Psi_\finitegraph(w), w\in A)$ under $\mathbb P^\finitegraph$.
  }
\end{equation}
The algorithm is as follows:

\begin{algorithm}[H]
  \caption{}
  \label{algorithm}
  \begin{algorithmic}[1]
    \State set $\mathsf{PQ}\coloneqq \emptyset$, $\mathsf{SQ} \coloneqq
    \{x\}$, $\mathsf{E} \coloneqq \emptyset$, $\mathsf{C} \coloneqq \emptyset$ and
    also $\mathsf{T}^w \coloneqq \emptyset$ for all $w\in \finitegraph$
    \While {secondary queue $\mathsf{SQ}$ is not empty}
    \State   take vertex $y$ out of $\mathsf{SQ}$
    \State   generate the random variable $\psi(y) \coloneqq
    a(y,\psi,\mathsf{E}) + \xi_y \cdot b(y,\mathsf{E})^\frac12$
    \State add $y$ to the set $\mathsf{E}$ of explored vertices
    \If{$\psi(y) \geq h$}
    \State add $y$ to the subtree $\mathsf{T}^y$ and to the
    set  $\mathsf{C}$
    \If {$|\mathsf{C}|\geq K_{h,\kappa} \ln(N_n)$}
    stop the algorithm
    \EndIf
    \State add all neighbours of $y$ which are
    neither already explored nor in any of the \phantom{atsdaf} two  queues to the
    primary  queue $\mathsf{PQ}$
    \While {primary queue $\mathsf{PQ}$ is not empty}
    \State take vertex $z$ out of $\mathsf{PQ}$
    \If {$z$ is not a good vertex at the boundary
      of $\mathsf{E}$, that is, $z \notin G_{\mathsf{E}}$,}
    \State add $z$ to the secondary queue
    $\mathsf{SQ}$
    \Else
    \State generate the random variable
    $\psi(z) \coloneqq a(z,\psi,\mathsf{E}) + \xi_z \cdot b(z,\mathsf{E})^\frac12$
    \State add $z$ to the set $\mathsf{E}$
    of explored vertices
    \If{$\psi(z)\geq h$}
    \State add $z$ to the subtree
    $\mathsf{T}^y$ and to the set $\mathsf{C}$
    \If {$|\mathsf{C} | \geq
      K_{h,\kappa} \ln(N_n)$} stop the algorithm
    \EndIf
    \State add all neighbours of
    $z$ which are neither already explored nor in \phantom{asdfasdftadasadf} any of
    the two queues to the primary queue $\mathsf{PQ}$
    \EndIf
    \EndIf
    \EndWhile
    \EndIf
    \EndWhile
  \end{algorithmic}
\end{algorithm}

Let  $\mathsf{E}_\textup{end}$, $\mathsf{C}_\textup{end}$ and
$\mathsf{T}_\textup{end}^w$, $w\in \finitegraph$, denote the sets
$\mathsf{E}$, $\mathsf{C}$ and $\mathsf{T}^w$, $w\in \finitegraph$, at
the end of the algorithm.  By that moment  we have constructed
$(\psi(z))_{z\in \mathsf{E}_\textup{end}}$ and (see \eqref{94})
\begin{equation}
  \label{555}
  \begin{split}
    &\text{$(\psi(z))_{z\in \mathsf{E}_\textup{end}}$ under $\mathbb P$ has the
      same distribution as $(\Psi_\finitegraph(z))_{z\in \mathsf{E}_\textup{end}}$
      under $\mathbb P^\finitegraph$.}
  \end{split}
\end{equation}
By construction of the algorithm one has
$|\mathsf{C}_\textup{end}| \leq K_{h,\kappa} \ln(N_n)+1$ and so by
\eqref{0} also
\begin{equation}
  \label{851}
  |\mathsf{E}_\textup{end}| \leq d (K_{h,\kappa} \ln(N_n)+1).
\end{equation}
This is due to $\mathsf{E} \subseteq B_\finitegraph(\mathsf{C},1)$ with
$B_\finitegraph(\emptyset,1)\coloneqq \{x\}$ holding at any moment of the
algorithm since a vertex can only get explored (except for $x$) if at
some point it was added to a queue, meaning it was a neighbour of a
vertex added to $\mathsf{C}$.

Note that, whenever some $y\in \finitegraph$ is taken out of $\mathsf{SQ}$
on line 3 of the algorithm (a bad vertex), one has
$\mathsf{PQ}= \emptyset$ at that moment by construction. Until the next
bad vertex is taken out of $\mathsf{SQ}$, all vertices
$z \in \finitegraph$ considered by the algorithm and which are found to
be good and in $\mathsf{C}$ will be part of $\mathsf{T}_\textup{end}^y$.
So if  $y_1,\ldots,y_{k_\textup{end}}$ denote the successive vertices
that were taken out of $\mathsf{SQ}$ during the algorithm, then
$\mathsf{C}_\textup{end} = \bigcup_{i=1}^{k_\textup{end}} \mathsf{T}_\textup{end}^{y_i}$.
In particular, $y_1= x$ and $k_\textup{end}$ is the total number of bad
vertices encountered by the algorithm.

Furthermore, on the event that the algorithm terminates because both
queues become empty (and not because at some point
  $|\mathsf{C}_\textup{end}| \geq K_{h,\kappa}\ln(N_n)$), note that
$|\mathsf{C}_\textup{end}|$  has the same distribution as
$|\mathcal C_{x}^{\finitegraph,h}|$ under $\mathbb P^\finitegraph$ by
\eqref{555}. Therefore
$\mathbb P\big[ |\mathsf{C}_\textup{end}| < K_{h,\kappa}\ln(N_n)\big]
= \mathbb P^\finitegraph \big[ |\mathcal C_{x}^{\finitegraph,h}|
  < K_{h,\kappa} \ln(N_n)  \big]$.

We want to distinguish the situation in which the field $\psi$ produced
by the algorithm has anomalous  values, meaning $|\psi(z)| \geq M_n$ for
some $z\in \mathsf{E}_\textup{end}$ and $M_n>0$. We are going to specify
this value now. Note that for any $\kappa>0$ there is $c_\kappa>0$ such that
\begin{equation}
  \label{531}
  \mathbb P^\finitegraph \Big[  \sup_{z \in \finitegraph} | \Psi_\finitegraph(z)
    | \geq c_\kappa\sqrt{\ln(N_n)} \Big]
  \leq  2 N_n^{-1-\kappa} \quad \text{for all $n\geq 1$.}
\end{equation}
This can be shown by the same computations as in \cite{7}, equations
(2.35)--(2.38), replacing $g(0)$ therein with $\sup_{z\in \finitegraph}
G_\finitegraph(z,z)$, which is bounded by $3\frac{d-1}{d-2}$ (see \eqref{30}).
Use also use $\mathbb P^\finitegraph \big[ \sup_{z \in \finitegraph} |
  \Psi_\finitegraph(z) | \geq a \big]
\leq 2 \mathbb P^\finitegraph \big[  \sup_{z \in \finitegraph}
  \Psi_\finitegraph(z)  \geq a \big]$ for \eqref{531}. We set
\begin{equation*}
  M_n \coloneqq c_\kappa \sqrt{\ln(N_n)}.
\end{equation*}
So one has
\begin{equation*}
  \begin{split}
    &\mathbb P^\finitegraph \big[ |\mathcal C_{x}^{\finitegraph,h}| \geq
      K_{h,\kappa} \ln(N_n)  \big]
    = \mathbb P\big[ |\mathsf{C}_\textup{end}| \geq K_{h,\kappa}\ln(N_n)\big]  \\
    &\overset{\phantom{\eqref{555}}}{\leq} \mathbb P\big[
      |\mathsf{C}_\textup{end}| \geq K_{h,\kappa}\ln(N_n) \, ,   \sup_{z\in
        \mathsf{E}_\textup{end}} |\psi(z)| < M_n \big] + \mathbb P \big[  |\psi(z)|
      \geq M_n \text{ for some $z \in \mathsf{E}_\textup{end}$} \big] \\
    &\stackrel[\eqref{531}]{\eqref{555}}{\leq} \mathbb P\Big[
      \sum_{i=1}^{k_\textup{end}} | \mathsf{T}_\textup{end}^{y_i}| \geq
      K_{h,\kappa}\ln(N_n) \, ,    \sup_{z\in \mathsf{E}_\textup{end}} |\psi(z)| <
      M_n \Big] + 2 N_n^{-1-\kappa}.
  \end{split}
\end{equation*}
Thus in order to show \eqref{556} and ultimately Theorem
\ref{microscopiccomponents} we need to show

\begin{proposition}
  \label{562}
  Let $h>h_\star$. Then for all $\kappa>0$ there exist $c_{h,\kappa}>0$
  and $K_{h,\kappa}>0$ such that for all $n\geq 1$ and $x\in \finitegraph$
  one has for the Algorithm \ref{algorithm} above
  \begin{equation}
    \label{810}
    \mathbb P\Big[ \sum_{i=1}^{k_\textup{end}} | \mathsf{T}_\textup{end}^{y_i} |
      \geq K_{h,\kappa}\ln(N_n) \, , \sup_{z\in \mathsf{E}_\textup{end}} |\psi(z)| <
      M_n \Big]  \leq c_{h,\kappa} N_n^{-1-\kappa}.
  \end{equation}
\end{proposition}

The proof of Proposition \ref{562} relies on the following two lemmas.
The first one (Lemma~\ref{532}, already proven in \cite{21}) bounds the
number of bad vertices $k_\textup{end}$ encountered by Algorithm
\ref{algorithm}, that is, the number of vertices of $\finitegraph$ that
at some point during the run of the algorithm were in the secondary queue
$\mathsf{SQ}$. The second one (Lemma~\ref{806}) constructs for each
$i=1,\ldots,k_\textup{end}$ a coupling of $\psi$ on
$\mathsf{T}_\textup{end}^{y_i}$ with an independent copy of
$\varphi_\treegraph$, showing that $\psi$ on
$\mathsf{T}_\textup{end}^{y_i}$ can be approximated by
$\varphi_\treegraph$. This makes use of  Proposition~\ref{535}. Via
Corollary \ref{841}   of Lemma \ref{806}  we then prove Proposition
\ref{562}.

\begin{lemma}
  \label{532}
  There exists $c_1>0$  such that for all $n\geq 1$ and
  $x\in \finitegraph$ one has for the above Algorithm \ref{algorithm}
  that
  $k_\textup{end} \leq c_1 K_{h,\kappa} s_n^2 \eqqcolon k_\textup{max}$
  (recall that $s_n$ is given in \eqref{471}).
\end{lemma}
\begin{proof}
  This follows from \cite{21}, Proposition 5.4. Although the algorithm
  employed there does not exactly match our algorithm, the proof does not
  rely on a specific algorithm (as explained in the proof of Proposition
    5.4 in \cite{21}). It is purely deterministic and only uses the
  properties \eqref{0}--\eqref{2} of  $\finitegraph$.
\end{proof}

\begin{lemma}
  \label{806}
  Let $h\in \mathbb R$ and $\varepsilon>0$. Consider  Algorithm
  \ref{algorithm} and recall $k_\textup{max}$ from Lemma~\ref{532}. Then
  on the same  auxiliary space $(\Omega,\mathcal A,\mathbb P)$ as $\psi$
  one can define centred Gaussian fields
  $\phi^1,\ldots,\phi^{k_\textup{max}}$ on $\treegraph$ such that,
  conditionally on $\psi(y_1),\ldots,\psi(y_{k_\textup{end}})$, the
  following properties hold (see \eqref{e:defPa} for notation):
  \begin{flalign}
    \quad \ \, \bullet& \ \, \text{for all $n$ large enough and all
      $i=1,\ldots,k_\textup{end}$ there exists a set $B^i$ with}    &
    \label{800}
    \\
    \quad \ \, & \ \, \text{$\mathsf{T}_\textup{end}^{y_i} \subseteq B^i\subseteq
      B_\finitegraph( \mathsf{T}_\textup{end}^{y_i},1)$ and an injection $\tau^i :
      B^i\to \treegraph$ such that $\tau^i(\mathsf{T}_\textup{end}^{y_i})$}  &
    \nonumber \\
    \quad \ \, & \ \, \text{is a connected subset of $\treegraph$ containing the
      root $\rot\in \treegraph$ and on the event} & \nonumber \\
    \quad \ \, & \ \, \text{$\{\textstyle \sup_{z\in \mathsf{E}_\textup{end}}
        |\psi(z)| < M_n\}$ one has $\big| \psi(z) - \phi^i(\tau^i(z)) \big| \leq
      \varepsilon$ for all $z\in B^i$} &\nonumber \\
    \quad \ \,  \bullet& \ \, \text{$\phi^i$ has the same distribution as
      $\varphi_\treegraph$ under $\mathbb P^\treegraph_{\psi(y_i)}$ for all
      $i=1,\ldots,k_\textup{end}$,} &
    \label{801}   \\
    \quad \ \, & \ \, \text{$\phi^i$ has the same distribution as
      $\varphi_\treegraph$ under $\mathbb P^\treegraph_{M_n}$ for all
      $i=k_\textup{end}+1,\ldots, k_\textup{max}$} &\nonumber  \\
    \quad \ \,  \bullet& \ \, \text{$\phi^1,\ldots,\phi^{k_\textup{max}}$ are
      independent.} &
    \label{802}
  \end{flalign}
\end{lemma}
\begin{proof}
  Let $Y_x^i$ for $x\in \treegraph\setminus\{\rot\}$ and
  $1\leq i \leq k_\textup{max}$ be a sequence of i.i.d.~random variables
  of distribution $\mathcal N(0,  \frac{d}{d-1})$ defined on the
  auxiliary probability space $(\Omega, \mathcal A, \mathbb P)$. Let
  $i\in\{1,\ldots,k_\textup{end} \}$. As explained below \eqref{851}, the
  subtree $\mathsf{T}_\textup{end}^{y_i}$ of $\finitegraph$  is
  constructed between line 3 (when $y_i$ is taken out of  $\mathsf{SQ}$)
  and line 26 of the algorithm (after which the next bad vertex $y_{i+1}$
    is taken out of $\mathsf{SQ}$ or the algorithm terminates because
    $i=k_{\textup{end}}$). The injection $\tau^i$ and the random field
  $\phi^i$  will be defined according to the  behaviour of the algorithm
  during this time.

  On line 4 of the algorithm we generate $\psi(y_i)$. If $\psi(y_i)<h$,
  then the algorithm continues back on line 2 and
  $\mathsf{T}_\textup{end}^{y_i}=\emptyset$. In this case define
  recursively $\phi^i(\rot) \coloneqq \psi(y_i)$ and
  $\phi^i(z) \coloneqq \frac{1}{d-1} \phi^i(\overline z) + Y_x^i$ for
  $z\in \treegraph \setminus \{\rot\}$. Then \eqref{801} holds for
  $\phi^i$ by \eqref{2544}--\eqref{e:defPa}. Moreover \eqref{800} is
  trivially satisfied since $\mathsf{T}_\textup{end}^{y_i}=\emptyset$
  (set $B^i \coloneqq \emptyset$). Otherwise we have $\psi(y_i)\geq h$
  and $y_i$ is added to $\mathsf{T}^{y_i}$. If the algorithm terminates
  on line 8, then $\mathsf{T}_\textup{end}^{y_i} = \{y_i\}$. In this case
  set $B^i \coloneqq \{y_i\}$ and
  $\tau^i(y_i) \coloneqq \rot \in \treegraph$ and again recursively
  define  $\phi^i(\rot) \coloneqq \psi(y_i)$ and
  $\phi^i(z) \coloneqq \frac{1}{d-1} \phi^i(\overline z) + Y_x^i$ for
  $z\in \treegraph \setminus \{\rot\}$.  Then \eqref{801} holds for
  $\phi^i$ by \eqref{2544}--\eqref{e:defPa} and also \eqref{800} is
  satisfied since $\big| \psi(y_i) - \phi^i(\tau^i(y_i)) \big|=0$. If the
  algorithm does not terminate on line 8, then on line 10 we now add all
  neither explored nor already queuing neighbours of $y_i$ to
  $\mathsf{PQ}$ (which before that was empty). Consider the
  \textbf{while}-loop on line 11. During this \textbf{while}-loop, if $z$
  is taken out of $\mathsf{PQ}$ and $z \notin G_\mathsf{E}$, then it is
  transferred to $\mathsf{SQ}$ and it will not be part of
  $\mathsf{T}_\textup{end}^{y_i}$. Let $z_1,\ldots,z_m$ be the successive
  vertices taken out of  $\mathsf{PQ}$  during the \textbf{while}-loop
  which are in $G_\mathsf{E}$ at the moment they are checked (on line
    13). Possibly there are no such vertices, so we might have
  $\{z_1,\ldots,z_m\} = \emptyset$ and
  $\mathsf{T}_\textup{end}^{y_i}=\{y_i\}$. In any case,
  $\mathsf{T}_\textup{end}^{y_i} = \{y_i\} \cup \{z_1,\ldots,z_m \, | \,
    \psi(z_i)\geq h\} \subseteq \{y_i,z_1,\ldots,z_m\}
  \subseteq B_\finitegraph( \mathsf{T}_\textup{end}^{y_i},1)$.
  The injection $\tau^i$ we are going to construct now, will map
  $B^i\coloneqq \{y_i,z_1,\ldots,z_m\}$ to $\treegraph$. By definition
  one has $\overline{z}_1=y_i$ whereas for $j=2,\ldots,m$ one has
  $\overline{z}_j =z$ for some $z\in \{y_i,z_1,\ldots,z_{j-1}\}$. More
  precisely, $\overline{z}_j$ is the unique neighbour of $z_j$ in
  $\mathsf{E}$ at the moment $z_j$ was added to $\mathsf{PQ}$ (which
    happened on line 10 or line 22). There cannot be more than one since
  at a later point $z_j \in G_\mathsf{E}$ and the set of explored
  vertices only grows. Since there are at most $d-1$ not explored
  neighbours that can be added on line 10 or 22 (except if $i=1$ when
    $y_1=x$ and on line 10 there are added exactly $d$ neighbours), this
  shows that for $z \in B^i$ there are at most $d-1$ elements
  $w\in \{z_1,\ldots,z_m\}$ such that $\overline{w}=z$ (exactly $d$
    elements if $i=1$ and $z=y_1$). Therefore, we can define
  $\tau^i: B^i \to \treegraph$ inductively by
  $\tau^i(y_i)\coloneqq \rot$ and such that $\tau^i$ restricted to
  $\{ w\in \{z_1,\ldots,z_m\} \, | \, \overline{w}=z\}$ is an injective
  map to $S_\treegraph(\rot,1)$ for $z=y_i$ and an injective map to
  $S_\treegraph(\tau^i(z),1)\setminus \{\tau^i(\overline{z})\}$ for
  $z\in \{z_1,\ldots,z_m\}$. Note that $\tau^i(B^i)$ is a connected
  subset of $\treegraph$ containing the root $\rot\in \treegraph$.
  By construction also $\tau^i(\mathsf{T}_\textup{end}^{y_i})$ is a
  connected subset of $\treegraph$ containing $\rot\in \treegraph$
  because $y_i \in \mathsf{T}_\textup{end}^{y_i}$ with
  $\tau^i(y_i)= \rot$ and
  $B^i \subseteq B_\finitegraph( \mathsf{T}_\textup{end}^{y_i},1)$. We
  now define $\phi^i$ on $\tau^i(B^i) \subseteq \treegraph$ and check the
  remaining properties in \eqref{800} and \eqref{801} for this case.

  Set $\phi^i(\rot) =\phi^i(\tau^i(y_i))  \coloneqq \psi(y_i)$ and
  for $j=1,\ldots,m$ define inductively
  $\phi^i(\tau^i(z_j)) \coloneqq \frac{1}{d-1} \phi^i(\tau^i(\overline{z}_j))
  + \xi_{z_j} \cdot (\frac{d}{d-1})^\frac12$.
  Recall that here $\xi_{z_j}$ for $j=1,\ldots,m$ are the i.i.d.~standard
  Gaussian random variables used to define $\psi(z_1),\ldots,\psi(z_m)$
  at the respective moments on line 16 of the algorithm. Note that
  conditionally on $\psi(y_i)$, the field $(\phi^i(\tau^i(z)))_{z\in B^i}$
  has the same distribution as
  $(\varphi_\treegraph(\tau^i(z)))_{z\in B^i}$ under
  $\mathbb P^\treegraph_{\psi(y_i)}$. This follows by
  \eqref{2544}--\eqref{e:defPa} since $\tau^i$ is defined in such a way
  that $\tau^i(\overline z)$ (in the notation of \eqref{485}) for
  $z\in \{z_1,\ldots,z_m\}$ is equal to $\overline{\tau^i(z)}$ (in the
    notation above \eqref{830}). We extend $\phi^i$ to all
  $w \in U \coloneqq \treegraph \setminus B^i$ by recursively defining
  $\phi^i(w)  \coloneqq \frac{1}{d-1} \phi^i(\overline w) + Y_w^i$. Then
  \eqref{801} holds for $\phi^i$ by \eqref{2544}--\eqref{e:defPa}. We
  proceed to show the remaining claim of \eqref{800}. Note that
  $\big| \psi(y_i) - \phi^i(\tau^i(y_i)) \big|
  = \big| \psi(y_i) - \phi^i(\rot) \big|
  = \big| \psi(y_i) - \psi(y_i) \big|=0$
  by definition. For  $j=1,\ldots,m$ one has
  \begin{equation*}
    \begin{split}
      &\big| \psi(z_{j}) - \phi^i(\tau^i(z_{j})) \big|  = \big|
      a(z_{j},\psi,\mathsf{E}) + \xi_{z_{j}} \cdot b(z_{j},\mathsf{E})^{\frac12}  -
      \tfrac{1}{d-1} \phi^i(\tau^i(\overline{z}_{j})) - \xi_{z_{j}}\cdot
      (\tfrac{d}{d-1})^\frac12 \big| \\
      &\quad \leq  \big| a(z_{j},\psi,\mathsf{E}) -
      \tfrac{1}{d-1}\psi(\overline{z}_{j}) \big| +  |\xi_{z_{j}}| \cdot \big|
      b(z_{j},\mathsf{E})^\frac12 - (\tfrac{d}{d-1})^\frac12 \big|  + \tfrac{1}{d-1}
      \big| \psi(\overline{z}_{j}) - \phi^i(\tau^i(\overline{z}_{j})) \big| .
    \end{split}
  \end{equation*}
  To the first two differences on the right hand side we can apply
  \eqref{518} and \eqref{547} (on the event
    $\{\textstyle \sup_{z\in \mathsf{E}_\textup{end}} |\psi(z)| < M_n\}$)
  since at any moment of the algorithm
  $|\mathsf{E}|\leq |\mathsf{E}_\textup{end}| \leq c_{h,\kappa} \ln(N_n)$
  by \eqref{851}. We also use the  inequality
  $|\sqrt{s} - \sqrt{t}| = \frac{|s-t|}{\sqrt{s}+\sqrt{t}} \leq \frac{1}{\sqrt{t}} |s-t|$.
  So  by Proposition~\ref{535} (for $b \coloneqq c_{h,\kappa}$ and
    $b' \coloneqq c_\kappa$) we find $c_{h,\kappa}'>0$ such that for all
  $j=1,\ldots,m$
  \begin{equation}
    \label{805}
    \begin{split}
      &\big| \psi(z_{j}) - \phi^i(\tau^i(z_{j})) \big|  \leq (1+
        |\xi_{z_{j}}|)c_{h,\kappa}'(\ln(N_n))^{-2}  + \tfrac{1}{d-1}  \big|
      \psi(\overline{z}_{j}) - \phi^i(\tau^i(\overline{z}_{j})) \big|.
    \end{split}
  \end{equation}
  Now note that on the event
  $\{\sup_{z\in \mathsf{E}_\textup{end}} |\psi(z)| < M_n\}$ one has,
  again by using Proposition~\ref{535} for the same
  $b \coloneqq c_{h,\kappa}$ and $b' \coloneqq c_\kappa$,   that for
  $j=1,\ldots,m$
  \begin{equation*}
    \begin{split}
      M_n
      &\overset{\phantom{\eqref{547}}}{>} |\psi(z_{j}) |
      = \big| a(z_{j},\psi,\mathsf{E}) + \xi_{z_{j}} \cdot
      b(z_{j},\mathsf{E})^{\frac12} \big|
      \geq |\xi_{z_{j}}| \cdot |b(z_{j},\mathsf{E})^{\frac12}| - |
      a(z_{j},\psi,\mathsf{E})|  \\
      &\stackrel[\eqref{547}]{\eqref{518}}{\geq} |\xi_{z_{j}}| \cdot \big(
        \tfrac{d}{d-1}-c_{h,\kappa}'(\ln(N_n))^{-3} \big)^{\frac12} - \big(
        \tfrac1{d-1}\psi(\overline z_j)+ c_{h,\kappa}'(\ln(N_n))^{-2} \big)   \geq
      |\xi_{z_{j}}| - M_n,
    \end{split}
  \end{equation*}
  where the last inequality holds if $n$ is large enough. Combine this
  with \eqref{805} to obtain that for $n$ large enough, $j=1,\ldots,m$
  and on the event $\{\sup_{z\in \mathsf{E}_\textup{end}} |\psi(z)| < M_n\}$
  \begin{equation}
    \label{560}
    \big| \psi(z_{j}) - \phi^i(\tau^i(z_{j})) \big|  \leq (1+ 2M_n
      )c_{h,\kappa}'(\ln(N_n))^{-2}  + \tfrac{1}{d-1}  \big| \psi(\overline{z}_{j}) -
    \phi^i(\tau^i(\overline{z}_{j})) \big|.
  \end{equation}
  This is the main ingredient to show the remainder of \eqref{800}. By
  induction we will now  show that for $n$ large enough and on the  event
  $\{\sup_{z\in \mathsf{E}_\textup{end}} |\psi(z)| < M_n\}$ one has
  \begin{equation}
    \label{0002}
    \big|
    \psi(z_{j}) - \phi^i(\tau^i(z_{j})) \big| \leq j (1+2M_n) c_{h,\kappa}'
    (\ln(N_n))^{-2} \quad \text{for } j=1,\ldots,m.
  \end{equation}
  For $j=1$ one has $\overline{z}_1=y_i$ and the last summand on the
  right hand side of \eqref{560} vanishes by definition of
  $\phi^i(\tau^i(y_i))$. Assume the statement holds for $1\leq j<m$. Then
  either $\overline{z}_{j+1}=y_i$ and therefore the last summand on the
  right hand side of \eqref{560} vanishes again, or
  $\overline{z}_{j+1}=z_k$ for some $k\in \{1,\ldots,j\}$ and the
  induction hypothesis implies
  $ \big| \psi(\overline{z}_{j+1}) - \phi^i(\tau^i(\overline{z}_{j+1})) \big|
  \leq k(1+2M_n)c_{h,\kappa}'(\ln(N_n))^{-2}
  \leq j(1+2M_n)c_{h,\kappa}'(\ln(N_n))^{-2}$.
  In any case by \eqref{560},
  $\big| \psi(z_{j+1}) - \phi^i(\tau^i(z_{j+1})) \big|
  \leq (1+2M_n)c_{h,\kappa}'(\ln(N_n))^{-2} + \frac{1}{d-1} j(1+2M_n)c_{h,\kappa}'(\ln(N_n))^{-2}
  \leq (j+1)(1+2M_n)c_{h,\kappa}'(\ln(N_n))^{-2}$.
  This concludes the induction.

  Let $\varepsilon>0$. Since
  $\{z_1,\ldots,z_m\} \subseteq \mathsf{E}_\textup{end}$, one has
  $m \leq |\mathsf{E}_\textup{end}| \leq c_{h,\kappa} \ln(N_n)$ by
  \eqref{851}. Moreover, $M_n=c_\kappa\sqrt{\ln(N_n)}$. So by
  \eqref{0002} we obtain that for $n$ large enough and on the event
  $\{\sup_{z\in \mathsf{E}_\textup{end}} |\psi(z)| < M_n\}$ one has
  $\big| \psi(z_j) - \phi^i(\tau^i(z_{j})) \big|
  \leq c_{h,\kappa}\ln(N_n)(1+2c_\kappa\sqrt{\ln(N_n)})
  c_{h,\kappa}' (\ln(N_n))^{-2} \leq \varepsilon$
  for $j=1,\ldots,m$. We deduce \eqref{800}.

  It remains to show \eqref{801} for
  $i=k_\textup{end}+1,\ldots,k_\textup{max}$ and \eqref{802}. For
  $i=k_\textup{end}+1,\ldots,k_\textup{max}$ define recursively
  $\phi^i(\rot) \coloneqq M_n$ and
  $\phi^i(z) \coloneqq \frac{1}{d-1} \phi^i(\overline z) + Y_x^i$ for
  $z\in \treegraph \setminus \{\rot\}$, so that \eqref{801} holds for
  $\phi^i$ by \eqref{2544}--\eqref{e:defPa}. Finally, note that for each
  $i=1,\ldots, k_\textup{max}$ the field $\phi^i$ is constructed using
  $(Y_x^i)_{x\in \treegraph \setminus \{\rot\}}$  and possibly the
  i.i.d.~random variables $(\xi_z)_{z\in B^i}$ and  $\Psi(y_i)$. Since
  for $j=1,\ldots, k_\textup{max}$ with $j\neq i$ one has that
  $(Y_x^i)_{x\in \treegraph \setminus \{\rot\}}$ is independent of
  $(Y_x^j)_{x\in \treegraph \setminus \{\rot\}}$  and
  $B^i \cap B^j = \emptyset$, this shows \eqref{802} conditionally on
  $\Psi(y_1),\ldots,\Psi(y_{k_\textup{end}})$ (all random variables are
    Gaussian). The proof is complete.
\end{proof}

\begin{corollary}
  \label{841}
  Let $h\in \mathbb R$ and $\varepsilon>0$. Consider  Algorithm
  \ref{algorithm} and recall $k_\textup{max}$ from Lemma~\ref{532}. Then
  on the same  auxiliary space $(\Omega,\mathcal A,\mathbb P)$ as $\psi$
  one can define random variables $Z^1 ,\ldots,Z^{k_\textup{max}}$ such
  that,  conditionally on $\psi(y_1),\ldots,\psi(y_{k_\textup{end}})$,
  the following properties hold (see \eqref{e:defPa} and below
    \eqref{0.5} for notation):
  \begin{flalign}
    \quad \ \, \bullet& \ \, \text{for all $n$ large enough and  all
      $i=1,\ldots,k_\textup{end}$ one has $Z^i \geq |\mathsf{T}_\textup{end}^{y_i}|$
    }    & \nonumber  \\
    \quad \ \, & \ \, \text{on the event $\{\textstyle \sup_{z\in
          \mathsf{E}_\textup{end}} |\psi(z)| < M_n\}$} &
    \label{563}  \\
    \quad \ \,  \bullet& \ \, \text{$Z^i$ is distributed as $|\mathcal
      C_{\rot}^{\treegraph,h-\varepsilon}|$ under $\mathbb
      P^\treegraph_{\psi(y_i)}$ for all $i=1,\ldots,k_\textup{end}$,} & \nonumber \\
    \quad \ \, & \ \, \text{$Z^i$ is distributed as $|\mathcal
      C_{\rot}^{\treegraph,h-\varepsilon}|$ under $\mathbb P^\treegraph_{M_n}$
      for all $i=k_\textup{end}+1,\ldots,k_\textup{max}$} &
    \label{565}  \\
    \quad \ \,  \bullet& \ \, \text{$Z^1,\ldots,Z^{k_\textup{max}}$ are
      independent.} &
    \label{564}
  \end{flalign}
\end{corollary}
\begin{proof}
  We consider Lemma \ref{806} and define $Z^i$ as the size of the
  connected component of
  $\{ w\in \treegraph \, | \, \phi^i(w) \geq h-\varepsilon \}$ containing
  the root $\rot\in \treegraph$. Then \eqref{565} and \eqref{564}
  follow from \eqref{801} and \eqref{802}. We turn to \eqref{563}. Note
  that for $i=1,\ldots,k_\textup{end}$ and
  $z\in \mathsf{T}_\textup{end}^{y_i}$ one has
  $\big| \psi(z) - \phi^i(\tau^i(z)) \big| \leq \varepsilon$  by
  \eqref{800} under the assumptions of \eqref{563}. This shows that
  $\phi^i(\tau^i(z)) \geq h-\varepsilon$ since $\psi(z) \geq h$ due to
  $z\in \mathsf{T}_\textup{end}^{y_i}$. Hence
  $\tau^i(\mathsf{T}_\textup{end}^{y_i})
  \subseteq \{ w\in \treegraph \, | \, \phi^i(w) \geq h-\varepsilon \}$.
  As $\tau^i(\mathsf{T}_\textup{end}^{y_i})$ is also a connected subset
  of $\treegraph$ containing $\rot\in \treegraph$, we conclude
  $| \tau^i(\mathsf{T}_\textup{end}^{y_i}) | \leq Z^i$ by definition of
  $Z^i$. The proof of \eqref{563} follows since $\tau^i$ is an injection
  and hence
  $| \tau^i(\mathsf{T}_\textup{end}^{y_i}) | = | \mathsf{T}_\textup{end}^{y_i} |$.
\end{proof}

We are now ready to show Proposition \ref{562}, which as explained above
its statement implies Theorem \ref{microscopiccomponents} and thereby
concludes Section \ref{subcriticalsection}.

\begin{proof}[Proof of Proposition \ref{562}]
  Let $h>h_\star$ and $\kappa>0$. Choose $\varepsilon>0$ small enough
  such that $h-\varepsilon>h_\star$. Moreover, let
  $\delta_{h-\varepsilon}>0$  be such that $g_{h-\varepsilon}$ defined in
  \eqref{2513} has the properties explained therein. Let
  $K=K_{h,\kappa}>0$ to be fixed later (below \eqref{809}). By
  conditioning on $\sigma(\psi(y_1),\ldots,\psi(y_{k_\textup{end}}))$ and
  then applying \eqref{563}, one has for $n$ large enough
  \begin{align}
    & \mathbb P\Big[ \sum_{i=1}^{k_\textup{end}} | \mathsf{T}_\textup{end}^{y_i} |
      \geq K\ln(N_n) \, ,  \sup_{z\in \mathsf{E}_\textup{end}} |\psi(z)| < M_n \Big]
    \nonumber                                                                       \\
    & \leq \mathbb E\bigg[ \mathbb P\Big[ \sum_{i=1}^{k_\textup{max}} Z^i \geq
        K\ln(N_n) \, , \sup_{z\in \mathsf{E}_\textup{end}} |\psi(z)| < M_n \, \Big| \,
        \sigma(\psi(y_1),\ldots,\psi(y_{k_\textup{end}})) \Big] \bigg]
    \label{571}
    \\
    & \leq \mathbb E\bigg[ \bbone_{\{|\psi(y_i)| < M_n \text{ for all
            $i=1,\ldots,k_\textup{end}$} \} } \mathbb P\Big[ \sum_{i=1}^{k_\textup{max}}
        Z^i \geq K\ln(N_n) \, \Big| \,
        \sigma(\psi(y_1),\ldots,\psi(y_{k_\textup{end}})) \Big] \bigg]. \nonumber
  \end{align}
  Since
  $\{\sum_{i=1}^{k_\textup{max}} Z^i \geq K\ln(N_n) \} =
  \{\prod_{i=1}^{k_\textup{max}} (1+\delta_{h-\varepsilon})^{Z^i}
    \geq (1+\delta_{h-\varepsilon})^{K\ln(N_n)} \}$,
  the conditional Markov inequality leads to $\mathbb P$-almost surely
  \begin{equation}
    \label{570}
    \begin{split}
      &\mathbb P\Big[ \sum_{i=1}^{k_\textup{max}} Z^i \geq K\ln(N_n) \, \Big| \,
        \sigma(\psi(y_1),\ldots,\psi(y_{k_\textup{end}})) \Big]   \\
      &\overset{\phantom{\eqref{564}}}{\leq} (1+\delta_{h-\varepsilon})^{-K\ln(N_n)} \mathbb E \Big[
        \prod_{i=1}^{k_\textup{max}} (1+\delta_{h-\varepsilon})^{Z^i}  \, \Big| \,
        \sigma(\psi(y_1),\ldots,\psi(y_{k_\textup{end}})) \Big] \\
      &\stackrel[\eqref{565}]{\eqref{564}}{=} (1+\delta_{h-\varepsilon})^{-K\ln(N_n)}
      \prod_{i=1}^{k_\textup{end}}  \mathbb E^\treegraph_{\psi(y_i)} \Big[
        (1+\delta_{h-\varepsilon})^{|\mathcal C_{\rot}^{\treegraph,h-\varepsilon}|} \Big]  \
      \cdot \!\!\!\!\! \prod_{i=k_\textup{end}+1}^{k_\textup{max}}
      \mathbb E^\treegraph_{M_n} \Big[  (1+\delta_{h-\varepsilon})^{| \mathcal
          C_{\rot}^{\treegraph,h-\varepsilon}|} \Big].
    \end{split}
  \end{equation}
  For $i=1,\ldots, k_\textup{end}$ one has
  $\mathbb E^\treegraph_{\psi(y_i)} \big[  (1+\delta_{h-\varepsilon})^{|\mathcal
      C_{\rot}^{\treegraph,h-\varepsilon}|} \big] \leq \mathbb
  E^\treegraph_{M_n} \big[  (1+\delta_{h-\varepsilon})^{|\mathcal
      C_{\rot}^{\treegraph,h-\varepsilon}|} \big]$  on the event $\{ \sup_{z\in
      \mathsf{E}_\textup{end}} |\psi(z)| < M_n \}$
  by \eqref{2544}--\eqref{e:defPa}.
  This combined with \eqref{571} and \eqref{570} shows  that for $n$ large enough
  \begin{equation}
    \label{808}
    \begin{split}
      &\mathbb P\Big[ \sum_{i=1}^{k_\textup{end}} | \mathsf{T}_\textup{end}^{y_i} |
        \geq K\ln(N_n) \, , \sup_{z\in \mathsf{E}_\textup{end}} |\psi(z)| < M_n \Big]
      \\
      &\qquad \qquad \qquad \qquad\qquad \leq   (1+\delta_{h-\varepsilon})^{-K\ln(N_n)} \mathbb
      E^\treegraph_{M_n} \Big[  (1+\delta_{h-\varepsilon})^{| \mathcal
          C_{\rot}^{\treegraph,h-\varepsilon}|} \Big]^{k_\textup{max}}.
    \end{split}
  \end{equation}
  Let us write
  $S_\treegraph(\rot,1) \eqqcolon \{x_1,\ldots,x_{d} \}$ so that
  $\treegraph = \{\rot \} \cup \bigcup_{i=1}^{d} U_{x_i}$ (see
    \eqref{830}). Note that for $n$ large enough one has $M_n \geq h$.
  Therefore \cite{AC1}, equation (1.11), implies that
  \begin{equation}
    \label{807}
    \begin{split}
      \mathbb E^\treegraph_{M_n} \Big[  (1+\delta_{h-\varepsilon})^{|\mathcal
          C_{\rot}^{\treegraph,h-\varepsilon}|} \Big]
      &=  (1+\delta_{h-\varepsilon}) \, \mathbb E^\treegraph_{M_n} \Big[  \prod_{i=1}^d
        (1+\delta_{h-\varepsilon})^{|\mathcal C_{\rot}^{\treegraph,h-\varepsilon} \cap U_{x_i} |
      }  \Big]    \\
      &=  (1+\delta_{h-\varepsilon}) \mathbb E^Y \bigg[ \mathbb E^\treegraph_{\frac{M_n}{d-1}+Y}
        \Big[ (1+\delta_{h-\varepsilon})^{|\mathcal C_{\rot}^{\treegraph,h-\varepsilon} \cap
            \treegraphplus| }  \Big]  \bigg]^d,
    \end{split}
  \end{equation}
  where $Y\sim \mathcal N(0,\tfrac{d}{d-1})$ and the expectation
  $\mathbb E^Y$ is taken with respect to $Y$. The inner expectation on
  the right hand side of \eqref{807} is equal to
  $g_{h-\varepsilon}(\tfrac{M_n}{d-1}+Y)$, see \eqref{2513}. Thus
  \eqref{807} shows that for $n$ large enough
  \begin{equation}
    \label{803}
    \mathbb E^\treegraph_{M_n} \Big[  (1+\delta_{h-\varepsilon})^{|\mathcal
        C_{\rot}^{\treegraph,h-\varepsilon}|} \Big]
    \leq \Big((1+\delta_{h-\varepsilon}) \, \mathbb E^Y \Big[
        g_{h-\varepsilon}(\tfrac{M_n}{d-1}+Y) \Big]^{d-1} \Big)^{\frac{d}{d-1}}.
  \end{equation}
  For $n$ large enough (so $M_n \geq h$) one has that
  $(1+\delta_{h-\varepsilon})\,
  \mathbb E^Y \big[ g_{h-\varepsilon}(\tfrac{M_n}{d-1}+Y) \big]^{d-1}
  = g_{h-\varepsilon}(M_n)$
  by \eqref{2513}. Hence \eqref{803} and \eqref{808} imply that for $n$
  large enough
  \begin{equation}
    \label{574}
    \mathbb P\Big[ \sum_{i=1}^{k_\textup{end}} | \mathsf{T}_\textup{end}^{y_i} |
      \geq K\ln(N_n) \, , \sup_{z\in \mathsf{E}_\textup{end}} |\psi(z)| < M_n \Big]
    \leq   (1+\delta_{h-\varepsilon})^{-K\ln(N_n)} \big(  g_{h-\varepsilon}(M_n)
      \big)^{\frac{d}{d-1} k_\textup{max}}.
  \end{equation}
  By  \eqref{2513} we know that there exist $c_{h}>0$   and $c_{h}'>0$
  such that for  $n$ large enough one has
  $g_{h-\varepsilon}(M_n) \leq c_{h} \exp(c_{h}'M_n^{3/2})
  = c_{h} \exp \big(c_{h}' (c_\kappa\sqrt{\ln(N_n)})^{3/2} \big)
  \leq c_{h} \exp(c_{h,\kappa}(\ln(N_n))^{3/4})$
  for some $c_{h,\kappa}>0$. Now recall that
  $k_\textup{max} = c_1 K s_n^2$. Therefore due to \eqref{471}, we can
  find  $c_{h}, c_{h,\kappa}>0$ for which
  $\big( g_{h-\varepsilon}(M_n) \big)^{\frac{d}{d-1} k_\textup{max}}
  \leq c_{h} \exp(c_{h,\kappa}K(\ln(N_n))^{7/8})$.
  So for some $c_h>0$ and $c_{h,\kappa}>0$ we obtain
  $(1+\delta_{h-\varepsilon})^{-K\ln(N_n)}
  \big(  g_{h-\varepsilon}(M_n) \big)^{\frac{d}{d-1} k_\textup{max}}
  \leq c_h\exp \big(-c_{h,\kappa} K \ln(N_n) \big)$
  for all $n$ large enough. Hence by \eqref{574}, for $n\geq1$,
  \begin{equation}
    \label{809}
    \mathbb P\Big[ \sum_{i=1}^{k_\textup{end}} | \mathsf{T}_\textup{end}^{y_i} |
      \geq K\ln(N_n) \, , \sup_{z\in \mathsf{E}_\textup{end}} |\psi(z)| < M_n \Big]
    \leq c_{h,\kappa} N_n^{-c_{h,\kappa}' K}.
  \end{equation}
  Take $K=K_{h,\kappa}>0$ large enough such that
  $c_{h,\kappa} N_n^{- c_{h,\kappa}' K_{h,\kappa}} \leq N_n^{- 1-\kappa}$.
  Then by \eqref{809} we can find $c_{h,\kappa}>0$ large enough such that
  \eqref{810} holds for all $n\geq 1$. This concludes the proof of
  Proposition \ref{562} and ultimately of Theorem
  \ref{microscopiccomponents}.
\end{proof}



\section{Mesoscopic components in the supercritical phase}
\label{supercriticalsection}

The last section of this article concerns the proof of \eqref{0.7} in the
form of Theorem \ref{manymesoscopiccomponents} below, that is, the
existence of a supercritical phase (complementary to the subcritical
  situation in Section \ref{subcriticalsection}) in which the connected
components of the levels sets of $\Psi_\finitegraph$ of at least
mesoscopic  size  contain a non-negligible fraction of the vertices of
$\finitegraph$. By mesoscopic size we mean that the number of vertices
contained is a fractional power of the total number of vertices of
$\finitegraph$. To be more specific, we recall the critical value
$h_\star$ (see \eqref{0.5}) and the notation
$\mathcal{C}_x^{\finitegraph,h}$ for the connected component of the level
set  of $\Psi_\finitegraph$ above level $h\in \mathbb R$ containing
$x\in \finitegraph$ (see beginning of Section \ref{subcriticalsection}).
Similarly, we denote  by $\mathcal{C}_x^{\treegraph,h}$ for
$x\in \treegraph$ and $h\in \mathbb R$ the connected component of the
level set of $\varphi_\treegraph$ above level $h$ containing $x$. We also
remind of the function  $\eta^+$ given in  \eqref{2512}. The main result
of this section is the following
\begin{theorem}
  \label{manymesoscopiccomponents}
  Let $h<h_\star$. Then there exist $c_h>0$ (see beginning of the proof
    of Lemma~\ref{expectationcomputation}) such that
  \begin{equation}
    \label{853}
    \lim_{n\to \infty} \mathbb P^\finitegraph \Big[ \sum_{x\in \finitegraph}
      \bbone_{ \big\{ |\mathcal{C}_x^{\finitegraph,h} | \geq N_n^{c_h} \big\} }
      \geq  \frac{\eta^+(h)}{2} \, N_n  \Big] =1.
  \end{equation}
\end{theorem}
As explained in the introduction below \eqref{0.7}, it remains open
whether in the supercritical phase $h<h_\star$, as the size of the graphs
tends to infinity, one actually observes the emergence of a (unique)
giant connected component of the level set above level $h$ (see also
  Remark \ref{openquestions}).

We now give the idea of the proof of Theorem
\ref{manymesoscopiccomponents}. Roughly, the strategy is to control the
expectation and variance of the sum in \eqref{853} and then to deduce
Theorem~\ref{manymesoscopiccomponents} via a second moment inequality.
Now recall that by \eqref{1} all vertices of $\finitegraph$ have an
almost tree-like neighbourhood. One can also show that only a negligible
fraction does not have  an exactly tree-like neighbourhood of smaller
size (Remark \ref{842}). So essentially we can consider only vertices
with a tree-like neighbourhood in the sum in \eqref{853}. Moreover,
instead of counting the vertices $x\in \finitegraph$ with
$\big |\mathcal C_x^{\finitegraph,h} \big| \geq N_n^\gamma$ for some
fixed $\gamma>0$ (i.e.~contained in a mesoscopic connected component of
  $E_{\Psi_\finitegraph}^{\geq h}$), it will be easier to only consider
the  vertices $x\in \finitegraph$ for which the connected component
$\mathcal C_x^{\finitegraph,h}$ is already mesoscopic when intersected
with the tree-like neighbourhood of $x$ (see \eqref{232}). We show that
the expected number of such vertices grows linearly in the total number
of vertices of the graph $\finitegraph$ as $n$ tends to infinity (Lemma
  \ref{expectationcomputation}). A variance computation  then implies
that the number of vertices contained in mesoscopic components
concentrates around its expectation  as $n$ goes to infinity (Lemma
  \ref{variancecomputation}). The computations concerning the expectation
and  variance rely on the local approximation of $\Psi_\finitegraph$  by
$\varphi_\treegraph$ around vertices with tree-like neighbourhood that we
developed in Section \ref{subsectionapproximation}, which allows us to
reduce the computations about $\Psi_\finitegraph$  to computations about
$\varphi_\treegraph$ and apply  results from Section \ref{sectiontree} on
$\varphi_\treegraph$. With a second moment inequality Theorem
\ref{manymesoscopiccomponents} promptly follows. The section ends with
open questions in the supercritical regime $h < h_\star$ (Remark
  \ref{openquestions}).

It will be convenient to introduce some additional notation. For
$x\in \finitegraph$, $n\geq 1$ and $R\geq 0$ we set
$S_\finitegraph^+(x,R) \coloneqq  \pi_{n,x} \big(S_\treegraph^+(\rot,R)\big)$
(see below \eqref{830} for the notation). We also define
\begin{equation}
  \label{213}
  \begin{split}
    r_n   \coloneqq  \max\{1, \lfloor \tfrac{c_0}{18} \log_{d-1}(N_n)  \rfloor\}
    \qquad \text{and} \qquad R_n \coloneqq \max\{1, \lfloor \tfrac{c_0}{6}
      \log_{d-1}(N_n)  \rfloor \}.
  \end{split}
\end{equation}
For $n\geq 1$, $h \in \mathbb R$ and $\gamma >0$ we  define the events
\begin{equation}
  \label{232}
  \begin{split}
    A_{x}^{\finitegraph,h,\gamma} &\coloneqq \big\{ \big|\mathcal
      C_x^{\finitegraph,h} \cap S_\finitegraph^+(x,r_n) \big| \geq N_n^\gamma \big\}
    \quad \text{for $x\in \finitegraph$,} \\
    A_{x}^{\treegraph,h,\gamma} &\coloneqq \big\{ \big|\mathcal
      C_{x}^{\treegraph,h} \cap S_\treegraph^+(x,r_n) \big| \geq N_n^\gamma \big\}
    \quad \text{for $x\in \treegraph$}.
  \end{split}
\end{equation}
Note that the dependency on $n$ in the definition of
$A_{x}^{\treegraph,h,\gamma}$ in \eqref{232} does not appear in the
notation. Finally, we define (with $c_0$ as in  \eqref{32})
\begin{equation}
  \label{220}
  \gamma_h \coloneqq \frac{c_0}{20} \log_{d-1}(\lambda_h) \quad \text{for } h\in
  \mathbb R.
\end{equation}
By  \eqref{19} note that $\gamma_h$ is decreasing in $h$  and
$\gamma_h > 0$ for $h< h_\star$.

In the remainder of this section we will apply several times Theorem
\ref{extendedtheorem2.2} for $r= r_n$ and $R = R_n$ given in \eqref{213}.
Note that, for $n$ large enough,
$1 \leq r_n < R_n  \leq \frac{c_0}{6} \log_{d-1}(N_n)$ as required by
Theorem \ref{extendedtheorem2.2} and furthermore
$r_n \leq  \frac{c_0}{18} \log_{d-1}(N_n)$ and
$R_n-2r_n \geq \big(\frac{c_0}{6} \log_{d-1}(N_n) -1 \big)
- 2 \frac{c_0}{18} \log_{d-1}(N_n) =  \frac{c_0 \log_{d-1}(N_n)}{18}-1$.
Therefore Theorem \ref{extendedtheorem2.2} directly implies (with the
  notation from the beginning of Section \ref{subsectionapproximation})
\begin{lemma}[Corollary of Theorem \ref{extendedtheorem2.2}]
  \label{845}
  There exist $c,c'>0$ such that for   all $n\geq 1$ and $x,x'\in \finitegraph$
  with   $\tx(B_\finitegraph(x,2R_n))=0$,
  $\tx(B_\finitegraph(x',2R_n))=0$ and $B_\finitegraph(x,2R_n) \cap
  B_\finitegraph(x',2R_n) = \emptyset$, there is a coupling $\mathbb Q_n$ of
  $\Psi_\finitegraph$ and $\varphi_\treegraph$ satisfying for all $\varepsilon
  >0$
  \begin{equation}
    \label{215}
    \begin{split}
      &\mathbb Q_n \Big[ \sup_{y \in B_\finitegraph(x,r_n)\cup
          B_\finitegraph(x',r_n)} \big| \Psi_\finitegraph(y) -
        \varphi_\treegraph(\rho_{x,x',2R_n}(y))  \big| >  \varepsilon \Big] \leq c
      \exp \big(- c'\varepsilon^2 N_n^{\frac{c_0}{18}} \big).
    \end{split}
  \end{equation}
  In particular, there exist $c,c'>0$ such that for all $n\geq 1$,
  $x\in \finitegraph$ with   $\tx(B_\finitegraph(x,2R_n))=0$, there is a
  coupling $\mathbb Q_n$ of $\Psi_\finitegraph$ and $\varphi_\treegraph$
  such that for all $\varepsilon >0$ the same bound as in \eqref{215}
  applies to
  $\mathbb Q_n \big[ \sup_{y \in B_\finitegraph(x,r_n)} \big| \Psi_\finitegraph(y)
    - \varphi_\treegraph(\rho_{x,2R_n}(y))  \big|  > \varepsilon \big]$.
\end{lemma}

As the following remark explains, the assumptions on the vertices in the
statement of Lemma \ref{845} are typical.
\begin{remark}
  \label{842}
  Recall $R_n$ from \eqref{213}. For $n$ large enough the number of
  vertices $x \in \finitegraph$ that do not satisfy
  $\tx(B_\finitegraph(x,2R_n))=0$ is negligible when compared to the
  total number of vertices of $\finitegraph$. Indeed, for $n$ large
  enough one has $2 R_n \leq  \lfloor \alpha \log_{d-1}(N_n) \rfloor$ by
  \eqref{32} and  \eqref{wlog} and thus
  $\tx(B_\finitegraph(x, 2 R_n)) \leq 1$ for all $x \in \finitegraph$ by
  assumption \eqref{1}. Now by \cite{21}, Lemma~6.1, we have for $n$
  large enough
  \begin{equation}
    \label{225}
    \big| \{ x \in \finitegraph \, | \,  \tx(B_\finitegraph(x, 2R_n)) = 1
      \} \big| \leq (d-1)^{-(\lfloor \alpha \log_{d-1}(N_n) \rfloor -2R_n)}N_n
    \overset{(*)}{\leq} (d-1)N_n^{1-\frac{2\alpha}{3}},
  \end{equation}
  where in $(*)$ we use that
  $\lfloor \alpha \log_{d-1}(N_n) \rfloor -2R_n  \geq \alpha \log_{d-1}(N_n)-1  -
  \frac{c_0}{3} \log_{d-1}(N_n) \geq \frac{2\alpha }{3}\log_{d-1}(N_n)-1$
  (because $c_0 \leq \alpha$ by \eqref{32} and \eqref{wlog}).
  Moreover, for $n$ large enough, also the number of pairs of vertices
  $x,x' \in \finitegraph$ for which
  $B_\finitegraph(x,2R_n) \cap B_\finitegraph(x',2R_n)\neq \emptyset$ is
  negligible when compared to the total number $N_n^2$ of pairs of
  vertices of $\finitegraph$. Indeed, for $n$ large enough and for such
  $x,x'\in \finitegraph$ one has $x' \in B_\finitegraph(x,4R_n)$ and hence
  \begin{equation}
    \label{242}
    \begin{split}
      &\big| \{ x,x' \in \finitegraph \, | \,  B_\finitegraph(x,2R_n) \cap
        B_\finitegraph(x',2R_n)\neq \emptyset \} \big| \leq  \sum_{x\in \finitegraph}
      |B_\finitegraph(x,4R_n)| \\
      &\overset{\eqref{0}}{\leq} N_n |B_\treegraph(\rot,4R_n)| = N_n
      \frac{d(d-1)^{4R_n}-2}{d-2} \overset{\eqref{213}}{\leq} N_n \, d(d-1)^{\frac{2
          c_0}{3} \log_{d-1}(N_n) } \leq d N_n^{\frac{5}{3}} ,
    \end{split}
  \end{equation}
  where the last inequality follows because $c_0 \leq 1$ (see \eqref{32}).
  \qed
\end{remark}

We are now ready to proceed with the expectation and variance computation
announced after the statement of Theorem \ref{manymesoscopiccomponents}.

\begin{lemma}
  \label{expectationcomputation}
  Let $h< h_\star$. There exists $c_h>0$
  such that for all $0<\varepsilon < \frac{h_\star - h}2$ and $\zeta>0$ one has
  for $n$ large enough
  \begin{equation}
    \label{223}
    \mathbb P^\finitegraph \big[ A_{x}^{\finitegraph,h,c_h} \big] \geq
    \eta^+(h+\varepsilon)-\zeta \quad \text{for $x\in \finitegraph$ with
      $\tx(B_\finitegraph(x,2R_n))=0$}.
  \end{equation}
  As a consequence, one has
  \begin{equation}
    \label{224}
    \liminf_{n\to \infty} \frac{1}{N_n} \mathbb E^\finitegraph \bigg[ \sum_{x\in
        \finitegraph} \bbone_{A_{x}^{\finitegraph,h,c_h}}  \bigg] \geq  \eta^+(h)
    >0.
  \end{equation}
\end{lemma}
\begin{proof}
  Let $h<h_\star$ and take $\delta\coloneqq\frac{h_\star - h}2>0$ so that
  $h+\delta<h_\star$. Let $\varepsilon<\delta$. Set also
  $c_h \coloneqq \gamma_{h+\delta}$. For $x\in \finitegraph$ as in the
  statement of \eqref{223} we can apply  Lemma \ref{845} and obtain that
  for $n\geq 1$  one has (recall that $\rho_{x,2R_n}$ is a graph
    isomorphism from $B_\finitegraph(x,2R_n)$ to
    $B_\treegraph(\rot,2R_n)$, see beginning of Section
    \ref{subsectionapproximation})
  \begin{equation}
    \label{219}
    \begin{split}
      \mathbb P^\finitegraph \big[ A_{x}^{\finitegraph,h,c_h} \big]
      &\overset{\phantom{\eqref{232}}}{\geq} \mathbb Q_n  \Big[
        A_{x}^{\finitegraph,h,\gamma_{h+\delta}} , \sup_{y \in B_\finitegraph(x,r_n)}
        \big| \Psi_\finitegraph(y) - \varphi_\treegraph(\rho_{x,2R_n}(y))  \big|  \leq
        \varepsilon \Big]   \\
      &\overset{\eqref{232}}{\geq}  \mathbb Q_n
      \Big[A_{\rot}^{\treegraph,h+\varepsilon,\gamma_{h+\delta}} , \sup_{y \in
          B_\finitegraph(x,r_n)} \big| \Psi_\finitegraph(y) -
        \varphi_\treegraph(\rho_{x,2R_n}(y))  \big|  \leq   \varepsilon \Big] \\
      & \overset{\phantom{\eqref{232}}}{\geq}   \mathbb P^\treegraph \big[
        A_{\rot}^{\treegraph,h+\varepsilon,\gamma_{h+\delta}} \big] - c \exp
      \big(-c' \varepsilon^2 N_n^{\frac{c_0}{18}} \big).
    \end{split}
  \end{equation}
  Note that since $0<\varepsilon<\delta$ one has
  $\gamma_{h+\delta}<\gamma_{h+\varepsilon}$ and hence
  \begin{equation}
    \label{92}
    \begin{split}
      \liminf_{n\to \infty}  \mathbb P^\treegraph & \big[
        A_{\rot}^{\treegraph,h+\varepsilon,\gamma_{h+\delta}} \big]
      \geq \liminf_{n\to \infty}   \mathbb P^\treegraph \big[
        A_{\rot}^{\treegraph,h+\varepsilon,\gamma_{h+\varepsilon}} \big] \\
      &\stackrel[\eqref{220}]{\eqref{232}}{=}   \liminf_{n\to \infty}   \mathbb
      P^\treegraph \Big[  \big|\mathcal C_{\rot}^{\treegraph,h+\varepsilon}
        \cap S_\treegraph^+(\rot,r_n) \big| \geq
        \lambda_{h+\varepsilon}^{\frac{c_0}{20}\log_{d-1}(N_n)} \Big]  \\
      &\overset{\  (*) \   }{\geq}   \liminf_{n\to \infty}  \mathbb P^\treegraph
      \Big[  \big|\mathcal C_{\rot}^{\treegraph,h+\varepsilon} \cap
        S_\treegraph^+(\rot,r_n) \big| \geq
        \frac{\lambda_{h+\varepsilon}^{r_n}}{r_n^2} \Big]
      \overset{\eqref{2141}}{=} \eta^+(h+\varepsilon),
    \end{split}
  \end{equation}
  where in $(*)$ we use that $\lambda_{h+\varepsilon}>1$ (see \eqref{19})
  and the definition of $r_n$ (see \eqref{213}). By combining \eqref{219}
  and \eqref{92} we find \eqref{223}. For \eqref{224} we only need to
  notice that,  for $n$ large enough,
  $\big| \{ x \in \finitegraph \, | \,  \tx(B_\finitegraph(x, 2R_n))
    = 0 \} \big| \geq  N_n-(d-1)N_n^{1-\frac{2\alpha}{3}}$
  by \eqref{225}. So \eqref{224} follows from \eqref{223} by summing only
  over  $x\in \finitegraph$ with $ \tx(B_\finitegraph(x, 2R_n)) = 0$ and
  applying~\eqref{2512}.
\end{proof}

As a next step we want to show that
$\sum_{x\in \finitegraph} \bbone_{ A_{x}^{\finitegraph,h,c_h}}$ for
$h<h_\star$ and the $c_h>0$ from Lemma \ref{expectationcomputation}
concentrates around its expectation. A variance computation will  be
enough. The main ingredient is contained in the next lemma.
\begin{lemma}
  Let $h< h_\star$. There exist $c,c'>0$ such that for all  $n\geq 1$,
  $x,x' \in \finitegraph$ with  $\tx(B_\finitegraph(x,2R_n))=0$,
  $\tx(B_\finitegraph(x',2R_n))=0$ and
  $B_\finitegraph(x,2R_n) \cap B_\finitegraph(x',2R_n) = \emptyset$ one
  has for all $\gamma>0$ and $\varepsilon>0$
  \begin{equation}
    \label{238}
    \mathbb P^\finitegraph \big[A_{x}^{\finitegraph,h,\gamma} \, , \,
      A_{x'}^{\finitegraph,h,\gamma} \big]  \,  \leq \,  \mathbb P^\treegraph \big[
      A_{\rot}^{\treegraph,h-\varepsilon,\gamma} \big]^2   + c \exp
    \big(-c'\varepsilon^2 N_n^{\frac{c_0}{18}} \big).
  \end{equation}
\end{lemma}

\begin{proof}
  Let us abbreviate
  $V\coloneqq B_\finitegraph(x,r_n) \cup B_\finitegraph(x',r_n)$. For
  $x,x'\in \finitegraph$ as in the assumptions we can apply Lemma
  \ref{845} and  obtain that for all $n\geq1$, $\gamma>0$ and
  $\varepsilon>0$ one has (recall the notation $\rho_{x,x',2R_n}$ and
    $z_{x,x'}$ from the beginning of Section \ref{subsectionapproximation})
  \begin{align}
    & \mathbb P^\finitegraph  \big[A_{x}^{\finitegraph,h,\gamma} \, , \,
      A_{x'}^{\finitegraph,h,\gamma} \big]   \nonumber                         \\
    &   \overset{\eqref{215}}{\leq}   \mathbb Q_n \Big[
      A_{x}^{\finitegraph,h,\gamma} , A_{x'}^{\finitegraph,h,\gamma} \, , \, \sup_{y
        \in V} \big| \Psi_\finitegraph(y) - \varphi_\treegraph(\rho_{x,x',2R_n}(y))
      \big|  \leq  \frac{\varepsilon}2 \Big]  + c \exp \big(-c' \varepsilon^2
      N_n^{\frac{c_0}{18}} \big) \nonumber                                       \\
    & \overset{\phantom{\eqref{215}}}{\leq}\mathbb P^\treegraph
    \big[A_{\rot}^{\treegraph,h-\frac{\varepsilon}2,\gamma} \, , \,
      A_{z_{x,x'}}^{\treegraph,h-\frac{\varepsilon}2,\gamma} \big] + c \exp \big(-c'
      \varepsilon^2 N_n^{\frac{c_0}{18}} \big).
    \label{233}
  \end{align}
  To further bound the probability on the right hand side of \eqref{233}
  we apply the decoupling inequality \cite{13}, Corollary 1.3, with
  \begin{equation*}
    \begin{split}
      &\delta \coloneqq \frac{\varepsilon}2, \, K_1 \coloneqq
      B_\treegraph(\rot,r_n), \, K_2 \coloneqq B_\treegraph(z_{x,x'},r_n)
      \text{ and } f_1,f_2:\mathbb R^\treegraph \to [0,1] \text{ such that } \\
      &f_1\big((\varphi_\treegraph(x))_{x\in \treegraph}\big) =
      \bbone_{A_{\rot}^{\treegraph,h-\frac{\varepsilon}2,\gamma}} \text{
        and } f_2\big((\varphi_\treegraph(x))_{x\in \treegraph}\big) =
      \bbone_{A_{z_{x,x'}}^{\treegraph,h-\frac{\varepsilon}2,\gamma}}
    \end{split}
  \end{equation*}
  (the decoupling inequality \cite{13}, Corollary 1.3, is stated for the
    Gaussian free field on $\mathbb Z^d$ but its proof directly applies
    also for the Gaussian free field $\varphi_\treegraph$ on $\treegraph$).
  We obtain that for all $n\geq 1$, $\gamma>0$ and $\varepsilon>0$
  \begin{align}
    & \mathbb P^\treegraph
    \big[A_{\rot}^{\treegraph,h-\frac{\varepsilon}2,\gamma} \, , \,
      A_{z_{x,x'}}^{\treegraph,h-\frac{\varepsilon}2,\gamma} \big]
    \label{234}
    \\
    & \leq \mathbb P^\treegraph
    \big[A_{\rot}^{\treegraph,h-\frac{\varepsilon}2,\gamma} \big]   \mathbb
    P^\treegraph \big[ A_{z_{x,x'}}^{\treegraph,h-\varepsilon,\gamma} \big]  + 2 \,
    \mathbb P^\treegraph \Big[  \sup_{y \in K_2 }  \Big| E_y^\treegraph \big[
        \varphi_\treegraph(X_{H_{K_1}})  \bbone_{\{H_{K_1} < \infty\}}   \big]\Big|
      > \frac{\varepsilon}{4}   \Big].    \nonumber
  \end{align}
  Note that, since we are on a tree and $K_1$ and $K_2$ are two disjoint
  connected sets, there is a unique pair of vertices $z_1\in K_1$,
  $z_2 \in K_2$ with
  $d_\treegraph(K_1,K_2)
  \coloneqq \inf_{z\in K_1,z' \in K_2} d_\treegraph(z,z') = d_\treegraph(z_1,z_2)$.
  Moreover, on the event $\{H_{K_1} < \infty\}$ one $P_y^\treegraph$-almost
  surely has $\varphi_\treegraph(X_{H_{K_1}}) = \varphi_\treegraph(z_1)$
  for $y \in K_2$. Therefore,
  \begin{align}
    & \mathbb P^\treegraph \Big[  \sup_{y \in K_2 } \Big| E_y^\treegraph \big[
        \varphi_\treegraph(X_{H_{K_1}})  \bbone_{\{H_{K_1} < \infty\}}   \big]\Big|
      > \frac{\varepsilon}{4}   \Big] = \mathbb P^\treegraph \Big[  \sup_{y \in K_2
      }  \Big|
      P_y^\treegraph \big[ H_{z_1} < \infty \big]  \varphi_\treegraph(z_1) \Big|  >
      \frac{\varepsilon}{4} \Big]  \nonumber                                     \\
    & \overset{\phantom{\eqref{1.1}}}{\leq}   \mathbb P^\treegraph \Big[  |
      \varphi_\treegraph(z_1) |  > \frac{\varepsilon}{4} P_{z_2}^\treegraph \big[
        H_{z_1} < \infty \big] ^{-1}   \Big]  \overset{(*)}{\leq} 2 \exp \bigg(-
      \frac{(\varepsilon/4)^2}{2P_{z_2}^\treegraph \big[ H_{z_1} < \infty \big]^2
        g_\treegraph(\rot,\rot)}   \bigg) \nonumber                  \\
    & \stackrel[\eqref{1.1}]{(**)}{\leq} 2 \exp \Big(- c \, \varepsilon^2
      (d-1)^{2d_\treegraph(z_1,z_2)}  \Big) ,
    \label{235}
  \end{align}
  where in $(*)$ we use the exponential Markov inequality for the centred
  Gaussian random variable $\varphi_\treegraph(z_1)$ and in $(**)$ we use
  that
  $P_{z_2}^\treegraph \big[ H_{z_1} < \infty \big]
  = (\frac{1}{d-1})^{d_\treegraph(z_1,z_2)}$
  (see e.g.~\cite{33}, proof of Lemma 1.24). Since
  $K_1 \subseteq B_\treegraph(\rot,2R_n)$,
  $K_2 \subseteq B_\treegraph(z_{x,x'},2R_n)$ and
  $B_\treegraph(\rot,2R_n) \cap B_\treegraph(z_{x,x'},2R_n) =\emptyset$
  by assumption, one has the estimate
  $d_\treegraph(z_1,z_2)
  = d_\treegraph(K_1,K_2) > 2(2R_n - r_n) \geq  \frac{5 c_0}{9}\log_{d-1}(N_n)-4$
  for $n$ large enough. Hence
  $\exp \big(-c \, \varepsilon^2 (d-1)^{2d_\treegraph(z_1,z_2)}\big)
  \leq c \exp \big(- c' \varepsilon^2 N_n^{\frac{10 c_0}{9}}  \big)$.
  Therefore we can combine \eqref{233}, \eqref{234} and \eqref{235} to
  obtain that for all $n\geq 1$, $\gamma>0$ and $\varepsilon>0$ (using
    also the symmetry of $\treegraph$)
  \begin{equation*}
    \begin{split}
      &\mathbb P^\finitegraph  \big[A_{x}^{\finitegraph,h,\gamma} \, , \,
        A_{x'}^{\finitegraph,h,\gamma} \big]
      \leq \mathbb P^\treegraph
      \big[A_{\rot}^{\treegraph,h-\frac{\varepsilon}2,\gamma} \big]   \mathbb
      P^\treegraph \big[ A_{\rot}^{\treegraph,h-\varepsilon,\gamma} \big]  +
      c\exp \big(- c' \varepsilon^2  N_n^{\frac{c_0}{18}} \big).
    \end{split}
  \end{equation*}
  This concludes the proof of \eqref{238} since by \eqref{232} it holds
  $A_{\rot}^{\treegraph,h-\frac{\varepsilon}2,\gamma} \subseteq
  A_{\rot}^{\treegraph,h-\varepsilon,\gamma}$.
\end{proof}

We are now ready to for the variance computation. This is the last
ingredient for the proof of Theorem \ref{manymesoscopiccomponents}.
\begin{lemma}
  \label{variancecomputation}
  Let $h< h_\star$. Then for the $c_h>0$ from Lemma \ref{expectationcomputation}
  one has
  \begin{equation}
    \label{91}
    \lim_{n\to \infty} \frac1{N_n^2} \Var_{\mathbb P^\finitegraph} \Big(
      \sum_{x\in \finitegraph} \bbone_{ A_{x}^{\finitegraph,h,c_h}} \Big) = 0.
  \end{equation}
\end{lemma}
\begin{proof}
  By expanding the variance one finds that for all $\gamma>0$
  \begin{align}
    \Var_{\mathbb P^\finitegraph}  \Big( \sum_{x\in \finitegraph}
      \bbone_{ A_{x}^{\finitegraph,h,\gamma}} \Big)
    & =  \sum_{x,x' \in \finitegraph} \Big( \mathbb P^\finitegraph
      \big[A_{x}^{\finitegraph,h,\gamma},  A_{x'}^{\finitegraph,h,\gamma} \big] -
      \mathbb P^\finitegraph \big[A_{x}^{\finitegraph,h,\gamma} \big] \mathbb
      P^\finitegraph \big[A_{x'}^{\finitegraph,h,\gamma} \big]  \Big) \nonumber      \\
    & = \sum_{x,x' \in \finitegraph}  \mathbb P^\finitegraph
    \big[A_{x}^{\finitegraph,h,\gamma},  A_{x'}^{\finitegraph,h,\gamma} \big] -
    \mathbb E^\finitegraph \bigg[ \sum_{x\in \finitegraph}
      \bbone_{A_{x}^{\finitegraph,h,\gamma}}  \bigg]^2  .
    \label{241}
  \end{align}
  We define $W \subseteq \finitegraph \times \finitegraph$ to be the set
  of pairs $(x,x') \in \finitegraph\times \finitegraph$ with
  $\tx(B_\finitegraph(x,2R_n))=0$, $\tx(B_\finitegraph(x',2R_n))=0$ and
  $B_\finitegraph(x,2R_n) \cap B_\finitegraph(x',2R_n) = \emptyset$. For
  $x,x' \in \finitegraph$ such that $(x,x') \notin W$ we can bound the
  probability on the right hand side of \eqref{241} by one. This will be
  good enough since for $n$ large enough
  $|(\finitegraph \times \finitegraph) \setminus W |
  \leq 2 N_n \cdot (d-1)N_n^{1-\frac{2\alpha}{3}} + dN_n^{\frac{5}{3}}
  \leq dN_n\big(2N_n^{1-\frac{2\alpha}{3}} + N_n^{\frac{2}{3}}\big)$
  by \eqref{225} and \eqref{242}.  For $x,x' \in \finitegraph$ such that
  $(x,x') \in W$ we use \eqref{238} instead. There are at most $N_n^2$
  such pairs. Thus we obtain for all $n\geq 1$, $\gamma>0$ and
  $\varepsilon>0$
  \begin{align}
    & \frac1{N_n^2} \Var_{\mathbb P^\finitegraph} \Big(  \sum_{x\in
        \finitegraph} \bbone_{A_{x}^{\finitegraph,h,\gamma}} \Big)
    \label{244} \\
    & \leq\mathbb P^\treegraph \big[
      A_{\rot}^{\treegraph,h-\varepsilon,\gamma} \big]^2  + c \exp
    \big(-c'\varepsilon^2 N_n^{\frac{c_0}{18}} \big)     +
    d\big(2N_n^{-\frac{2\alpha}{3}} + N_n^{-\frac{1}{3}}\big)   - \frac{1}{N_n^2}
    \mathbb E^\finitegraph \bigg[ \sum_{x\in \finitegraph}
      \bbone_{A_{x}^{\finitegraph,h,\gamma}}  \bigg]^2. \nonumber
  \end{align}
  Now we apply \eqref{244} to $\gamma \coloneqq c_{h}>0$ for the $c_h$ from Lemma
  \ref{expectationcomputation} and deduce that for all $0<\varepsilon <
  \frac{h_\star - h}2$
  \begin{equation*}
    \begin{split}
      &\limsup_{n\to \infty} \frac1{N_n^2} \Var_{\mathbb P^\finitegraph}
      \Big(  \sum_{x\in \finitegraph} \bbone_{A_{x}^{\finitegraph,h,c_h}} \Big)
      \\
      &\qquad \stackrel[]{\eqref{244}}{\leq} \limsup_{n\to \infty}  \mathbb
      P^\treegraph \big[ A_{\rot}^{\treegraph,h-\varepsilon,c_h} \big]^2  -
      \liminf_{n\to \infty} \frac{1}{N_n^2} \mathbb E^\finitegraph \bigg[
        \sum_{x\in \finitegraph} \bbone_{A_{x}^{\finitegraph,h,c_h}}  \bigg]^2  \\
      &\qquad \stackrel[\eqref{224}]{\eqref{2512}}{\leq}  \eta^+(h-\varepsilon)^2 -
      \eta^+(h)^2.
    \end{split}
  \end{equation*}
  The statement follows by letting   $\varepsilon$ tend to zero and applying
  \eqref{2512}.
\end{proof}

\begin{proof}[Proof of Theorem \ref{manymesoscopiccomponents}]
  We will show that the probability of the complementary event tends to
  zero. For $n\geq 1$ let us define
  $W_n^{\geq h} \coloneqq  \sum_{x\in \finitegraph} \bbone_{A_{x}^{\finitegraph,h,c_h}}$
  with $c_h>0$  as in Lemma~\ref{expectationcomputation}. Then we can
  estimate
  \begin{equation*}
    \begin{split}
      &\mathbb P^\finitegraph \Big[ \sum_{x\in \finitegraph} \bbone_{ \big\{
            |\mathcal{C}_x^{\finitegraph,h} | \geq N_n^{c_h} \big\} }  <
        \frac{\eta^+(h)}{2} \, N_n  \Big]
      \leq \mathbb P^\finitegraph \Big[ W_n^{\geq h}  < \frac{\eta^+(h)}{2} \, N_n
        \Big] \\
      &\qquad \qquad \qquad =  \mathbb P^\finitegraph \Big[  \frac{1}{N_n}\mathbb
        E^\finitegraph[W_n^{\geq h}] - \frac{1}{N_n} W_n^{\geq h}   >
        \frac{1}{N_n}\mathbb E^\finitegraph[W_n^{\geq h}]  - \frac{\eta^+(h)}{2}
        \Big]
    \end{split}
  \end{equation*}
  and therefore
  \begin{equation*}
    \begin{split}
      \limsup_{n\to\infty}  \mathbb P^\finitegraph \Big[ \sum_{x\in \finitegraph}
        &\bbone_{ \big\{ |\mathcal{C}_x^{\finitegraph,h} | \geq N_n^{c_h} \big\} }
        < \frac{\eta^+(h)}{2} \, N_n  \Big]   \\
      &\overset{\eqref{224}}{\leq}  \limsup_{n\to\infty} \mathbb P^\finitegraph \Big[
        \frac{1}{N_n}\mathbb E^\finitegraph[W_n^{\geq h}] -  \frac{1}{N_n} W_n^{\geq
          h}   > \frac{\eta^+(h)}{2}   \Big] \\
      &\overset{\ \, (*) \, \ }{\leq}  \limsup_{n\to\infty} \frac{4}{\eta^+(h)^2}
      \Var_{\mathbb P^\finitegraph} \Big( \frac{1}{N_n} W_n^{\geq h}\Big)
      \overset{\eqref{91}}{=} 0,
    \end{split}
  \end{equation*}
  where in $(*)$ we use Chebyshev's inequality. This concludes the proof of
  Theorem \ref{manymesoscopiccomponents}.
\end{proof}

\begin{remark}
  \label{openquestions}
  It remains open whether in the supercritical phase $h<h_\star$, with
  high probability for large $n$, there actually is a macroscopic (giant)
  connected component of the level set above level $h$ (i.e.~containing a
    number of vertices comparable to $\finitegraph$), and whether this
  giant component is unique (meaning the size of the second-largest
    connected component is negligible compared to $\finitegraph$). For
  other probabilistic models on essentially the same class of graphs this
  has been shown. One  example is the emergence of a unique giant
  connected component for Bernoulli bond percolation on $d$-regular
  expanders of large girth  (see \cite{38} and also \cite{46}). A second
  example is the emergence of a unique giant connected component in the
  vacant set of simple random walk on the same graphs
  $(\finitegraph)_{n\geq 1}$ as considered here (see \cite{21}). As
  briefly mentioned in the introduction below \eqref{0.7}, such results
  are typically obtained by a sprinkling argument out of an intermediary
  result like Theorem \ref{manymesoscopiccomponents}. In our setting, the
  zero-average property of $\Psi_\finitegraph$ (see below \eqref{1.7})
  prevents us from easily implementing such a strategy. In particular,
  due to the zero-average property, the field $\Psi_\finitegraph$ neither
  satisfies an FKG-inequality nor does it possess the domain Markov
  property of the Gaussian free field $\varphi_\treegraph$ (compare
    \eqref{460} with \eqref{1.18}). In contrast, the sprinkling argument
  in \cite{DreRod} for constructing an infinite connected component for
  the Gaussian free field on $\mathbb Z^d$ for high-dimension $d$
  crucially relies on the domain Markov property of the Gaussian free
  field on $\mathbb Z^d$ for $d\geq 3$. \qed
\end{remark}



\bibliographystyle{alpha}
\bibliography{bibliographyfile}

\end{document}